\documentclass[11pt]{amsart}

\usepackage{etex, amsfonts, amsmath, amsthm, amssymb, stmaryrd,
  epsfig, graphics, psfrag, latexsym, color, mathtools,
  mathrsfs,enumitem, subfig, longtable, booktabs, yfonts}
\usepackage[all,cmtip]{xy} \usepackage[percent]{overpic}
\usepackage{tikz} \usetikzlibrary{matrix,arrows}
\definecolor{darkblue}{rgb}{0,0,0.4} 
\usepackage{xr-hyper}
\usepackage[colorlinks=true,citecolor=darkblue, filecolor=darkblue, linkcolor=darkblue,
urlcolor=darkblue]{hyperref} 
\usepackage[all]{hypcap}

\newlength{\tempfigdim}

\graphicspath{{draws/}{}}

\numberwithin{equation}{section}

\newtheorem{lem}{Lemma}[section]               
\newtheorem{lemma}[lem]{Lemma}               

\newtheorem{cor}[lem]{Corollary}
\newtheorem{corollary}[lem]{Corollary}               
\newtheorem{prop}[lem]{Proposition}
\newtheorem{proposition}[lem]{Proposition}

\theoremstyle{definition}

\newtheorem{definition}[lem]{Definition}

\theoremstyle{remark}     
\newtheorem{rem}{Remark}[section]

\newtheorem{example}[rem]{Example}
\numberwithin{figure}{section}

\newcommand{\Appendix}[1]{\hyperref[app:#1]{Appendix~\ref*{app:#1}}}
\newcommand{\Section}[1]{\hyperref[sec:#1]{Section~\ref*{sec:#1}}}
\newcommand{\Subsection}[1]{\hyperref[subsec:#1]{Subsection~\ref*{subsec:#1}}}
\newcommand{\Lemma}[1]{\hyperref[lem:#1]{Lemma~\ref*{lem:#1}}}
\newcommand{\Theorem}[1]{\hyperref[thm:#1]{Theorem~\ref*{thm:#1}}}
\newcommand{\ThmCor}[1]{\hyperref[thmcor:#1]{Corollary~\ref*{thmcor:#1}}}
\newcommand{\Citethm}[1]{\hyperref[citethm:#1]{Theorem~\ref*{citethm:#1}}}
\newcommand{\Definition}[1]{\hyperref[def:#1]{Definition~\ref*{def:#1}}}
\newcommand{\Remark}[1]{\hyperref[rem:#1]{Remark~\ref*{rem:#1}}}
\newcommand{\Figure}[1]{\hyperref[fig:#1]{Figure~\ref*{fig:#1}}}
\newcommand{\Conjecture}[1]{\hyperref[conj:#1]{Conjecture~\ref*{conj:#1}}}
\newcommand{\Corollary}[1]{\hyperref[cor:#1]{Corollary~\ref*{cor:#1}}}
\newcommand{\Proposition}[1]{\hyperref[prop:#1]{Proposition~\ref*{prop:#1}}}
\newcommand{\Question}[1]{\hyperref[ques:#1]{Question~\ref*{ques:#1}}}
\newcommand{\Example}[1]{\hyperref[exam:#1]{Example~\ref*{exam:#1}}}
\newcommand{\Table}[1]{\hyperref[table:#1]{Table~\ref*{table:#1}}}
\newcommand{\Restric}[1]{\hyperref[restric:#1]{Restriction~\ref*{restric:#1}}}

\newcommand{\Equation}[2][{}]{Equation#1~(\ref{eq:#2})}

\newcommand{\R}{\mathbb{R}}
\newcommand{\Z}{\mathbb{Z}}
\newcommand{\F}{\mathbb{F}}
\newcommand{\C}{\mathbb{C}}

\newcommand{\mc}{\mathcal}

\newcommand{\mf}{\mathfrak}
\newcommand{\bm}{\mathbf}

\newcommand{\wt}{\widetilde}
\newcommand{\ol}{\overline}

\newcommand{\sbs}{\subset}
\newcommand{\sbseq}{\subseteq}
\newcommand{\sm}{\setminus}
\renewcommand{\emptyset}{\varnothing}

\newcommand{\al}{\alpha}
\newcommand{\be}{\beta}
\newcommand{\ga}{\gamma}
\newcommand{\de}{\delta}

\newcommand{\from}{\colon}
\newcommand{\into}{\hookrightarrow}

\newcommand{\onto}{\twoheadrightarrow}

\newcommand{\Zoltan}{Zolt\'{a}n}
\newcommand{\Szabo}{Szab\'{o}}
\newcommand{\Ozsvath}{Ozsv\'{a}th}

\newcommand{\set}[2]{\{#1\mid#2\}}

\renewcommand{\th}{^{\text{th}}}

\DeclareMathOperator{\Sym}{Sym}

\DeclareMathOperator{\Ob}{Ob}
\DeclareMathOperator{\Id}{Id}

\DeclareMathOperator{\Hom}{Hom}

\newcommand{\Kh}{\mathit{Kh}}
\newcommand{\BN}{\mathit{BN}}

\newcommand{\Cx}[1][{}]{\mathcal{C}_{#1}}
\newcommand{\diff}{\delta}

\newcommand{\KhCx}{\mathcal{C}_{\mathit{Kh}}}
\newcommand{\diffKh}{\delta_{\mathit{Kh}}}

\newcommand{\BNCx}{\mathcal{C}_{\mathit{BN}}}
\newcommand{\fBNCx}{\mathcal{C}_{\mathit{fBN}}}
\newcommand{\lBNCx}{\{H\}^{-1}\mathcal{C}_{\mathit{BN}}}
\newcommand{\diffBN}{\delta_{\mathit{BN}}}

\newcommand{\SzCx}{\mathcal{C}_{\mathit{Sz}}}
\newcommand{\fSzCx}{\mathcal{C}_{\mathit{fSz}}}

\newcommand{\diffSz}{\delta_{\mathit{Sz}}}

\newcommand{\OurCx}{\mathcal{C}_{\mathit{tot}}}
\newcommand{\fOurCx}{\mathcal{C}_{\mathit{ftot}}}
\newcommand{\HfOurCx}{\mathcal{C}_{\mathit{fHtot}}}

\newcommand{\HlOurCx}{\{H\}^{-1}\mathcal{C}_{\mathit{tot}}}

\newcommand{\diffOur}{\delta_{\mathit{tot}}}

\newcommand{\Crossings}{\mf{C}}
\newcommand{\Basis}[1]{\mf{B}_{#1}}
\newcommand{\Mirror}[1]{m(#1)}
\newcommand{\Reverse}[1]{r(#1)}
\newcommand{\Dual}[1]{{#1}^*}
\newcommand{\cont}[1]{\mf{#1}}
\newcommand{\enmor}[1]{\bm{#1}}
\newcommand{\dd}{\enmor{d}}
\newcommand{\hh}{\enmor{h}}
\newcommand{\ff}{\enmor{f}}
\newcommand{\fg}{\enmor{g}}
\newcommand{\Commute}[2]{[#1\colon #2]}
\newcommand{\card}[1]{\left\vert{#1}\right\vert}

\newcommand{\gr}{\mathrm{gr}}
\newcommand{\intgr}{\gr_{q}}
\newcommand{\homgr}{\gr_{h}}

\newcommand{\AssRes}[2][D]{{#1}_{#2}}
\newcommand{\ResConfig}[3][D]{{#1}_{#2}^{#3}}

\newcommand{\upright}{\mathcal{U}}

\newcommand*{\defeq}{\mathrel{\vcenter{\baselineskip0.5ex \lineskiplimit0pt
                     \hbox{\scriptsize.}\hbox{\scriptsize.}}}%
                     =}

\newcommand{\Filt}{\mathcal{F}}

\newcommand{\hocat}{\mc{K}}
\newcommand{\hocatpair}{\mc{K}_{\mathrm{p}}}
\newcommand{\Cone}{\mathrm{Cone}}

\newcommand{\tuplethree}[3]{#1,\allowbreak #2,\allowbreak #3}

\begin{document}

\title[A perturbation of the geometric spectral sequence]{A perturbation of the geometric spectral sequence in Khovanov homology}

\author{Sucharit Sarkar}
\thanks{SS was partially supported by NSF CAREER Grant DMS-1350037}
\email{\href{mailto:sucharit@math.princeton.edu}{sucharit@math.princeton.edu}}

\author{Cotton Seed}
\thanks{}
\email{\href{mailto:cseed@math.princeton.edu}{cseed@math.princeton.edu}}

\author{\Zoltan{} \Szabo{}}
\thanks{ZSz was partially supported by NSF Grants DMS-1006006 and DMS-1309152}
\email{\href{mailto:szabo@math.princeton.edu}{szabo@math.princeton.edu}}

\subjclass[2010]{\href{http://www.ams.org/mathscinet/search/mscdoc.html?code=57M25}{57M25}}

\address{Department of Mathematics, Princeton University, Princeton, NJ 08544}
\keywords{}

\date{\today}

\begin{abstract}
  We study the relationship between Bar-Natan's perturbation in
  Khovanov homology and \Szabo's geometric spectral sequence, and
  construct a link invariant that generalizes both into a common
  theory. We study a few properties of the new invariant, and
  introduce a family of $s$-invariants from the new theory in the same
  spirit as Rasmussen's $s$-invariant.
\end{abstract}

\maketitle


\section{Introduction}

In \cite{Kho-kh-categorification} Khovanov categorified the Jones
polynomial to construct the first link homology theory, usually known
as the Khovanov homology $\Kh(L)$ of links $L\subset S^3$. Several
variants and improvements were then constructed---the reduced version
\cite{Kho-kh-patterns}, Lee's perturbation \cite{Lee-kh-endomorphism},
Bar-Natan's tangle invariant \cite{Bar-kh-tangle-cob} (yielding a
Bar-Natan perturbation similar to Lee's), Khovanov's tangle invariant
\cite{Kho-kh-tangles}, functoriality for link cobordisms in
$\R^3\times I$ \cite{Jac-kh-cobordisms, Bar-kh-tangle-cob,
  Kho-kh-cob}, Khovanov-Rozansky's $\mathfrak{sl}(n)$-homology and
HOMFLYPT homology \cite{KR-kh-matrixfactorizations,
  KR-kh-matrixfactorizations2}, \Ozsvath-Rasmussen-\Szabo's odd
Khovanov homology \cite{OSzR-kh-oddkhovanov}, Seidel-Smith's
symplectic Khovanov homology \cite{SS-kh-symplectic} equipped with
localization spectral sequences \cite{SS-kh-involution}---to name a
few.  In \cite{Ras-kh-slice} Rasmussen used Lee's perturbation to
construct a numerical invariant $s(K)$ leading to the first
combinatorial proof of a theorem due to Kronheimer-Mrowka
\cite{KM-gauge-embeddedsurfaces1} on the four-ball genus of torus
knots.  Khovanov homology was also used by Ng to construct a bound on
the Thurston-Bennequin number of Legendrian links
\cite{Ng-kh-tb-bound}.  In \cite{OSz-hf-spectral} \Ozsvath-\Szabo{}
constructed a spectral sequence from the reduced version of the
Khovanov homology of the mirror a link to the Heegaard Floer homology
of its branched double cover. A similar spectral sequence was
constructed by Bloom \cite{Blo-gauge-spectralsequence}, but abutting
to the monopole Floer homology of the branched double cover.
Motivated by these spectral sequences, \Szabo{} constructed the
geometric spectral sequence \cite{Szab-kh-geometric}, a combinatorial
construction that shares many formal properties with its holomorphic
geometry and gauge theory counterparts. On a slightly different note,
Kronheimer-Mrowka constructed a spectral sequence from Khovanov
homology to the instanton knot Floer homology \cite{KM-kh-spectral},
which in turn established that Khovanov homology detects the unknot.

In our present paper, we concern ourselves with the interplay between
the Bar-Natan perturbation in Khovanov homology, as introduced in
\cite{Bar-kh-tangle-cob} and studied further in
\cite{Nao-kh-universal,Turner-kh-BNSeq}, and the geometric spectral
sequence constructed by \Szabo{} in \cite{Szab-kh-geometric}.  We
present a brief survey of the existing constructions
in \Section{background}. We produce our new endomorphism and prove
that it is a differential in \Section{construction}. We
devote \Section{invariance} to proving invariance
and \Section{properties} to studying a few properties of the new
invariant.  Finally in \Section{new-s} we introduce a family of
$s$-invariants, each mimicking Rasmussen's $s$-invariant.  We
summarize our construction, state a few of its salient features, and
discuss a bit on the motivation.

\subsection*{Construction} (\Definition{main-complex} and
\Proposition{is-differential}) We construct a
$(\homgr,\intgr)$-bigraded chain complex $\OurCx$ over $\F_2[H,W]$,
with $H$ and $W$ in bigradings $(0,-2)$ and $(-1,-2)$ and the
differential in bigrading $(1,0)$. The chain group is freely generated
(over $\F_2[H,W]$) by the Khovanov generators coming from some link
diagram. Setting $W=0$ recovers the Bar-Natan theory, while setting
$H=0$ recovers \Szabo's geometric spectral sequence. For link diagrams
equipped with a single basepoint, one can construct two reduced
versions, the minus version $\OurCx^-$ as a subcomplex and the plus
version $\OurCx^+$ as the corresponding quotient complex.

\subsection*{Properties} We list a few key properties of these new
invariants. 
\begin{itemize}[leftmargin=*]
\item (\Proposition{main-invariance} and \Corollary{easy-invariance})
  The chain homotopy type of $\OurCx$ over $\F_2[H,W]$ is a link
  invariant, while the chain homotopy type of the reduced versions
  (along with the maps to and from the unreduced theory) are
  invariants of links equipped with a single basepoint.
\item (\Proposition{total-rank}) For an $l$-component link, the
  homology of the localized version
  $\HlOurCx=\OurCx\otimes_{\F_2[H,W]}\F_2[H,H^{-1},W]$ is $2^l$ copies
  of $\F_2[H,H^{-1},W]$, with the copies corresponding to the $2^l$
  possible orientations of $L$.
\item (\Proposition{lower-bound}) Half the absolute value of each of
  the new $s$-invariants from \Definition{new-s-invariants} is a lower
  bound for the four-ball genus of knots.
\end{itemize}

\subsection*{Motivation} One of the main motivations behind this
construction is the Seidel-Smith symplectic Khovanov homology
\cite{SS-kh-symplectic}, which is conjecturally isomorphic to Khovanov
homology (the conjecture has been verified over characteristic
$0$~\cite{AS-kh-Khovanov-Floer}, although it remains open over other
characteristics, in particular, over characteristic $2$). To expound a
bit, after viewing the link $L$ as a plat closure of a $2n$-braid, and
letting $p\from\C\to\C$ denote the degree-$2n$ polynomial whose zeroes
are the $2n$ points of the braid, Seidel and Smith consider the
complex variety
\[
C\defeq\set{(u,v,z)\in\C^3}{u^2+v^2=p(z)}
\]
and define symplectic Khovanov homology to be the Lagrangian Floer
homology in the $n\th$ symmetric product $\Sym^n(C)$ (or rather, in a
certain open subset $\mathcal{Y}$ of the Hilbert scheme
$\mathrm{Hilb}^n(C)$ which is a smooth resolution of $\Sym^n(C)$, see
\cite{Man-kh-nilpotent} for more details), with each Lagrangian being
a product of $n$ disjoint spheres in $C$; one of the Lagrangians is
standard, coming from the $n$ plats at either end, while the other
comes from applying the braid group action.

There is a $\Z/2$-action induced by
$(u,v)\stackrel{\tau}{\to}(u,-v)$, and its fixed points can be
identified with the complex curve
\[
C^{\tau}\defeq\set{(u,z)\in\C^2}{u^2=p(z)}
\]
which, along with the $\al$ and $\be$ curves (which are the fixed
points of the spheres on $C$), is easily seen to be a Heegaard diagram
(in the sense of \cite{OSz-hf-3manifolds}) for the double branched
cover of $L$; and the fixed points of the Lagrangians can be
identified with the Lagrangian tori from Heegaard Floer homology. The
fixed points $\mc{Y}^{\tau}$ may be viewed as the complement of a
certain divisor in $\Sym^n(C^{\tau})$; therefore, there is a spectral
sequence relating the Floer homology in $\Sym^n(C^{\tau})$ and
$\mc{Y}^{\tau}$, which is conjecturally trivial.  The Seidel-Smith
localization theorem \cite{SS-kh-involution} produces a chain complex
over $H^*(B\Z/2;\F_2)=\F_2[w_1]$ inducing a spectral sequence relating
the symplectic Khovanov homology (which is conjecturally isomorphic to
Khovanov homology) and Lagrangian Floer homology in $\mc{Y}^{\tau}$
(which is conjecturally isomorphic to Heegaard Floer homology, namely
Lagrangian Floer homology in $\Sym^n(C^{\tau})$); and the spectral
sequence conjecturally equals the \Ozsvath-\Szabo{} spectral sequence
and the \Szabo{} geometric spectral sequence.

On the other hand, we may also consider the $S^1$-action
$(u,v)\stackrel{\eta_{e^{i\theta}}}{\longrightarrow}(u\cos\theta-v\sin\theta,u\sin\theta+v\cos\theta)$,
  whose fixed points are the $2n$ zeroes of $p$
\[
C^{\eta}\defeq\set{z\in\C}{p(z)=0}.
\]
The fixed points $\mc{Y}^{\eta}$ is the subspace of $\Sym^n(C^{\eta})$
consisting of distinct points.  Furthermore, each Lagrangian is a
product of $n$ zero-spheres in $C^{\eta}$; one corresponds to the
matching of the $2n$ points induced by the plats, while the other is
the matching gotten by applying the braid group action.  It is easily
seen that the Lagrangians intersect in $2^l$ points, where $l$ is the
number of link components of $L$, and therefore the Lagrangian Floer
homology in the discrete space $\mc{Y}^{\eta}$ is $2^l$-dimensional.
If there were a Seidel-Smith localization theorem for $S^1$-actions,
it would have produced a chain complex over $H^*(BS^1;\Z)=\Z[c_1]$
inducing a spectral sequence from the symplectic Khovanov homology
(over $\Z$) to $\Z^{2^l}$, and it is conceivable that this spectral
sequence would equal the Bar-Natan spectral sequence.

The two actions described above actually combine to induce an $O(2)$
action. Therefore, one might expect that the above two theories can be
combined into a single theory: over $\F_2$, we should expect a chain
complex over $\F_2[w_1,w_2]\cong H^*(BO(2);\F_2)$, so that the
specialization $w_2=0$ produces the \Szabo{} chain complex while the
specialization $w_1=0$ produces the Bar-Natan chain complex. In this
paper, we construct such a theory where $W$ plays the role of $w_1$
and $H$ plays the role of $w_2$. The gradings also work out:
symplectic Khovanov homology carries a single grading, which is
conjectured to equal $\homgr-\intgr$ on the Khovanov side; the
universal Stiefel-Whitney classes $w_1$ and $w_2$ live in gradings $1$
and $2$, respectively, while the formal variables $W$ and $H$ live in
$(\homgr,\intgr)$-bigradings $(-1,-2)$ and $(0,-2)$, respectively.

\subsection*{Acknowledgment}
We thank Kristen Hendricks, Robert Lipshitz, and Ciprian Manolescu for
several illuminating conversations about the Seidel-Smith
construction, in particular for pointing out the $O(2)$-action,
Alexander Shumakovitch for some invaluable help during computations,
and the referees for pointing out some mistakes and for several other
helpful comments.

\section{Background}\label{sec:background}
Before we define our complex, let us first review the Khovanov chain
complex, the Bar-Natan theory, and the \Szabo{} geometric spectral
sequence. All the three existing variants and our new fourth variant
are defined in similar settings, particularly if we are working over
$\F_2$, which we are. The relevant aspects are presented below in an
enumerated list and a few definitions.
\captionsetup[subfloat]{width={0.22\textwidth}}
\begin{figure}
\centering
\subfloat[A crossing $c$.]{\label{fig:res-crossing}\includegraphics[width=0.22\textwidth]{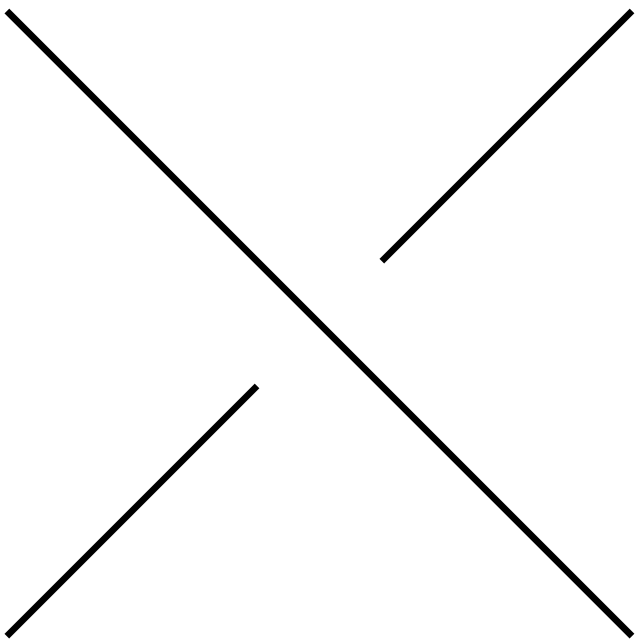}}
\hspace{0.025\textwidth}
\subfloat[The $0$-resolution at $c$.]{\label{fig:res-0}\includegraphics[width=0.22\textwidth]{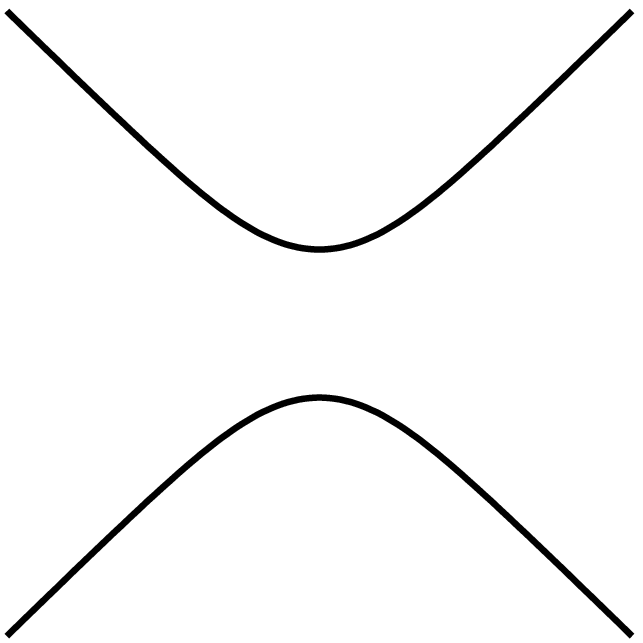}}
\hspace{0.025\textwidth}
\subfloat[The $1$-resolution at $c$.]{\label{fig:res-1}\includegraphics[width=0.22\textwidth]{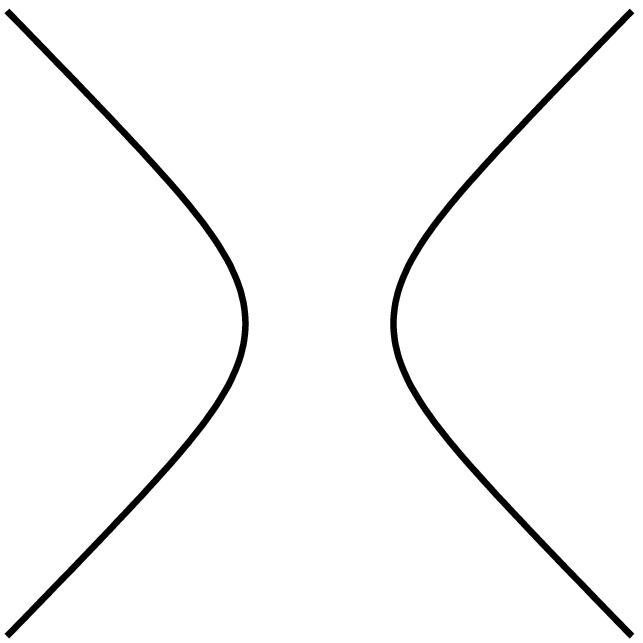}}
\hspace{0.025\textwidth}
\subfloat[The $0$-resolution at $c$ along with the surgery arc $\al_c$.]{\label{fig:res-0-al}\begin{overpic}[width=0.22\textwidth]{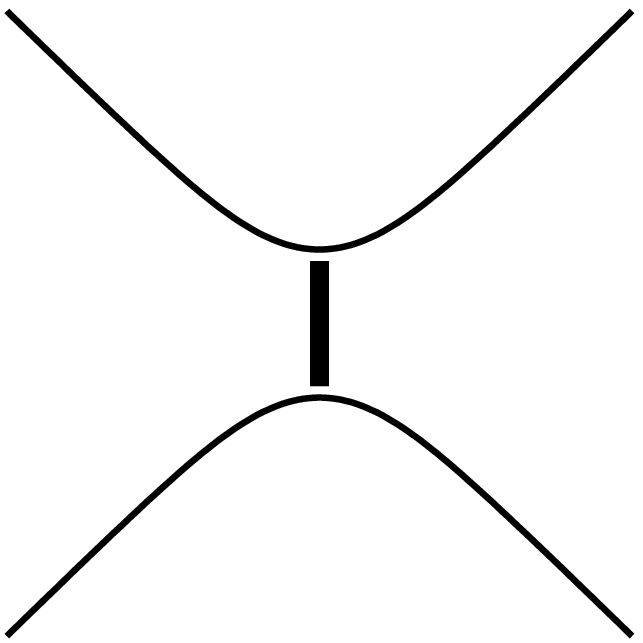}\put(52,48){\tiny$\al_c$}\end{overpic}}
\caption{A crossing, its $0$-resolution, its $1$-resolution, and its
  $0$-resolution along with the surgery arc. If we forget the surgery
  arc, we get the $0$-resolution; and if we perform embedded surgery
  along it, we get the $1$-resolution.}\label{fig:0-1}
\end{figure}
\begin{enumerate}[label=(X-\arabic*), ref=(X-\arabic*)]
\item $L\sbs \R^3$ is a link represented by an $n$-crossing link
  diagram $D\sbs\R^2=\R^2\times\{0\}\sbs \R^3$; let $n_+$
  (respectively, $n_-$) denote the number of positive (respectively,
  negative) crossings in $D$. Sometimes we work with a pointed link,
  that is, a link along with a single basepoint on it, and we
  represent it by a pointed link diagram $(D,p)$ where $p$ is some
  basepoint on $D$.
\item Let $\Crossings$ be the set of crossings in $D$. For any subset
  $u\sbseq\Crossings$, let $\AssRes{u}$ be the following complete
  resolution of $D$: resolve a crossing $c$ by the $0$-resolution if
  $c\notin u$, and resolve it by the $1$-resolution otherwise, see
  Figure~\ref{fig:res-crossing}--\ref{fig:res-1}; let $Z(\AssRes{u})$
  denote the set of circles in $\AssRes{u}$.
\item\label{item:Khovanov-generators} The Khovanov chain group
  $\Cx[u]$ at $u$ is the exterior algebra generated by the circles in
  $\AssRes{u}$; i.e., if $Z(\AssRes{u})=\{x_1,\dots,x_k\}$, then
  $\Cx[u]=\wedge^*\F_2\langle x_1,\dots,x_k\rangle$; in particular,
  $\Cx[u]$ has a distinguished basis $\Basis{u}$ consisting of the
  square-free monomials in $x_i$. If $(D,p)$ is a pointed link
  diagram, let $\Basis{u}^-\sbs\Basis{u}$ be the subset consisting of
  the monomials where the circle containing $p$ appears; and let
  $\Basis{u}^+=\Basis{u}\sm\Basis{u}^-$. The minus and the plus
  Khovanov chain groups, $\Cx[u]^-$ and $\Cx[u]^+$, are the
  $\F_2$-vector subspaces of $\Cx[u]$ generated by $\Basis{u}^-$ and
  $\Basis{u}^+$, respectively.
\item\label{item:Kh-hom-grading} A Khovanov generator is a pair
  $(u,x)$ where $u\subseteq\Crossings$ and $x\in\Basis{u}$.  The
  Khovanov generators are endowed with a bigrading $(\gr_h,\gr_q)$;
  the homological grading $\gr_h$ and a quantum grading $\gr_q$ are
  given by
  \begin{align*}
    \gr_h((u,x))&=-n_-+\card{u}\\
    \gr_q((u,x))&=n_+-2n_-+\card{u}+\card{Z(\AssRes{u})}-2\deg(x).
  \end{align*}
  The so-called delta grading $\gr_{\diff}$ is defined as
  $\gr_h-\gr_q/2$.  The set of all Khovanov generators forms a basis
  for the total $(\gr_h,\gr_q)$-bigraded Khovanov chain group
  $\Cx=\displaystyle\mathop\oplus_{u\subseteq\Crossings}\Cx[u]$.
\item When working with a pointed link diagram, a minus (respectively,
  plus) Khovanov generator is a pair $(u,x)$ where $u\subseteq\Crossings$
  and $x\in\Basis{u}^-$ (respectively, $x\in\Basis{u}^+$).  The set of
  all minus (respectively, plus) Khovanov generators forms a basis for
  the total minus (respectively, plus) Khovanov chain group
  $\Cx^-=\displaystyle\mathop\oplus_{u\subseteq\Crossings}\Cx[u]^-$
  (respectively,
  $\Cx^+=\displaystyle\mathop\oplus_{u\subseteq\Crossings}\Cx[u]^+$). The
  $\gr_q$-gradings are shifted by $-1$ (respectively, $1$) for the
  minus (respectively, plus) theory; that is, $\Cx^-$ is a subspace of
  $\Cx{}\{-1\}$, while $\Cx^+$ is a subspace of $\Cx{}\{1\}$ where
  $\{\,\}$ denotes the grading shift operator for the second grading
  (with the usual convention $C\{a\}^{i,j}\cong C^{i,j+a}$ for any
  bigraded $\F_2$-vector space $C$, any bigrading $(i,j)$, and any
  shift $a$ for the second grading).
\end{enumerate}

\begin{definition}
  A \emph{resolution configuration} $R$ consists of a set $Z(R)$ of
  smoothly embedded disjoint circles in $\R^2$ and a set $A(R)$ of
  properly embedded disjoint arcs in $(\R^2,Z(R))$.
  \begin{enumerate}
  \item The number of arcs in $A(R)$ is called the index of the
    resolution configuration.
  \item If the arcs in $A(R)$ are all oriented, then $R$ is called an
    oriented resolution configuration.
  \item Two (oriented) resolution configurations are equivalent if
    there is an isotopy, or equivalently, an orientation-preserving
    diffeomorphism of $S^2=\R^2\cup\{\infty\}$, that carries one to
    the other.
  \item The mirror $\Mirror{R}$ of a resolution configuration $R$ is
    obtained from $R$ by reflecting it along the line $\{0\}\times\R$.
  \item The reverse $\Reverse{R}$ of an oriented resolution
    configuration $R$ is obtained from $R$ by reversing the
    orientation of all the arcs in $A(R)$.
  \item The dual of an index-$k$ (oriented) resolution configuration
    $R$ is another index-$k$ (oriented) resolution configuration
    $\Dual{R}$, so that the circles in $Z(\Dual{R})$ are obtained from
    the circles in $Z(R)$ by performing embedded surgeries along the
    arcs in $A(R)$, and the arcs in $A(\Dual{R})$ are obtained by
    rotating the arcs in $A(R)$ by $90^{\circ}$ counter-clockwise.
  \item The circles in $Z(R)$ are called the starting circles of $R$,
    and the circles in $Z(\Dual{R})$ are called the ending circles of
    $R$.
  \item The circles in $Z(R)$ that are disjoint from all the arcs in
    $A(R)$ are called the passive circles; the rest of the circles are
    called the active circles. Note, the passive circles of $R$ are in
    natural correspondence with the passive circles of $\Dual{R}$.
  \item A labeled resolution configuration $(R,x,y)$ is a resolution
    configuration $R$ along with a square-free monomial $x$ in the
    starting circles $\{x_i\}$ of $R$ and a square-free monomial $y$
    in the ending circles $\{y_i\}$ of $R$. When we draw a labeled
    resolution configuration, we label the starting circles $0$ or
    $1$: a starting circle appears in the monomial $x$ if it is
    labeled $1$, and does not appear in $x$ if it is labeled $0$; and
    we draw the ending circles solid red or dashed blue: an ending
    circle appears in the monomial $y$ if it is solid and colored red,
    and does not appear in $y$ if it is dashed and colored blue.
  \item The dual $(\Dual{R},\Dual{y},\Dual{x})$ of a labeled
    resolution configuration $(R,x,y)$ is defined as follows: a
    starting circle of $\Dual{R}$ appears in the monomial $\Dual{y}$
    if and only if the corresponding ending circle of $R$ does not
    appear in $y$; and an ending circle of $\Dual{R}$ appears in the
    monomial $\Dual{x}$ if and only if the corresponding starting
    circle of $R$ does not appear in $x$. See \Figure{resconfig-dual}
    for an oriented labeled resolution configuration and its dual.
  \end{enumerate}
\end{definition}
\captionsetup[subfloat]{width={0.4\textwidth}}
\begin{figure}
  \centering \subfloat[An oriented labeled resolution configuration
  $(R,x,y)$.]{\begin{overpic}[width=0.4\textwidth]{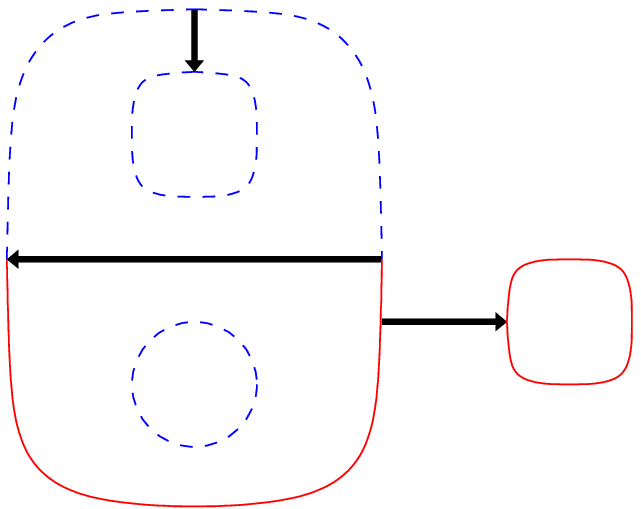}\put(89,22){\tiny
        $1$}\put(29,12){\tiny $0$}\put(7,7){\tiny
        $0$}\put(29,52){\tiny $0$}\end{overpic}}
  \hspace{0.1\textwidth} \subfloat[The dual oriented labeled
  resolution configuration
  $(\Dual{R},\Dual{y},\Dual{x})$.]{\begin{overpic}[width=0.4\textwidth]{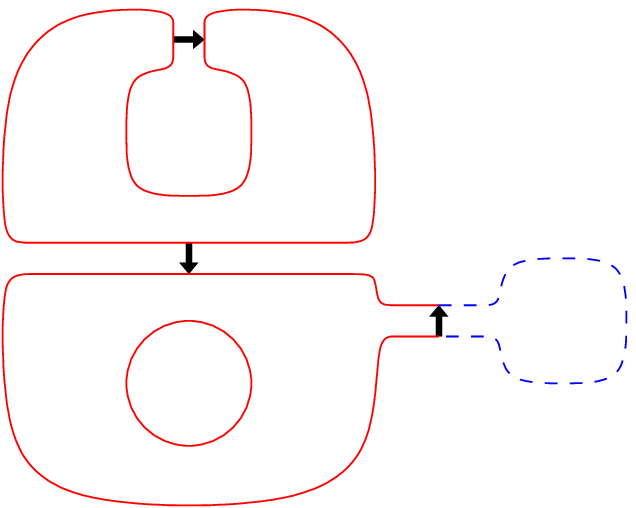}\put(29,12){\tiny
        $1$}\put(7,7){\tiny $0$}\put(7,43){\tiny $1$}\end{overpic}}
  \caption{An oriented labeled resolution configuration and its
    dual. Recall, a starting circle appears in the monomial if it is
    labeled $1$ and does not appear in the monomial if it is labeled
    $0$; and an ending circle appears in the monomial if it is solid
    red and does not appear in the monomial if it is dashed
    blue.}\label{fig:resconfig-dual}
\end{figure}

\begin{definition}
 For nested subsets $u\sbseq v\sbseq\Crossings$ with
  $\card{v}-\card{u}=k$, let $\ResConfig{u}{v}$ denote the following
  index-$k$ resolution configuration: resolve a crossing $c$ by the
  $0$-resolution if $c\notin v$, resolve it by the $0$-resolution and
  add the surgery arc $\al_c$ if $c\in v\sm u$, and resolve it by the
  $1$-resolution otherwise; see \Figure{0-1}. Note, the starting
  circles of $Z(\ResConfig{u}{v})$ are the circles in $Z(\AssRes{u})$
  and the ending circles of $Z(\ResConfig{u}{v})$ are the circles of
  $Z(\AssRes{v})$. 

  A \emph{decoration} is a choice of an orientation of all the arcs in
  $A(\ResConfig{\varnothing}{\Crossings})$. We specify a decoration by
  drawing an arrowhead near each crossing of the link diagram $D$ so
  that the arcs in $A(\ResConfig{\varnothing}{\Crossings})$ are
  oriented in accord with the arrowheads. A decoration induces an
  orientation of all arcs in all resolution configurations
  $\ResConfig{u}{v}$, see \Figure{figure-eight}.
\end{definition}
\captionsetup[subfloat]{width={0.4\textwidth}}
\begin{figure}
  \centering \subfloat[A decorated link diagram $D$ with
  $\Crossings=\{c_1,\dots,c_4\}$.]{\begin{overpic}[width=0.4\textwidth]{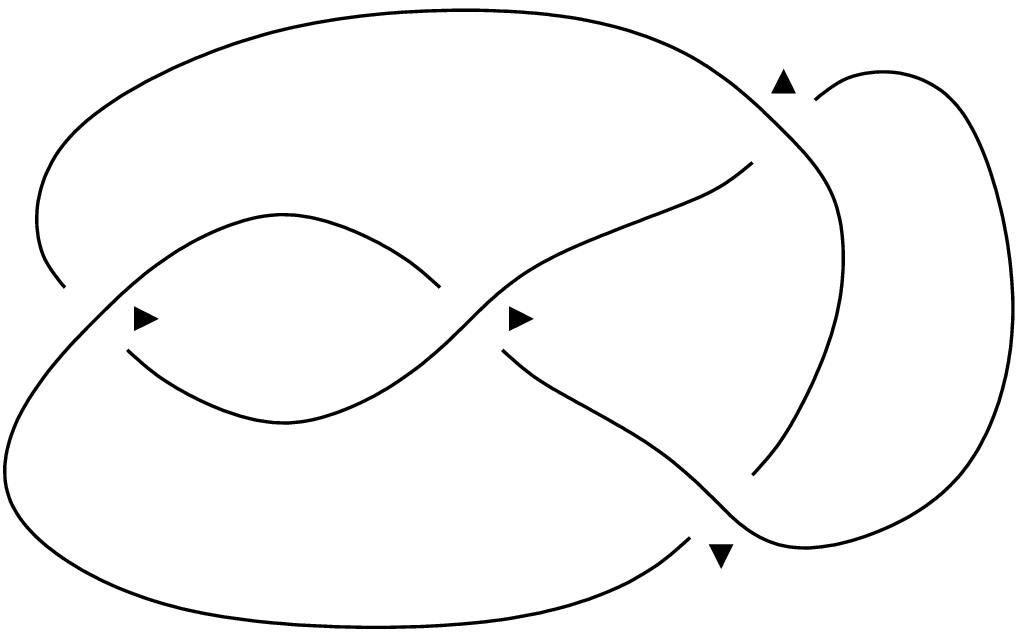}\put(9,25){\tiny
        $c_1$}\put(45,25){\tiny $c_2$}\put(65,11){\tiny
        $c_3$}\put(71,49){\tiny $c_4$}\end{overpic}}
  \hspace{0.1\textwidth} \subfloat[The oriented resolution
  configuration
  $\ResConfig{\{c_1\}}{\{c_1,c_2,c_3\}}$.]{\includegraphics[width=0.4\textwidth]{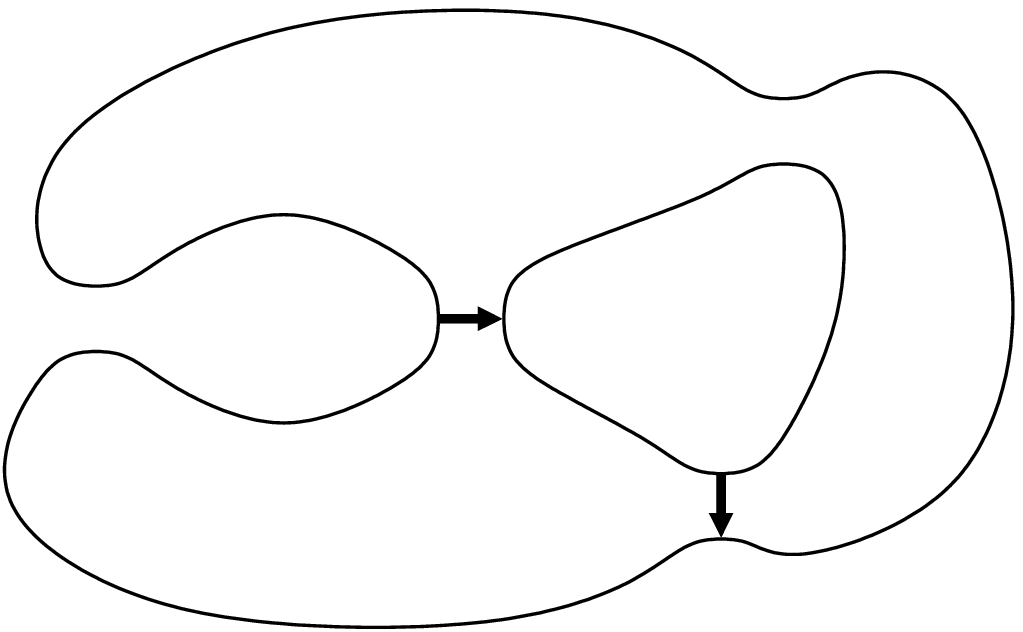}}
\caption{A decorated link diagram for the figure eight knot, and one of the
  induced oriented resolution configurations.}\label{fig:figure-eight}
\end{figure}

\begin{definition}
  A \emph{contribution function} $\cont{f}$ is a function from the set
  of all oriented labeled resolution configurations to $\F_2$.
  \begin{enumerate}
  \item $\cont{f}$ satisfies the naturality rule if it is preserved
    under equivalence. That is, if there is an isotopy of $S^2$
    carrying $(R,x,y)$ to $(R',x',y')$, then
    $\cont{f}((R,x,y))=\cont{f}((R',x',y'))$.
  \item $\cont{f}$ satisfies the conjugation rule if it is preserved
    under reversing. That is, for any $(R,x,y)$, 
    $\cont{f}((R,x,y))=\cont{f}((\Reverse{R},x,y))$.
  \item $\cont{f}$ satisfies the disoriented rule if it is preserved
    under an arbitrary re-orientation of arcs. That is, if $R$ and
    $R'$ differ only in the orientation of their arcs, then
    $\cont{f}((R,x,y))=\cont{f}((R',x,y))$.
  \item $\cont{f}$ satisfies the duality rule if it is preserved under
    dualizing and taking mirrors. That is, if
    $\Mirror{\Dual{R},\Dual{y},\Dual{x}}$ denotes the mirror of the
    dual of $(R,x,y)$, then
    $\cont{f}(\Mirror{\Dual{R},\Dual{y},\Dual{x}})=\cont{f}((R,x,y))$.
  \item $\cont{f}$ satisfies the extension rule if it only depends on
    the active part of the resolution configuration. That is, given
    $(R,x,y)$, let $R_a$ (respectively, $x_a$, $y_a$) be the active
    part of $R$ (respectively, $x$, $y$) and let $R_p$ (respectively,
    $x_p$, $y_p$) be the passive part of $R$ (respectively, $x$, $y$)
    so that $R=R_a\coprod R_p$ (respectively, $x=x_ax_p$,
    $y=y_ay_p$); then
    \[
    \cont{f}((R,x,y))=\begin{cases}\cont{f}((R_a,x_a,y_a))&\text{ if $x_p=y_p$,}\\0&\text{ otherwise.}\end{cases}
    \]
  \item $\cont{f}$ satisfies the filtration rule if it does not
    contribute whenever there is some point $p$ such that the starting
    circle containing $p$ is in the starting monomial, but the ending
    circle containing $p$ is not in the ending monomial. That is, if
    $\cont{f}((R,x,y))\neq 0$, then, every starting circle that appears
    in $x$ is disjoint from every ending circle that does not appear
    in $y$. 
  \end{enumerate}
 A contribution function $\cont{f}$ satisfying the naturality rule
  defines an endomorphism $\enmor{f}$ of the Khovanov chain group $\Cx$
  coming from a decorated link diagram as follows. Given Khovanov
  generators $(u,x)$ and $(v,y)$,
  \[
  \langle\enmor{f}((u,x)),(v,y)\rangle=\begin{cases}\cont{f}((\ResConfig{u}{v},x,y))&\text{ if $u\sbseq
      v$,}\\0&\text{ otherwise.}\end{cases}
  \]
  Clearly, if the contribution function $\cont{f}$ satisfies the
  disoriented rule, then the endomorphism $\enmor{f}$ does not depend on
  the choice of decoration of the link diagram. Also, if $\cont{f}$
  satisfies the filtration rule, and we are working with a pointed
  link diagram, then $\enmor{f}$ restricts to an
  endomorphism of $\Cx^-$.
\end{definition}

With these concepts in place, we are ready to recall the definitions
of the Khovanov differential, the Bar-Natan differential, and the
\Szabo{} differential, although not entirely in their original forms.
\begin{definition}[{\cite[Section~4.2]{Kho-kh-categorification}}]\label{def:khovanov-differential}
  The \emph{Khovanov contribution function} $\cont{k}$ satisfies the
  naturality rule, the disoriented rule (and hence the conjugation
  rule), the duality rule, the extension rule, and the filtration
  rule.  A labeled resolution configuration $(R,x,y)$ has non-zero
  contribution in $\cont{k}$ if and only if $x$ and $y$ agree on the
  passive circles, and the active part is equivalent to one of the
  four configurations in \Figure{khovanov-configurations}.

  In the Merge configurations, there are two starting circles with a
  single arc between them. In the Merge-A configuration, both the
  starting and the ending monomials are $1$, while in the Merge-B
  configuration, the starting monomial and the ending monomial each
  contains exactly one circle. The Split-A configuration and the
  Split-B configuration are the duals of the Merge-A configuration and
  the Merge-B configuration, respectively.

  Starting with a link diagram, the Khovanov differential
  $\diffKh=\dd_1$ is the $(1,0)$-graded endomorphism on $\Cx$ induced
  by $\cont{k}$,
  \[
  \langle\dd_1((u,x)),(v,y)\rangle=\begin{cases}\cont{k}((\ResConfig{u}{v},x,y))&\text{ if $u\sbseq
      v$,}\\0&\text{ otherwise.}\end{cases}
  \]
  It is indeed a differential, i.e., $(\diffKh)^2=0$, and the Khovanov
  chain complex is defined as
  \[
  \KhCx=(\Cx,\diffKh).
  \]
  
  Since $\cont{k}$ satisfies the filtration rule, for a pointed link
  diagram, $\diffKh$ restricts to an endomorphism on $\Cx^-$. Therefore,
  we get a subcomplex $\KhCx^-$ generated by $\Cx^-$ and a quotient
  complex $\KhCx^+$ generated by $\Cx^+$, which fit into a short exact
  sequence
  \[
  0\to\KhCx^-\{1\}\to\KhCx\to\KhCx^+\{-1\}\to 0.
  \]
\end{definition}

\begin{figure}
\begin{overpic}[width=0.9\textwidth]{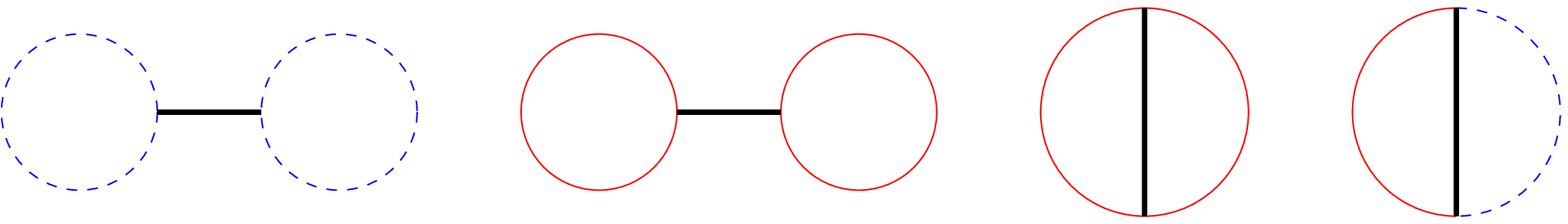}
\put(9,1){Merge-A}
\put(42,1){Merge-B}
\put(70,1){Split-A}
\put(90,1){Split-B}
\put(4.5,9){\tiny $0$}
\put(21,9){\tiny $0$}
\put(89,9){\tiny $0$}
\put(38,9){\tiny $0$}
\put(54,9){\tiny $1$}
\put(69,9){\tiny $1$}
\end{overpic}
  \caption{The labeled configurations that contribute to the
    function $\cont{k}$.}\label{fig:khovanov-configurations}
\end{figure}

\begin{figure}
\begin{overpic}[width=0.9\textwidth]{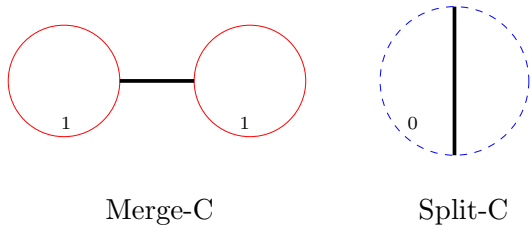}
\put(42,1){Merge-C}
\put(70,1){Split-C}
\put(38,9){\tiny $1$}
\put(54,9){\tiny $1$}
\put(69,9){\tiny $0$}
\end{overpic}
  \caption{The labeled configurations that contribute to the
    function $\cont{b}$.}\label{fig:bar-natan-configurations}
\end{figure}

\begin{definition}[{\cite[Section~9.3]{Bar-kh-tangle-cob}}]\label{def:barnatan-differential}
  The \emph{Bar-Natan contribution function} $\cont{b}$ satisfies the
  naturality rule, the disoriented rule (and hence the conjugation
  rule), the duality rule, the extension rule, and the filtration
  rule.  A labeled resolution configuration $(R,x,y)$ has non-zero
  contribution in $\cont{b}$ if and only if $x$ and $y$ agree on the
  passive circles, and the active part is equivalent to one of the two
  configurations in \Figure{bar-natan-configurations}.

  In the Merge-C configuration, there are two starting circles with a
  single arc between them, and both the starting and the ending
  monomials contain all the circles; the Split-C configuration is the
  dual of the Merge-C configuration.

  Starting with a link diagram, the Bar-Natan perturbation $\hh_1$ is
  the $(1,2)$-graded endomorphism on $\Cx$ induced by $\cont{b}$,
  \[
  \langle\hh_1((u,x)),(v,y)\rangle=\begin{cases}\cont{b}((\ResConfig{u}{v},x,y))&\text{ if $u\sbseq
      v$,}\\0&\text{ otherwise.}\end{cases}
  \]
  The Bar-Natan differential is the $(1,0)$-graded endomorphism
  $\diffBN=\dd_1+H\hh_1$ on $\Cx\otimes\F_2[H]$, with $H$ being a
  formal variable in $(\gr_h,\gr_q)$-bigrading $(0,-2)$. Once again,
  this is a differential, i.e., $(\diffBN)^2=0$, and the Bar-Natan
  chain complex is defined as
  \[
  \BNCx=(\Cx\otimes\F_2[H],\diffBN).
  \]
  Clearly, $\KhCx$ is obtained from $\BNCx$ by setting $H=0$. The
  filtered and the localized versions are defined as 
  \begin{align*}
    \fBNCx&=\BNCx/\{H=1\}=(\Cx,\dd_1+\hh_1)\\
    \lBNCx&=\BNCx\otimes_{\F_2[H]}\F_2[H,H^{-1}]=(\Cx\otimes\F_2[H,H^{-1}],\diffBN).
  \end{align*}

  Once again, since both $\cont{k}$ and $\cont{b}$ satisfy the
  filtration rule, for a pointed link diagram, we get a subcomplex
  $\BNCx^-$ generated by $\Cx^-$ and a quotient complex $\BNCx^+$
  generated by $\Cx^+$ fitting into a short exact sequence over
  $\F_2[H]$
  \[
  0\to\BNCx^-\{1\}\to\BNCx\to\BNCx^+\{-1\}\to 0.
  \]
\end{definition}

\captionsetup[subfloat]{width={0.4\textwidth}}
\begin{figure}
  \centering 
  \subfloat[Type-A
  configuration.]{\label{fig:A}\begin{overpic}[width=0.4\textwidth]{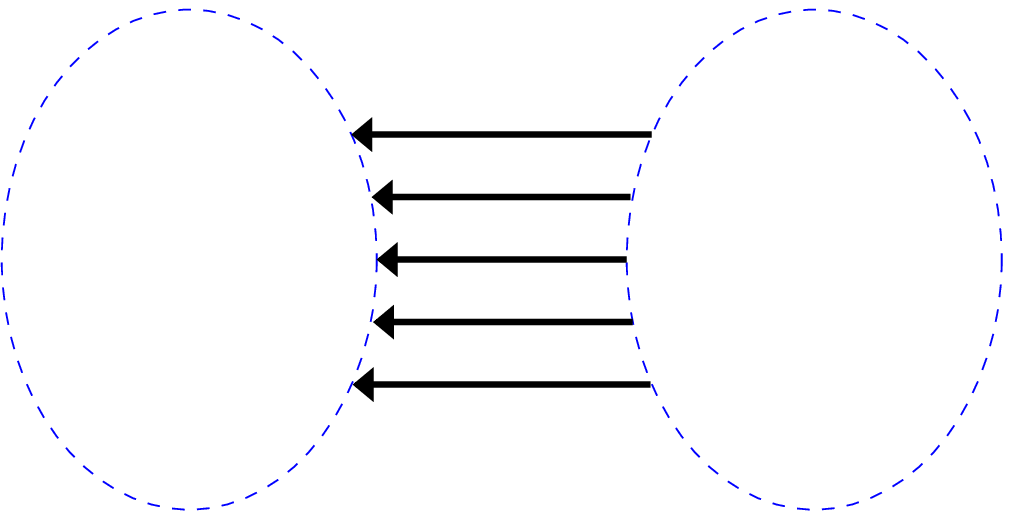}\put(7,7){\tiny $0$}\put(91,7){\tiny $0$}\end{overpic}}
  \hspace{0.1\textwidth} 
  \subfloat[Type-B configuration.]{\label{fig:B}\begin{overpic}[width=0.4\textwidth]{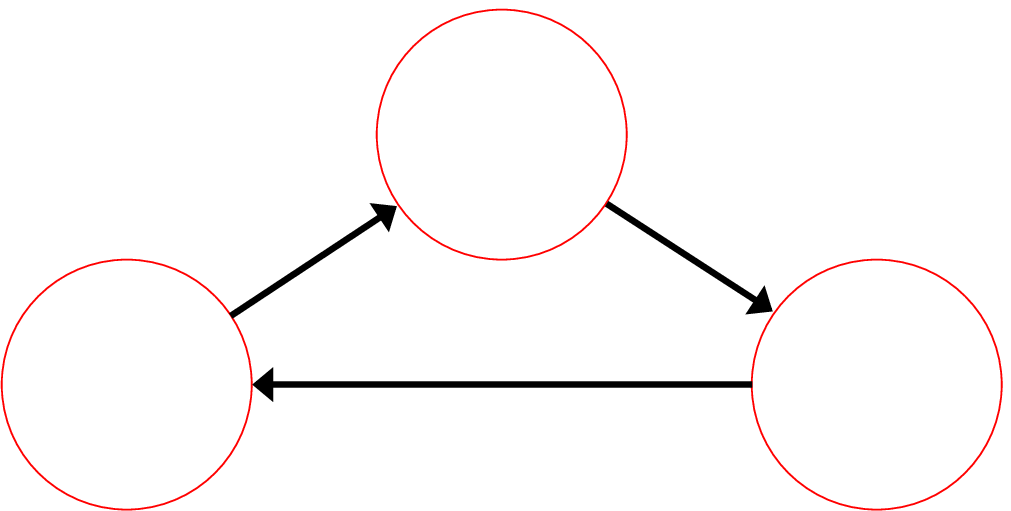}\put(5,5){\tiny
        $1$}\put(93,5){\tiny $1$}\put(49,28){\tiny $1$}\end{overpic}}
  \hspace{0.1\textwidth} 
  \subfloat[Type-C
  configuration.]{\label{fig:C}\begin{overpic}[width=0.4\textwidth]{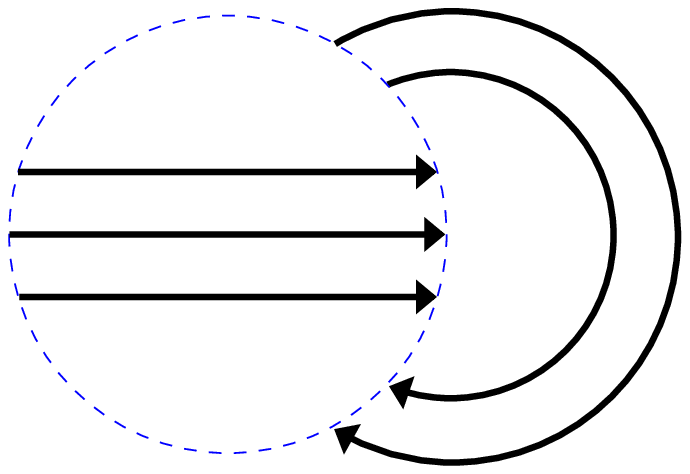}\put(29,9){\tiny $0$}\end{overpic}}
  \hspace{0.1\textwidth} 
  \subfloat[Type-D
  configuration.]{\label{fig:D}\begin{overpic}[width=0.4\textwidth]{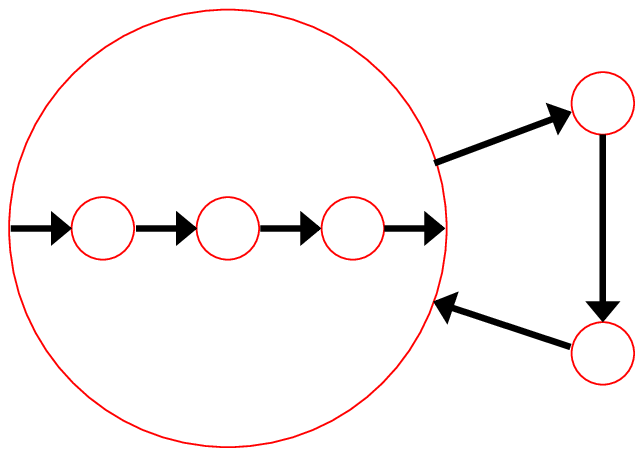}\put(30,9){\tiny
        $1$}\put(18,20){\tiny $1$}\put(30,20){\tiny
        $1$}\put(43,20){\tiny $1$}\put(73,39){\tiny $1$}\put(73,15){\tiny $1$}\end{overpic}}
  \hspace{0.1\textwidth} 
  \subfloat[Type-E configuration.]{\label{fig:E}\begin{overpic}[width=0.4\textwidth]{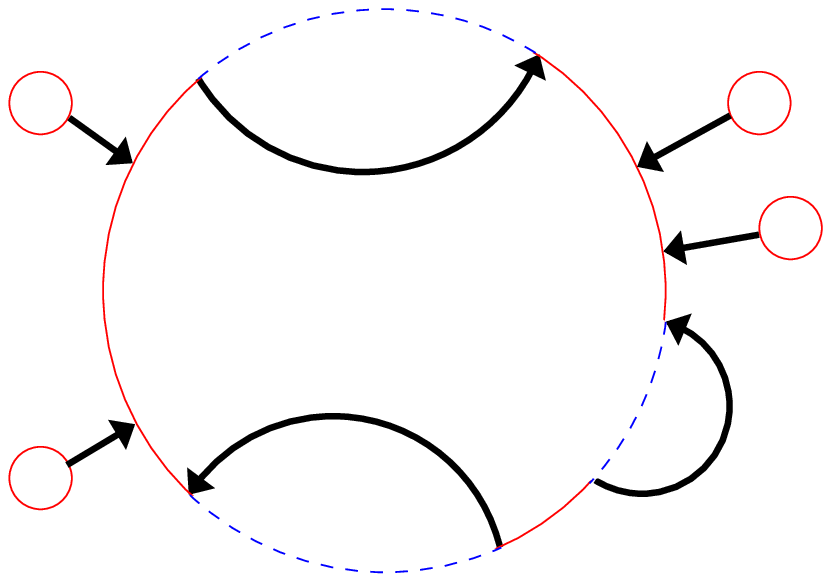}\put(15,3){\tiny
        $1$}\put(15,40){\tiny $1$}\put(23,25){\tiny
        $0$}\put(91,45){\tiny $1$}\put(94,33){\tiny $1$}\end{overpic}}
   \hspace{0.1\textwidth} 
  \subfloat[Type-E configuration.]{\label{fig:rE}\begin{overpic}[width=0.4\textwidth]{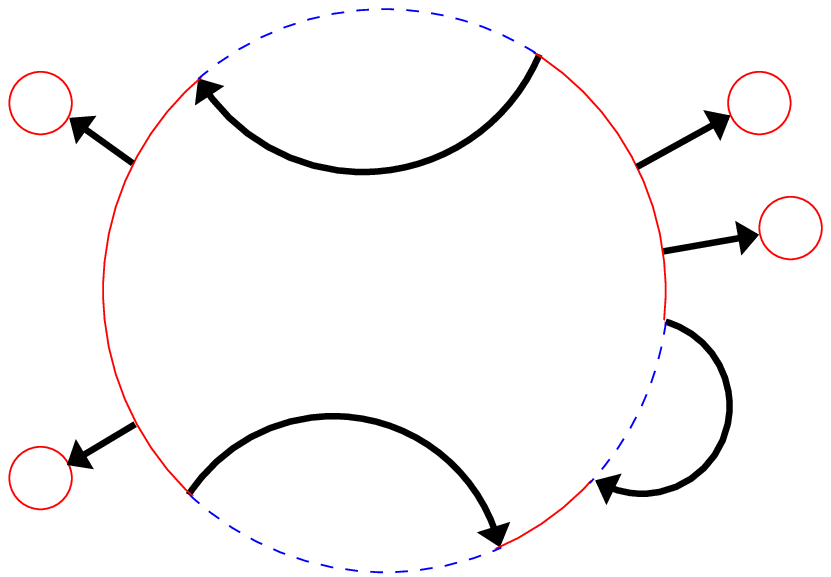}\put(15,3){\tiny
        $1$}\put(15,40){\tiny $1$}\put(23,25){\tiny
        $0$}\put(91,45){\tiny $1$}\put(94,33){\tiny $1$}\end{overpic}}
  \caption{The oriented labeled configurations that contribute to the
    function $\cont{d}$.}\label{fig:zoltan-configurations}
\end{figure}
\begin{definition}[{\cite[Section~3--4]{Szab-kh-geometric}}]\label{def:szabo-differential}
  The \emph{\Szabo{} contribution function} $\cont{d}$ satisfies the
  naturality rule, the conjugation rule (but not the disoriented
  rule), the duality rule, the extension rule, and the filtration
  rule. An oriented labeled resolution configuration $(R,x,y)$ has
  non-zero contribution in $\cont{d}$ if and only if $x$ and $y$ agree
  on the passive circles, and the active part is equivalent to a
  configuration from one of the five families in
  \cite[Figure~3]{Szab-kh-geometric} or
  \Figure{zoltan-configurations}.

  In a Type-A configuration (\Figure{A}), there are two starting
  circles, with some (at least one) parallel arcs between them; both
  the starting and the ending monomials are $1$. A Type-B
  configuration (\Figure{B}) is the dual of a Type-A configuration. In
  a Type-C configuration (\Figure{C}), there is a single starting
  circle, with some (at least one) parallel arcs inside, and some (at
  least one) parallel arcs outside, so that the inside arcs and the
  outside arcs are linked, and the orientations are as shown; both the
  starting and ending monomials are $1$. A Type-D configuration
  (\Figure{D}) is the mirror of the dual of a Type-C
  configuration. The Type-E configurations are slightly harder to
  describe. In the Type-E configuration from \Figure{E}, there is
  exactly one circle, call it the special starting circle, that does
  not appear in the starting monomial; and there is exactly one
  circle, call it the special ending circle, that does appear in the
  ending monomial. All the non-special starting circles are incident
  to exactly one arc; and all the non-special ending circles are
  incident to exactly one dual arc. The arcs either run from the
  non-special starting circles to the special starting circle, or run
  from the special starting circle to itself; in the latter case, they
  are oriented as shown (depending on whether they are inside or
  outside of the special starting circle). We also impose the
  condition that the index of the configuration is at least one. The
  Type-E configuration from \Figure{rE} is the reverse of the Type-E
  configuration from \Figure{E}.

  Working with a decorated link diagram, the induced endomorphism
  $\dd$ of $\Cx$, defined as
  \[
  \langle\enmor{d}((u,x)),(v,y)\rangle=\begin{cases}\cont{d}((\ResConfig{u}{v},x,y))&\text{
      if $u\sbseq v$,}\\0&\text{ otherwise,}\end{cases}
  \]
  increases the homological grading $\gr_h$ by at least one, and
  increases delta grading $\gr_{\diff}=\gr_h-\gr_q/2$ by exactly one;
  therefore, $\dd$ can be (uniquely) written as
  $\sum_{i=1}^{\infty}\dd_i$, where $\dd_i$ increases $\gr_h$ by
  exactly $i$ and increases $\gr_q$ by exactly $2i-2$. Indeed, $\dd_1$
  is independent of the initial decoration, and is the Khovanov
  differential from \Definition{khovanov-differential}. The \Szabo{}
  geometric differential is the $(1,0)$-graded endomorphism
  $\diffSz=\dd_1+W\dd_2+W^2\dd_3+\cdots$ on $\Cx\otimes\F_2[W]$ with
  $W$ being a formal variable in $(\gr_h,\gr_q)$-bigrading
  $(-1,-2)$. It is shown in \cite[Section~6]{Szab-kh-geometric} that
  this is a differential, i.e., $(\diffSz)^2=0$, and the \Szabo{}
  chain complex is defined as
  \[
  \SzCx=(\Cx\otimes\F_2[W],\diffSz).
  \]
  The Khovanov chain complex $\KhCx$ can be recovered from $\SzCx$ by
  setting $W=0$. The filtered version is defined
  as
  \[
    \fSzCx=\SzCx/\{W=1\}=(\Cx,\dd_1+\dd_2+\dd_3+\dots)=(\Cx,\dd).
  \]

  As before, since $\cont{d}$ satisfies the filtration rule, for a
  pointed decorated link diagram, we get a subcomplex $\SzCx^-$
  generated by $\Cx^-$ and a quotient complex $\SzCx^+$ generated by
  $\Cx^+$ fitting into a short exact sequence over $\F_2[W]$
  \[
  0\to\SzCx^-\{1\}\to\SzCx\to\SzCx^+\{-1\}\to 0.
  \]
\end{definition}

Therefore, to summarize the existing story, we have a $(1,2)$-graded
endomorphism $\hh_1$, and for all $i>0$, $(i,2i-2)$-graded
endomorphisms $\dd_i$ for $\Cx$, satisfying $(\dd_1+\hh_1)^2=\dd^2=0$,
where $\dd=(\dd_1+\dd_2+\dd_3+\cdots)$. From grading considerations,
this breaks up into the following equations
\begin{align*}
\hh_1^2&=0&\Commute{\hh_1}{\dd_1}&=0&\dd_1^2&=0\\
&&&&\Commute{\dd_1}{\dd_2}&=0\\
&&&&\Commute{\dd_1}{\dd_3}+\dd_2^2&=0\\
&&&&\cdots&
\end{align*}
Therefore, the Khovanov chain complex, the Bar-Natan chain complex,
and the \Szabo{} geometric chain complex can be defined as the
following objects, respectively:
\begin{align*}
  \KhCx&=(\Cx,\diffKh\defeq\dd_1)\\
  \BNCx&=(\Cx\otimes\F_2[H],\diffBN\defeq\dd_1+H\hh_1)\\
  \SzCx&=(\Cx\otimes\F_2[W],\diffSz\defeq\dd_1+W\dd_2+W^2\dd_3+\cdots)
\end{align*}
Here $H$ and $W$ are formal variables carrying
$(\gr_h,\gr_q)$-bigradings $(0,-2)$ and $(-1,-2)$, respectively, so
that above complexes are also bigraded with the differential being of
grading $(1,0)$.

In the next section, we will extend the definition of $\hh_1$ to
construct $(i,2i)$-graded endomorphisms $\hh_i$ for all $i$. Setting
$\hh=(\hh_1+\hh_2+\hh_3+\cdots)$, it will satisfy
$\hh^2=\Commute{\hh}{\dd}=0$. Expanding along gradings, we will get equations
\begin{align*}
  \hh_1^2&=0&\Commute{\hh_1}{\dd_1}&=0\\
  \Commute{\hh_1}{\hh_1}&=0&\Commute{\hh_1}{\dd_2}+\Commute{\hh_2}{\dd_1}&=0\\
  \Commute{\hh_1}{\hh_3}+\hh_2^2&=0&\Commute{\hh_1}{\dd_3}+\Commute{\hh_2}{\dd_2}+\Commute{\hh_3}{\dd_1}&=0\\
  \cdots&&\cdots&
\end{align*}
and we will define our chain complex as
\[
\OurCx=(\Cx\otimes\F_2[H,W],\diffOur\defeq\dd_1+W\dd_2+W^2\dd_3
+\cdots+H\hh_1+HW\hh_2+HW^2\hh_3+\cdots).
\]
with the differential lying in grading $(1,0)$. Clearly, $\BNCx$ and
$\SzCx$ can be obtained from $\OurCx$ by setting $W=0$ and $H=0$,
respectively, while $\KhCx$ can be obtained by setting both to zero;
therefore, $\OurCx$ will be the master chain complex.

The endomorphism $\dd$ was defined from a contribution function
$\cont{d}$, and $\dd_i$ was simply the graded part of $\dd$ that
increased the homological grading by $i$. We will also define the
endomorphism $\hh$ from some contribution function $\cont{h}$, and
$\hh_i$ will be graded part of $\hh$ that increases the homological
grading by $i$.  Therefore, in order to introduce our chain complex,
all that remains is to define the contribution function $\cont{h}$. We
turn to this in the next section.

\section{Construction}\label{sec:construction}

\begin{definition}\label{def:tree}
  A \emph{tree} is a labeled resolution configuration of some index $k\geq
  0$, with exactly $(k+1)$ starting circles, exactly one ending
  circle, and no passive circles, and with all the starting circles
  appearing in the starting monomial, and the ending circle appearing
  in the ending monomial.
\end{definition}

\begin{definition}\label{def:our-contribution-function}
  The contribution function $\cont{h}$ is defined as follows: It is
  non-zero for a labeled resolution configuration $(R,x,y)$ if only if
  $(R,x,y)$ is a disjoint union of trees and dual trees, and the
  index of $R$ is at least one. See \Figure{mixed-forest} for a
  resolution configuration that contributes to $\cont{h}$.
\end{definition}

\begin{figure}
\begin{overpic}[width=0.8\textwidth]{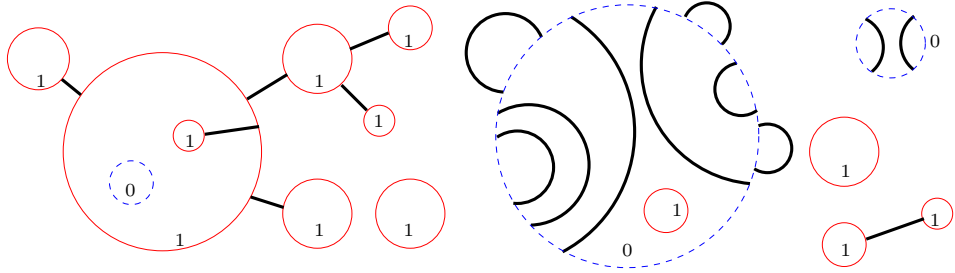}
\put(15,8){\tiny $0$}
\put(65,2){\tiny $0$}
\put(96,23){\tiny $0$}
\put(6,19){\tiny $1$}
\put(34,19){\tiny $1$}
\put(20,3){\tiny $1$}
\put(21,13){\tiny $1$}
\put(34,4){\tiny $1$}
\put(43,4){\tiny $1$}
\put(40,15){\tiny $1$}
\put(43,23){\tiny $1$}
\put(70,6){\tiny $1$}
\put(87,10){\tiny $1$}
\put(87,2){\tiny $1$}
\put(96,5){\tiny $1$}
\end{overpic}
\caption{A labeled configuration that contributes to $\cont{h}$.}\label{fig:mixed-forest}
\end{figure}

\begin{lemma}
  The contribution function $\cont{h}$ satisfies the naturality rule,
  the disoriented rule (and hence the conjugation rule), the duality
  rule, the extension rule, and the filtration rule. The induced
  endomorphism $\hh$ of $\Cx$ coming from some link diagram $D$ (or of
  $\Cx^\pm$ coming from some pointed link diagram $(D,p)$), defined as
  \[
  \langle\enmor{h}((u,x)),(v,y)\rangle=\begin{cases}\cont{h}((\ResConfig{u}{v},x,y))&\text{ if $u\sbseq
      v$,}\\0&\text{ otherwise,}\end{cases}
  \] 
  increases the homological grading $\gr_h$ by at least one, and
  preserves the delta grading $\gr_{\diff}=\gr_h-\gr_q/2$; therefore,
  $\hh$ can be (uniquely) written as $\sum_{i=1}^{\infty}\hh_i$, where
  $\hh_i$ increases $\gr_h$ by $i$ and increases $\gr_q$ by
  $2i$. Furthermore, $\hh_1$ is the Bar-Natan perturbation from
  \Definition{barnatan-differential}.
\end{lemma}

\begin{proof}
This is immediate from the definitions.
\end{proof}

\begin{prop}\label{prop:is-differential}
The endomorphism $\dd+\hh$ is a differential on $\Cx$. 
\end{prop}

The main aim of this section is to prove the above statement, and we
will give the proof at the very end. In the meantime,
\Proposition{is-differential} allows us to define the following chain
complex.

\begin{definition}\label{def:main-complex}
  Starting from a decorated link diagram, our chain complex is defined
  over $\F_2[H,W]$ as
  \[
  \OurCx=(\Cx\otimes\F_2[H,W],\diffOur=\dd_1+W\dd_2+W^2\dd_3
  +\cdots+H\hh_1+HW\hh_2+HW^2\hh_3+\cdots)
  \]
  with $H$ and $W$ being formal variables lying in grading $(0,-2)$
  and $(-1,-2)$, and the differential being of grading
  $(1,0)$. Setting $H=0$ recovers the \Szabo{} geometric chain complex
  from \Definition{szabo-differential}, while setting $W=0$ recovers
  the Bar-Natan chain complex from
  \Definition{barnatan-differential}. Some filtered versions and a localized
  version are defined as
  \begin{align*}
    \fOurCx&=\OurCx/\{H=W=1\}=(\Cx,\dd+\hh)=(\Cx,\dd_1+\dd_2+\dots+\hh_1+\hh_2+\cdots)\\
    \HfOurCx&=\OurCx/\{H=1\}=(\Cx,(\dd_1+\hh_1)+W(\dd_2+\hh_2)+W^2(\dd_3+\hh_3)+\cdots)\\
    \HlOurCx&=\OurCx\otimes_{\F_2[H,W]}\F_2[H,H^{-1},W]=(\Cx\otimes\F_2[H,H^{-1},W],\diffOur).
  \end{align*}
  For a pointed decorated link diagram, $\OurCx^-$ (respectively,
  $\OurCx^+$) is defined to be the subcomplex (respectively, quotient
  complex) generated by $\Cx^-$ (respectively, $\Cx^+$); they fit into
  a short exact sequence over $\F_2[H,W]$
  \[
  0\to\OurCx^-\{1\}\to\OurCx\to\OurCx^+\{-1\}\to 0.
  \]
\end{definition}

\begin{example}\label{exam:trefoil-complex}
  Consider the decorated link diagram for the trefoil and its cube of
  resolutions in \Figure{trefoil-decor}. We have numbered the three
  crossings from left to right as $c_1,c_2,c_3$; the complete
  resolution at $u\subseteq\Crossings$ has been labeled by the
  corresponding vertex $\ol{u}\in\{0,1\}^3$, where $\ol{u}_i=1$ iff
  $c_i\in u$. Furthermore, the circles in each individual resolution
  are also numbered (usually left to right, sometimes top to bottom);
  let $x^{\ol{u}}_i$ be the circle labeled $i$ in the resolution at
  $u$. We will use the superscript $\ol{u}$ for the Khovanov
  generators living over $u$.
  
  \setlength\tempfigdim{0.08\textheight}
  \begin{figure}
    \[
    \xymatrix{
      \vcenter{\hbox{\begin{overpic}[height=\tempfigdim]{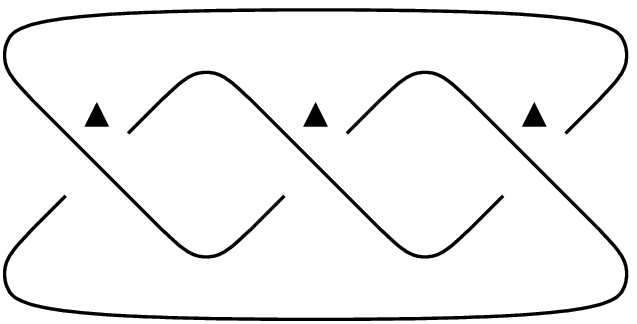}
            \put(12,15){\tiny $c_1$}\put(47,15){\tiny
              $c_2$}\put(82,15){\tiny $c_3$}
          \end{overpic}}}&&
      \vcenter{\hbox{\begin{overpic}[height=\tempfigdim]{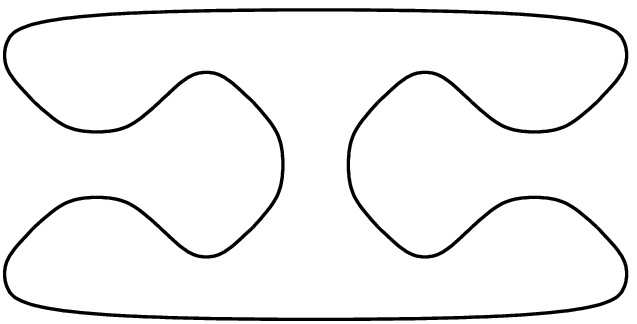}
            \put(5,5){\tiny $010$} 
            \put(5,40){\tiny $1$}
          \end{overpic}}}\ar[r]\ar[d]|(0.5)\hole&
      \vcenter{\hbox{\begin{overpic}[height=\tempfigdim]{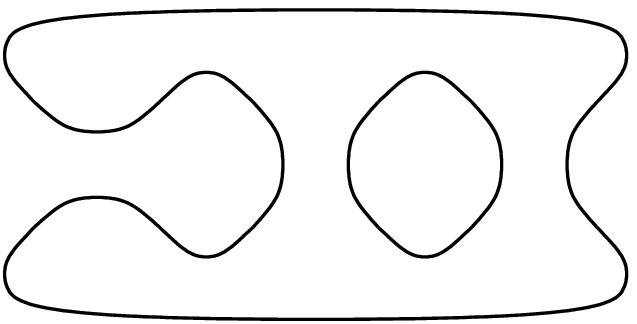}
            \put(5,5){\tiny $011$}
            \put(5,40){\tiny $1$}\put(65,30){\tiny $2$}
          \end{overpic}}}\ar[d]\\
      \vcenter{\hbox{\begin{overpic}[height=\tempfigdim]{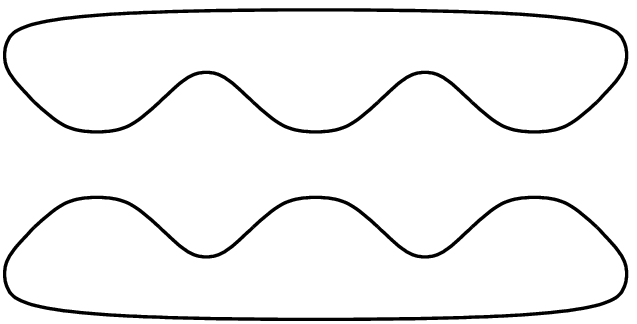}
            \put(5,5){\tiny $000$}
            \put(5,40){\tiny $1$}\put(47,10){\tiny $2$}
          \end{overpic}}}\ar[r]\ar[d]\ar[urr]&
      \vcenter{\hbox{\begin{overpic}[height=\tempfigdim]{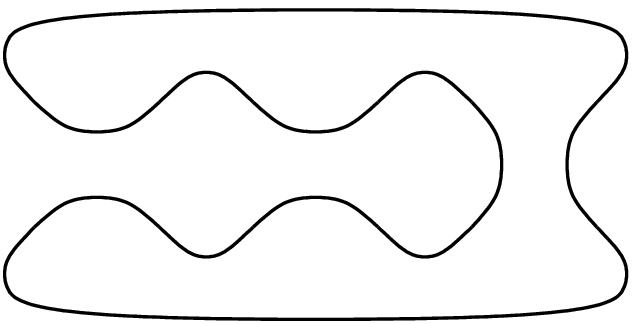}
            \put(5,5){\tiny $001$}
            \put(5,40){\tiny $1$} 
         \end{overpic}}}\ar[d]\ar[urr]&
      \vcenter{\hbox{\begin{overpic}[height=\tempfigdim]{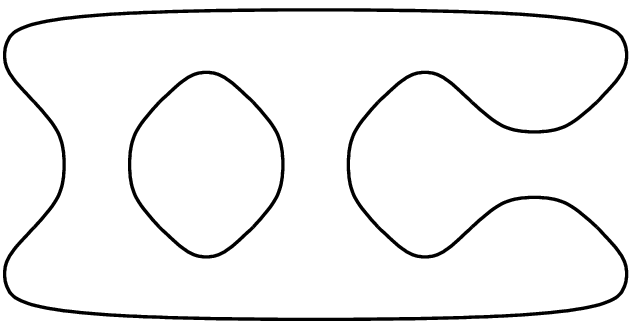}
            \put(5,5){\tiny $110$}
            \put(5,40){\tiny
              $1$}\put(30,30){\tiny $2$}
          \end{overpic}}}\ar[r]&
      \vcenter{\hbox{\begin{overpic}[height=\tempfigdim]{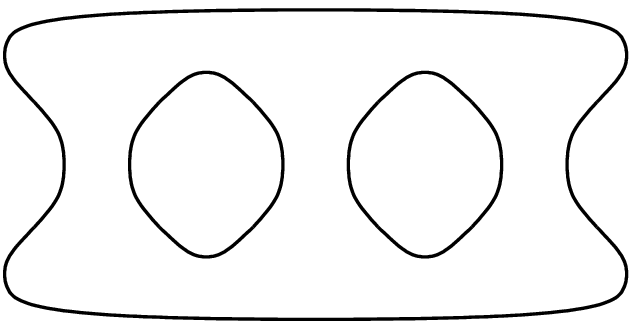}
            \put(5,5){\tiny $111$}
            \put(5,40){\tiny
              $1$}\put(30,30){\tiny $2$}\put(65,30){\tiny $3$}
          \end{overpic}}}\\
      \vcenter{\hbox{\begin{overpic}[height=\tempfigdim]{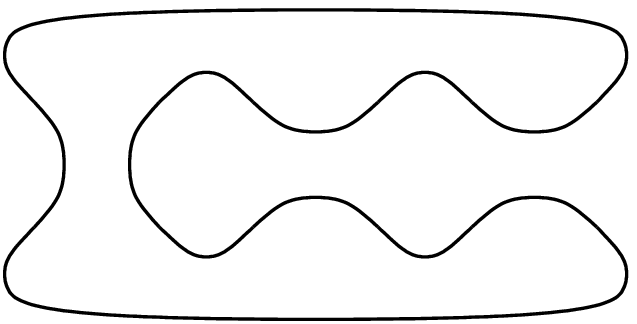}
            \put(5,5){\tiny $100$}
            \put(5,40){\tiny $1$}
          \end{overpic}}}\ar[r]\ar[urr]|(0.5)\hole&
      \vcenter{\hbox{\begin{overpic}[height=\tempfigdim]{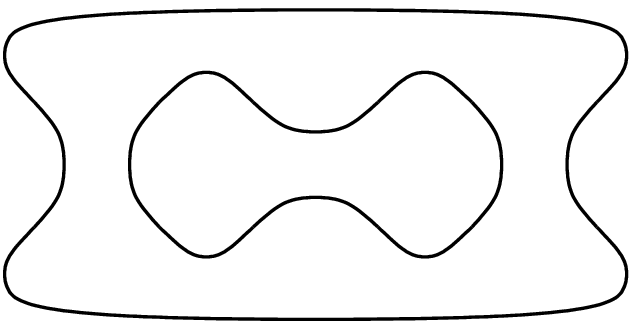}
            \put(5,5){\tiny $101$}
            \put(5,40){\tiny
              $1$}\put(30,30){\tiny $2$}
          \end{overpic}}}\ar[urr] 
    }
    \]
    \caption{A decorated link diagram for the trefoil and its cube of
      resolutions.}\label{fig:trefoil-decor}
  \end{figure}
  
  The complex $\OurCx$ is freely generated over $\F_2[H,W]$ by the
  thirty Khovanov generators. The differential $\diffOur$ is the
  following:
  {\small
  \begin{align*}
    &\diffOur(1^{000})=1^{100}+1^{010}+1^{001}+W(1^{110}+1^{101}+1^{011})+W^21^{111}\\
    &\diffOur(x^{000}_1x^{000}_2)=H\diffOur(x^{000}_1)=H\diffOur(x^{000}_2)=H(x_1^{100}+x_1^{010}+x_1^{001})\allowdisplaybreaks\\
    &\diffOur(1^{100})+x^{011}_1+x^{011}_2+H 1^{011}=\diffOur(1^{010})+x^{101}_1+x^{101}_2+H 1^{101}=\diffOur(1^{001})+x^{110}_1+x^{110}_2+H 1^{110}\\
    &\qquad{}=x^{110}_1+x^{110}_2+x^{101}_1+x^{101}_2+x^{011}_1+x^{011}_2+H(1^{110}+1^{101}+1^{011})+HW 1^{111}\allowdisplaybreaks\\
   &\diffOur(x_1^{100})+x^{011}_1x^{011}_2=\diffOur(x_1^{010})+x^{101}_1x^{101}_2=\diffOur(x_1^{001})+x^{110}_1x^{110}_2=x^{110}_1x^{110}_2+x^{101}_1x^{101}_2+x^{011}_1x^{011}_2\allowdisplaybreaks\\
   &\diffOur(1^{110})+x^{111}_2=\diffOur(1^{101})+x^{111}_1=\diffOur(1^{011})+x^{111}_3=x^{111}_1+x^{111}_2+x^{111}_3+H 1^{111}\allowdisplaybreaks\\
   &\diffOur(x^{110}_1)+x^{111}_1x^{111}_3=\diffOur(x^{101}_2)+x^{111}_2x^{111}_3=\diffOur(x^{011}_1)+x^{111}_1x^{111}_2=0\allowdisplaybreaks\\
   &\diffOur(x^{110}_2)+x^{111}_1x^{111}_3+Hx^{111}_2=\diffOur(x^{101}_1)+x^{111}_2x^{111}_3+Hx^{111}_1=\diffOur(x^{011}_2)+x^{111}_1x^{111}_2+Hx^{111}_3\\
   &\qquad{}=x^{111}_1x^{111}_2+x^{111}_1x^{111}_3+x^{111}_2x^{111}_3\\
   &\diffOur(x^{110}_1x^{110}_2)=\diffOur(x^{101}_1x^{101}_2)=\diffOur(x^{011}_1x^{011}_2)=x^{111}_1x^{111}_2x^{111}_3;
  \end{align*}
}
  and it is zero on the other (eight) Khovanov generators.
\end{example}

We will now spend the rest of the section working towards a proof of
\Proposition{is-differential}.  Towards that end, we construct a
decomposition of $\hh$, generalize the notion of trees and dual trees,
and prove some relevant lemmas.

\begin{definition}
For each non-empty subset $\varnothing\neq s\sbseq\Crossings$, define the following endomorphism
$\hh_s$ of $\Cx$.
\[
\langle\hh_s((u,x)),(v,y)\rangle=\begin{cases}
\langle\hh((u,x)),(v,y)\rangle&\text{ if $v=u\coprod s$,}\\
0&\text{ otherwise.}
\end{cases}
\]
It is clear that $\hh=\displaystyle\sum_{\varnothing\neq s\sbseq\Crossings}\hh_s$;
indeed, $\hh_i=\displaystyle\sum_{\substack{s \sbseq\Crossings\\\card{s}=i}}\hh_s$.
\end{definition}

\begin{definition}
  In a similar vein to \Definition{tree}, let a \emph{potted tree}
  denote a part of a labeled resolution configuration consisting of
  $k$ initial circles and $(k-1)$ arcs that are arranged as a tree
  (with $k\geq 1$), such that there is exactly one additional arc
  joining these $k$ circles to the rest of the resolution
  configuration; furthermore, all these $k$ circles appear in the
  starting monomial, and the unique ending circle they correspond to,
  appears in the ending monomial. A \emph{potted dual tree} is part of
  a labeled resolution configuration whose dual is a potted tree.
\end{definition}

\begin{lemma}\label{lem:h-behavior}
For any non-empty subsets $s,t$ of $\Crossings$,
\[
\hh_s\hh_t=
\begin{cases}
\hh_{s\coprod t}&\text{ if $s\bigcap t=\varnothing$,}\\
0&\text{ otherwise.}
\end{cases}
\]
\end{lemma}

\begin{proof}
  We want to show $\langle\hh_s\hh_t((u,x)),(v,y)\rangle=
  \langle\hh_{s\coprod t}((u,x)),(v,y)\rangle$ for all disjoint
  non-empty subsets $s,t$, and for all Khovanov generators $(u,x)$ and
  $(v,y)$. Either can be non-zero only if $v=u\coprod s\coprod t$, so
  we may assume $u$ is disjoint from $s$ and $t$, and $v=u\coprod
  s\coprod t$. Let $w=u\coprod t$.

  The contribution $\langle\hh_{s\coprod t}((u,x)),(v,y)\rangle$ is
  non-zero if and only if $(\ResConfig{u}{v},x,y)$ is a disjoint union
  of trees and dual trees. Since graph minors of forests are also
  forests, for a unique monomial $z$ on $Z(\AssRes{w})$, both
  $(\ResConfig{u}{w},x,z)$ and $(\ResConfig{w}{v},z,y)$ are disjoint
  unions of trees and dual trees. Therefore, $\hh_t((u,x))=(w,z)$ and
  $\hh_s((w,z))=(v,y)$, and hence
  $\langle\hh_s\hh_t((u,x)),(v,y)\rangle\neq 0$.

  Conversely, if $\langle\hh_s\hh_t((u,x)),(v,y)\rangle\neq 0$,
  let $(w,z)=\hh_t((u,x))$ (note, $\hh_t$ evaluated on a Khovanov
  generator can have at most one non-zero term). Then both
  $(\ResConfig{u}{w},x,z)$ and $(\ResConfig{w}{v},z,y)$ are disjoint
  unions of trees and dual trees. In other words, the resolution
  configuration $(\ResConfig{u}{v},x,y)$ is obtained from
  $(\ResConfig{u}{w},x,z)$ by adding a bunch of potted trees and
  potted dual trees. Therefore, $(\ResConfig{u}{v},x,y)$ is a disjoint
  union of trees and dual trees as well, and hence
  $\langle\hh_{s\coprod t}((u,x)),(v,y)\rangle\neq 0$.
\end{proof}

\begin{cor}\label{cor:h-commutes}
For any non-empty subsets $s,t$ of $\Crossings$,
\[
\Commute{\hh_s}{\hh_t}=0.
\]
Therefore, for any non-empty subset $s$ of $\Crossings$,
\(
\Commute{\hh_s}{\hh}=0;
\)
and
\(
\hh^2=0.
\)
\end{cor}

\begin{lemma}\label{lem:zoltan-homotopy}
  Let $D$ and $D'$ be two decorated link diagrams, which differ in the
  choice of decorations at precisely the crossings in some subset
  $\mf{S}\sbseq\Crossings$, but are otherwise identical. Let $\dd$ and
  $\dd'$ be the respective endomorphisms of $\Cx$ defined via the
  contribution function $\cont{d}$. Then,
  \[
  \dd'= \dd+\sum_{\varnothing\neq
    s\sbseq\mf{S}}\Commute{\hh_s}{\dd}+\sum_{\substack{\varnothing\neq
    s,t\sbseq\mf{S}\\s\cap t=\varnothing}}\hh_s\dd\hh_t=\dd+\sum_{\varnothing\neq
    s\sbseq\mf{S}}\Commute{\hh_s}{\dd}+\sum_{\varnothing\neq
    s,t\sbseq\mf{S}}\hh_s\dd\hh_t.
  \]
\end{lemma}

\begin{proof}
  The second and the third expressions agree because $\hh_s\dd\hh_t=0$
  if $s\bigcap t\neq\varnothing$. Therefore, it is enough to prove
  that the first and the second expressions agree. We do it by
  induction on the number of elements in $\mf{S}$.

The base case, when $\mf{S}$ consists of a single crossing, is
essentially \cite[Theorem~5.4]{Szab-kh-geometric}. So we only need the
induction step. Let $D''$ be the decorated link diagram obtained from
$D$ by changing the decoration at a single crossing $c\in\mf{S}$, and
let $\dd''$ be the corresponding endomorphism. From the base case,
\[
\dd''=\dd+\Commute{\hh_{\{c\}}}{\dd},
\]
and from the induction step,
\[
\dd'=\dd''+\sum_{\varnothing\neq
  s\sbseq\mf{S}\sm\{c\}}\Commute{\hh_s}{\dd''}+\sum_{\substack{\varnothing\neq
    s,t\sbseq\mf{S}\sm\{c\}\\s\cap t=\varnothing}}\hh_s\dd''\hh_t.
\]
Combining (and using \Lemma{h-behavior}), we get
\begin{align*}
  \dd'&=\dd+\Commute{\hh_{\{c\}}}{\dd}+\sum_{\varnothing\neq
    s\sbseq\mf{S}\sm\{c\}}\Commute{\hh_s}{\dd+\Commute{\hh_{\{c\}}}{\dd}}
  +\sum_{\substack{\varnothing\neq s,t\sbseq\mf{S}\sm\{c\}\\s\cap
      t=\varnothing}}\hh_s(\dd+\Commute{\hh_{\{c\}}}{\dd})\hh_t\\
  &=\dd+\Commute{\hh_{\{c\}}}{\dd}+\sum_{\varnothing\neq
    s\sbseq\mf{S}\sm\{c\}}(\Commute{\hh_s}{\dd}+\Commute{\hh_{s\coprod\{c\}}}{\dd})+\sum_{\varnothing\neq
    s\sbseq\mf{S}\sm\{c\}}(\hh_s\dd\hh_{\{c\}}+\hh_{\{c\}}\dd\hh_s)\\
  &\qquad{}+\sum_{\substack{\varnothing\neq s,t\sbseq\mf{S}\sm\{c\}\\s\cap
      t=\varnothing}}(\hh_s\dd\hh_t+\hh_{s\coprod\{c\}}\dd\hh_t+\hh_s\dd\hh_{t\coprod\{c\}})\\
  &=\dd+\sum_{\varnothing\neq
    s\sbseq\mf{S}}\Commute{\hh_s}{\dd}+\sum_{\substack{\varnothing\neq
      s,t\sbseq\mf{S}\\s\cap t=\varnothing}}\hh_s\dd\hh_t.\qedhere
\end{align*}
\end{proof}

\begin{lemma}\label{lem:dec-independence}
  Let $D$ and $D'$ be two decorated link diagrams, which might differ
  in the choice of decorations, but are otherwise identical. Let $\dd$
  and $\dd'$ be the respective endomorphisms of $\Cx$. Fix $k\geq
  1$. If
  \[
  \langle\Commute{\hh}{\dd}((u,x)),(v,y)\rangle=0
  \]
  for all pairs of Khovanov generators $(u,x)$ and $(v,y)$ with
  $\card{v}-\card{u}< k$, then
  \[
  \langle\Commute{\hh}{\dd}((u,x)),(v,y)\rangle=\langle\Commute{\hh}{\dd'}((u,x)),(v,y)\rangle
  \] 
  for all pairs of Khovanov generators $(u,x)$ and $(v,y)$ with
  $\card{v}-\card{u}= k$.
\end{lemma}

\begin{proof}
Let $\mf{S}$ be the set of crossings where the decorations of $D$ and
$D'$ differ. By \Lemma{zoltan-homotopy},
\[
\dd'-\dd=\sum_{\varnothing\neq
  s\sbseq\mf{S}}\Commute{\hh_s}{\dd}+\sum_{\varnothing\neq
  s,t\sbseq\mf{S}}\hh_s\dd\hh_t,
\]
and hence (with the aid of \Corollary{h-commutes}),
\begin{align*}
\Commute{\hh}{\dd'-\dd}&=\sum_{\varnothing\neq
  s\sbseq\mf{S}}\Commute{\hh}{\Commute{\hh_s}{\dd}}+\sum_{\varnothing\neq
  s,t\sbseq\mf{S}}\Commute{\hh}{\hh_s\dd\hh_t}\\
&=\sum_{\varnothing\neq
  s\sbseq\mf{S}}\Commute{\hh_s}{\Commute{\hh}{\dd}}+\sum_{\varnothing\neq
  s,t\sbseq\mf{S}}\hh_s\Commute{\hh}{\dd}\hh_t.
\end{align*}
Therefore, for any $(u,x)$ and $(v,y)$ with $\gr_h((v,y))-\gr_h((u,x))=\card{v}-\card{u}=k$,
\begin{align*}
  \langle\Commute{\hh}{\dd'-\dd}((u,x)),(v,y)\rangle&=\sum_{\varnothing\neq
    s\sbseq\mf{S}}\langle\Commute{\hh_s}{\Commute{\hh}{\dd}}((u,x)),(v,y)\rangle\\
  &\qquad{}+\sum_{\varnothing\neq
    s,t\sbseq\mf{S}}\langle\hh_s\Commute{\hh}{\dd}\hh_t((u,x)),(v,y)\rangle\\
  &=\sum_{\substack{s,(w,z)\\\varnothing\neq
    s\sbseq\mf{S}\\\mathclap{\card{w}-\card{u}=k-\card{s}}}}\langle\hh_s((w,z)),(v,y)\rangle\langle\Commute{\hh}{\dd}((u,x)),(w,z)\rangle\\
&\qquad{}+\sum_{\substack{s,(w,z)\\\varnothing\neq
    s\sbseq\mf{S}\\\mathclap{\card{v}-\card{w}=k-\card{s}}}}\langle\Commute{\hh}{\dd}((w,z)),(v,y)\rangle\langle\hh_s((u,x)),(w,z)\rangle\\
  &\mathllap{{}+\sum_{\substack{\mathclap{s,t,(w_1,z_1),(w_2,z_2)}\\\varnothing\neq
    s,t\sbseq\mf{S}\\\mathclap{\card{w_1}-\card{u}=\card{t}}\\\mathclap{\card{w_2}-\card{w_1}=k-\card{s}-\card{t}}}}}\langle\hh_s((w_2,z_2)),(v,y)\rangle\langle\Commute{\hh}{\dd}((w_1,z_1)),(w_2,z_2)\rangle\langle\hh_t((u,x)),(w_1,z_1)\rangle\\
  &=0
\end{align*}
by the hypothesis.
\end{proof}

Let us now define five families of labeled resolution configurations
that will be of relevance very soon.

\begin{definition}\label{def:combined-config}
  A Type-$\al$ configuration (\Figure{a} after ignoring the
  orientations) is obtained from a Type-A configuration by forgetting
  the orientation of the arcs, and adding some number (possibly zero)
  of potted dual trees on the two initial circles. A Type-$\be$
  configuration (\Figure{b} after ignoring the orientations) is the
  dual of a Type-$\al$ configuration; therefore, it is obtained from a
  Type-B configuration by forgetting the orientation of the arcs, and
  adding potted trees on the initial circles. A Type-$\ga$
  configuration (\Figure{c} after ignoring the orientations) is
  obtained from a Type-C configuration by forgetting the orientation
  of the arcs, and adding potted dual trees on the initial circle. A
  Type-$\de$ configuration (\Figure{d} after ignoring the
  orientations) is the dual of a Type-$\ga$ configuration; therefore,
  it is obtained from a Type-D configuration by forgetting the
  orientation of the arcs, and adding potted trees on the initial
  circles. A Type-$\varepsilon$ configuration (\Figure{e} after
  ignoring the orientations) is obtained by adding potted trees and
  potted dual trees on a single circle, and requiring that the index
  of the resulting resolution configuration be at least one.
\end{definition}

\captionsetup[subfloat]{width={0.4\textwidth}}
\begin{figure}
  \centering 
  \subfloat[Type-$\al$
  configuration.]{\label{fig:a}\begin{overpic}[width=0.4\textwidth]{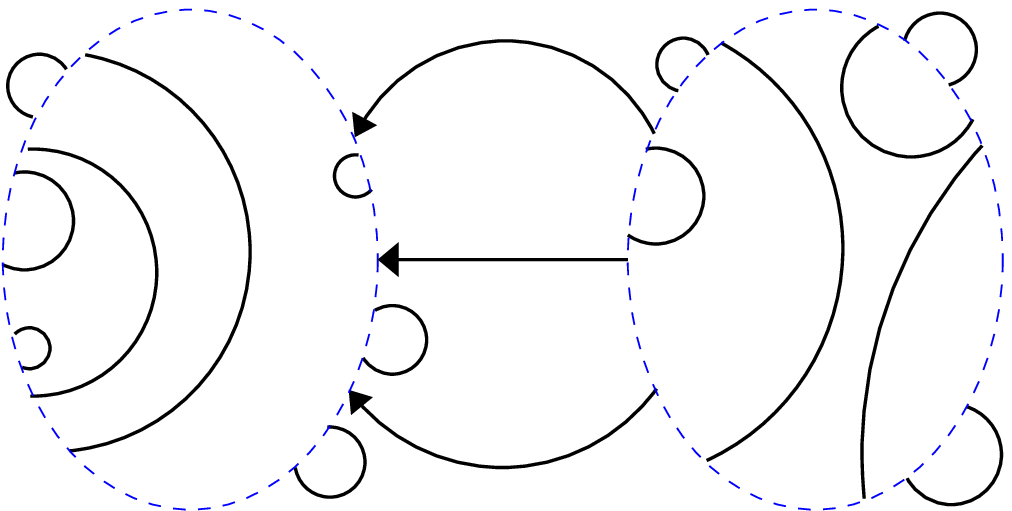}\put(19,1){\tiny
        $0$}\put(79,1){\tiny $0$}\end{overpic}}
  \hspace{0.1\textwidth}
  \subfloat[Type-$\be$
  configuration.]{\label{fig:b}\begin{overpic}[width=0.4\textwidth]{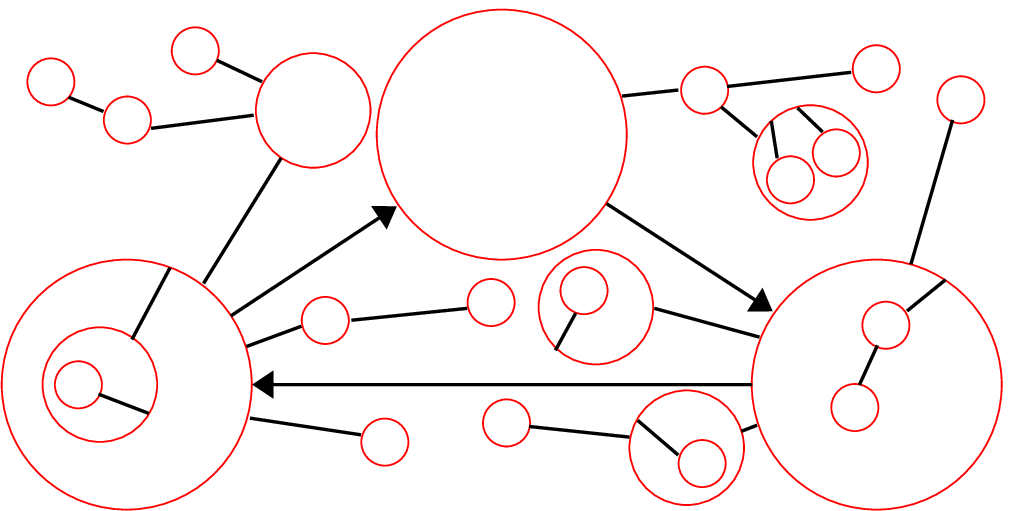}\put(5,5){\tiny
        $1$}\put(12,1){\tiny $1$}\put(6,11){\tiny $1$}\put(4,41){\tiny
        $1$}\put(12,38){\tiny $1$}\put(18,45){\tiny
        $1$}\put(32,35){\tiny $1$}\put(31,18){\tiny
        $1$}\put(37,6){\tiny $1$}\put(48,19){\tiny
        $1$}\put(50,8){\tiny $1$}\put(57,20){\tiny
        $1$}\put(58,15){\tiny $1$}\put(69,41){\tiny
        $1$}\put(69,3){\tiny $1$}\put(63,3){\tiny
        $1$}\put(73,34){\tiny $1$}\put(78,32){\tiny
        $1$}\put(82,34){\tiny $1$}\put(87,43){\tiny
        $1$}\put(84,9){\tiny $1$}\put(87,17){\tiny
        $1$}\put(95,40){\tiny $1$} \put(88,1){\tiny
        $1$}\put(49,27){\tiny $1$}\end{overpic}}
  \hspace{0.1\textwidth}
  \subfloat[Type-$\ga$
  configuration.]{\label{fig:c}\begin{overpic}[width=0.4\textwidth]{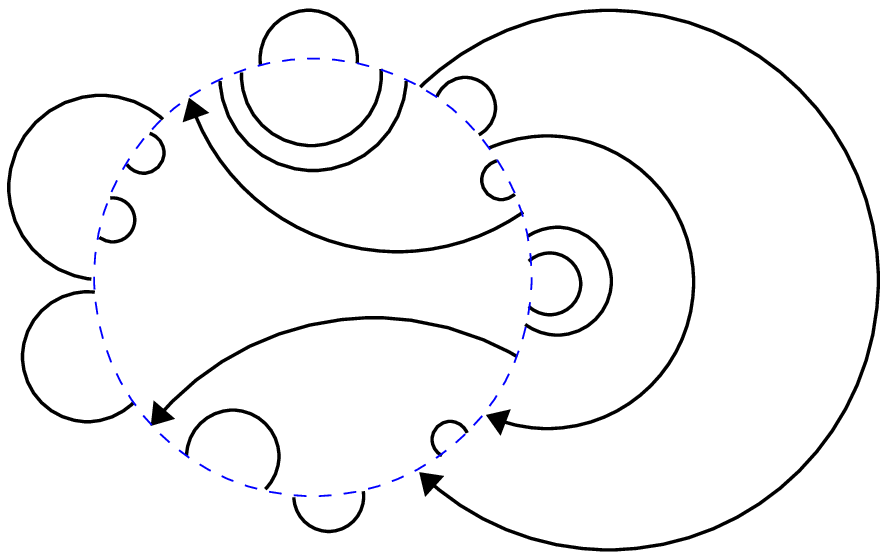}\put(29,8){\tiny
        $0$}\end{overpic}}
  \hspace{0.1\textwidth} 
  \subfloat[Type-$\de$
  configuration.]{\label{fig:d}\begin{overpic}[width=0.4\textwidth]{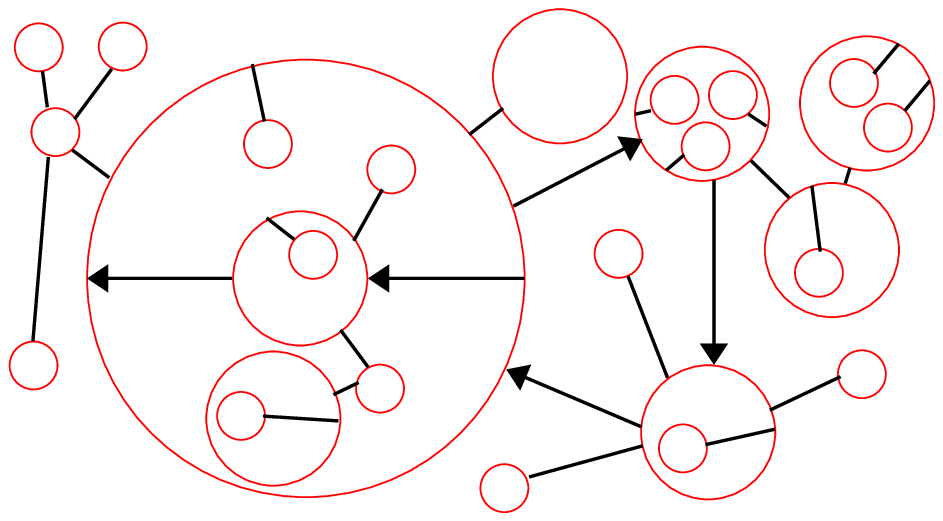}\put(4,18){\tiny
        $1$}\put(6,41){\tiny $1$}\put(4,50){\tiny
        $1$}\put(12,50){\tiny $1$}\put(15,15){\tiny
        $1$}\put(27,40){\tiny $1$} \put(23,13){\tiny $1$}
      \put(26,8){\tiny $1$} \put(26,23){\tiny $1$} \put(31,29){\tiny
        $1$} \put(38,16){\tiny $1$} \put(38,38){\tiny $1$}
      \put(50,6){\tiny $1$} \put(54,42){\tiny $1$} \put(62,29){\tiny
        $1$} \put(68,10){\tiny $1$} \put(73,16){\tiny $1$}
      \put(86,17){\tiny $1$} \put(81,27){\tiny $1$} \put(88,29){\tiny
        $1$} \put(81,44){\tiny $1$} \put(88,42){\tiny $1$}
      \put(85,46){\tiny $1$} \put(70,48){\tiny $1$} \put(70,40){\tiny
        $1$} \put(67,45){\tiny $1$} \put(72,45){\tiny $1$}
\end{overpic}}
  \hspace{0.1\textwidth}
  \subfloat[Type-$\varepsilon$
  configuration.]{\label{fig:e}\begin{overpic}[width=0.4\textwidth]{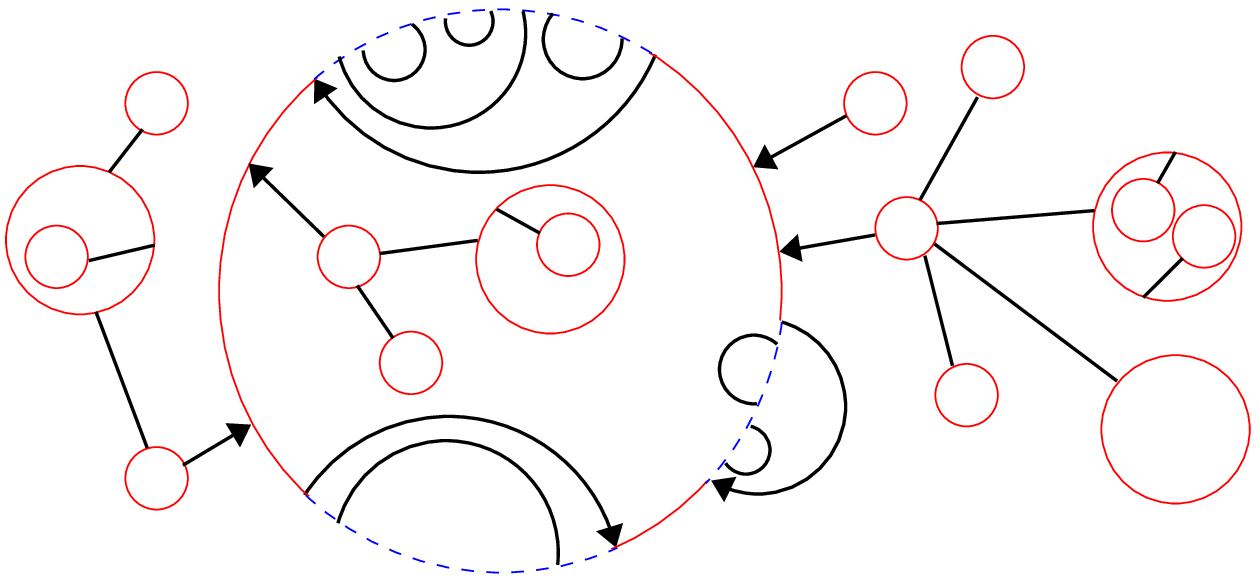}
      \put(39,1){\tiny $0$} \put(11,7){\tiny $1$} \put(3,24){\tiny
        $1$} \put(5,29){\tiny $1$} \put(11,36){\tiny $1$}
      \put(27,24){\tiny $1$} \put(32,16){\tiny $1$} \put(43,20){\tiny
        $1$} \put(44,25){\tiny $1$} \put(69,36){\tiny $1$}
      \put(71,27){\tiny $1$} \put(76,13){\tiny $1$} \put(78,39){\tiny
        $1$} \put(93,6){\tiny $1$} \put(95,34){\tiny $1$}
      \put(90,28){\tiny $1$} \put(95,26){\tiny $1$}
    \end{overpic}}
  \caption{After ignoring the orientations of the arcs, these are the
    labeled resolution configurations from
    \Definition{combined-config}.}\label{fig:combined-configurations}
\end{figure}

\begin{lemma}\label{lem:main-enumeration}
  If $(u,x)$, $(v,y)$ and $(w,z)$ are Khovanov generators such that
  $\langle\dd((u,x)),(w,z)\rangle\allowbreak\neq 0$ and
  $\langle\hh((w,z)),(v,y)\rangle\neq 0$, then,
  $(\ResConfig{u}{v},x,y)$, after forgetting the orientations of the
  arcs, is equivalent to a disjoint union of trees, dual trees and
  exactly one configuration from one of the five families in
  \Definition{combined-config}.
\end{lemma}

\begin{proof}
  The proof is a fairly straightforward case analysis. Since
  $\langle\dd((u,x)),(w,z)\rangle\allowbreak\neq 0$, the active part
  of the resolution configuration $(\ResConfig{u}{w},x,z)$ must be in
  one of the five families described in \Section{construction} or
  \Figure{zoltan-configurations}, and since
  $\langle\hh((w,z)),(v,y)\rangle\neq 0$, the resolution configuration
  $(\ResConfig{w}{v},z,y)$ must be a disjoint union of trees and dual
  trees. Therefore, we need to show that the union of the following is
  a configuration from one of the five families in
  \Definition{combined-config}:
  \begin{enumerate}[label=(\roman*)]
  \item the active part of $(\ResConfig{u}{w},x,z)$;
  \item the arcs of $\ResConfig{w}{v}$ that are in the same connected
    component (of $\ResConfig{w}{v}$) as some active ending circle of
    $\ResConfig{u}{w}$; and
  \item the starting circles of $(\ResConfig{w}{v},z,y)$ that are
    passive as ending circles of $\ResConfig{u}{w}$ but are in the
    same connected component (of $\ResConfig{w}{v}$) as some active
    ending circle of $\ResConfig{u}{w}$.
  \end{enumerate}

  If the active part of $(\ResConfig{u}{w},x,z)$ is a Type-A
  configuration, then adding the relevant portion from $(\ResConfig{w}{v},z,y)$
  has the effect of adding some (possibly zero) more parallel arcs
  between the two active starting circles of $\ResConfig{u}{w}$,
  followed by adding some potted dual trees; consequently, we get a
  Type-$\al$ configuration.

  If the active part of $(\ResConfig{u}{w},x,z)$ is a Type-B
  configuration, then adding the relevant portion from
  $(\ResConfig{w}{v},z,y)$ could achieve one of the following two
  things. If the new arcs and circles do not connect the two active
  ending circles of $\ResConfig{u}{w}$, then we get a Type-$\be$
  configuration; and if the new arcs and circles do connect the two
  active ending circles of $\ResConfig{u}{w}$, then we get a
  Type-$\de$ configuration.

  If the active part of $(\ResConfig{u}{w},x,z)$ is a Type-C
  configuration, then adding the relevant portion from $(\ResConfig{w}{v},z,y)$
  has the effect of adding some (possibly zero) number of arcs
  parallel to inside or outside arcs of $\ResConfig{u}{w}$, followed
  by adding some potted dual trees; therefore, we get a Type-$\ga$
  configuration.

  If the active part of $(\ResConfig{u}{w},x,z)$ is a Type-D
  configuration, then adding the relevant portion from $(\ResConfig{w}{v},z,y)$
  is simply the addition of some potted trees, and we get a
  Type-$\de$ configuration.

  Finally, if the active part of $(\ResConfig{u}{w},x,z)$ is a Type-E
  configuration, then adding the relevant portion from $(\ResConfig{w}{v},z,y)$
  is same as adding some potted trees and potted dual trees, and
  therefore, we get a Type-$\varepsilon$ configuration.
\end{proof}

\begin{lemma}\label{lem:subs-enumeration}
  If $(u,x)$, $(v,y)$ and $(w,z)$ are Khovanov generators such that
  $\langle\hh((u,x)),(w,z)\rangle\allowbreak\neq 0$ and
  $\langle\dd((w,z)),(v,y)\rangle\neq 0$, then,
  $(\ResConfig{u}{v},x,y)$, after forgetting the orientations of the
  arcs, is equivalent to a disjoint union of trees, dual trees and
  exactly one configuration from one of the five families in
  \Definition{combined-config}.
\end{lemma}

\begin{proof}
  This follows immediately from \Lemma{main-enumeration}. Since both
  the contribution functions $\cont{d}$ and $\cont{h}$ satisfy the
  duality rule, the resolution configuration $(\ResConfig{u}{v},x,y)$
  must be dual to some resolution configuration from
  \Lemma{main-enumeration}. Therefore, it must be dual to a disjoint
  union of trees, dual trees, and exactly one configuration from
  \Definition{combined-config}. However, the resolution configurations
  from \Definition{combined-config} are closed under taking duals;
  namely, Type-$\al$ configurations are dual to Type-$\be$
  configurations, Type-$\ga$ configurations are dual to Type-$\de$
  configurations, and Type-$\varepsilon$ configurations are self-dual.
\end{proof}

We are finally ready to embark upon the proof of our main result for
this section.
\begin{proof}[Proof of \Proposition{is-differential}]
We want to show $(\dd+\hh)^2=0$ on $\Cx$. Since $\dd$ increases
$\gr_{\diff}$ by one, and $\hh$ preserves it, this is equivalent to
the following three statements, arranged in the increasing order of
difficulty:
\[
\hh^2=0\qquad\qquad\Commute{\hh}{\dd}=0\qquad\qquad\dd^2=0.
\]
\Corollary{h-commutes} already takes care of the easiest of these
three statements; and the hardest is proved in
\cite[Section~6]{Szab-kh-geometric}. Therefore, we will merely prove
$\Commute{\hh}{\dd}=0$. Certain formulas for certain counts take a
simpler form if we write them for the equivalent statement
\[
\Commute{\hh+\Id}{\dd}=0,
\]
and so that is what we will prove.

We prove the above equation by induction on the index. The base case
when the index is zero, is vacuous; hence we only prove the induction
step. We assume
\[
\langle\Commute{\hh+\Id}{\dd}((u,x)),(v,y)\rangle=0
\]
for all pairs of Khovanov generators $(u,x)$ and $(v,y)$ with
$\gr_h((v,y))-\gr_h((u,x))<k$, and for any decoration of the original
link diagram $D$.  Then fix Khovanov generators $(u,x)$ and $(v,y)$
with $u\sbs v$ and $\gr_h((v,y))-\gr_h((u,x))=k$, and fix some
decoration on $D$. We will prove that 
\[
\langle\Commute{\hh+\Id}{\dd}((u,x)),(v,y)\rangle=0
\]
holds for that decoration.

Thanks to \Lemma{dec-independence}, it is enough to prove this for
some decoration; that is, after fixing $(u,x)$ and $(v,y)$, we are
free to choose our decoration. By \Lemma{main-enumeration} and
\Lemma{subs-enumeration}, we may assume that $(\ResConfig{u}{v},x,y)$
is a disjoint union of trees, dual trees, and exactly one
configuration from \Definition{combined-config}. These configurations
are shown in \Figure{combined-configurations}, with some of the arcs
being oriented; assume the number of oriented arcs is $\ell$. Choose
some decoration on $D$ so that those $\ell$ arcs are oriented as in
\Figure{combined-configurations}.  We will finish the proof by
analyzing each of the configurations from
\Figure{combined-configurations}.

If $(\ResConfig{u}{v},x,y)$ is equivalent to a disjoint union of
trees, dual trees and the configuration from \Figure{a}, then it
breaks up in the following ways:
\begin{itemize}
\item We first use all the arcs except the $\ell$ oriented ones to get
  a configuration that contributes to $(\hh+\Id)$; and we follow it by
  the $\ell$ oriented arcs to get a Type-A configuration which
  contributes to $\dd$.
\item We use a non-empty subset of the $\ell$ oriented arcs to get a
  Type-A configuration that contributes to $\dd$; and we follow it by
  the rest of the arcs to get a configuration that contributes to
  $(\hh+\Id)$.
\end{itemize}
Therefore, we get a total of $1+(2^{\ell}-1)$ contributions to
$\Commute{\hh}{\dd}$, and hence are done with this case.

If $(\ResConfig{u}{v},x,y)$ is equivalent to a disjoint union of
trees, dual trees and the configuration from \Figure{b}, then it is
dual to the previous case, and therefore follows from the previous
case since all the relevant contribution functions satisfy the duality
rule.

Next assume that $(\ResConfig{u}{v},x,y)$ is equivalent to a disjoint
union of trees, dual trees and the configuration from \Figure{c};
assume the number of oriented arcs inside (respectively, outside) the
starting circle is $\ell_1$ (respectively, $\ell_2$) with
$\ell=\ell_1+\ell_2$. The configuration breaks up in the following
ways:
\begin{itemize}
\item We first use all except the $\ell_1$ oriented arcs inside to get
  a configuration that contributes to $(\hh+\Id)$; and we follow it by
  the $\ell_1$ oriented arcs to get a Type-A configuration that
  contributes to $\dd$.
\item We first use all except the $\ell_2$ oriented arcs outside to get
  a configuration that contributes to $(\hh+\Id)$; and we follow it by
  the $\ell_2$ oriented arcs to get a Type-A configuration that
  contributes to $\dd$.
\end{itemize}
Note, we have chosen the orientations of the $\ell$ arcs in such a way
that we do not have to encounter Type-C configurations. Therefore, we
get a total of $1+1$ contributions to $\Commute{\hh}{\dd}$, and once
again we are done.

If $(\ResConfig{u}{v},x,y)$ is equivalent to a disjoint union of
trees, dual trees and the configuration from \Figure{d}, then it is
dual to the previous case, and therefore follows from it. 

Finally assume that $(\ResConfig{u}{v},x,y)$ is equivalent to a
disjoint union of trees, dual trees and the configuration from
\Figure{e}. Assume the number of potted trees on the special starting
circle is $\ell_1$, and the number of potted dual trees on the special
starting circle is $\ell_2$, with $\ell_1+\ell_2=\ell\geq 1$. 
There are the following subcases.
\begin{enumerate}
\item $\ell_1\neq 0$, $\ell_2=0$. The configuration breaks up in the
  following ways:
  \begin{itemize}
  \item We first use all except the $\ell_1$ oriented arcs to get a
    configuration that contributes to $(\hh+\Id)$; and we follow it by
    the $\ell_1$ oriented arcs to get a Type-E configuration that
    contributes to $\dd$.
  \item We use a non-empty subset of the $\ell_1$ oriented arcs to get a
    Type-E configuration that contributes to $\dd$; and we follow it by
    the rest of the arcs to get a configuration that contributes to
    $(\hh+\Id)$.
  \end{itemize}
\item $\ell_1=0$, $\ell_2\neq 0$. The configuration breaks up in the
  following ways:
  \begin{itemize}
  \item We first use all the $\ell_2$ oriented arcs to get a Type-E
    configuration that contributes to $\dd$; and we follow it by the
    rest of the arcs to get a configuration that contributes to
    $(\hh+\Id)$.
  \item We use all the non-oriented arcs and a proper subset of the
    $\ell_2$ oriented arcs to get a configuration that contributes to
    $(\hh+\Id)$; and we follow it by the remaining arcs to get a
    Type-E configuration that contributes to $\dd$.
  \end{itemize}
\item $\ell_1\neq 0$, $\ell_2\neq 0$. The configuration breaks up in the
  following ways:
  \begin{itemize}
  \item We use all the non-oriented arcs and all of the $\ell_2$
    oriented arcs to get a configuration that contributes to
    $(\hh+\Id)$; and we follow it by the rest of the arcs to get a
    Type-E configuration that contributes to $\dd$.
  \item We use all the $\ell_2$ oriented arcs to get a Type-E
    configuration that contributes to $\dd$; and we follow it by the
    rest of the arcs to get a configuration that contributes to
    $(\hh+\Id)$.
  \end{itemize}
\end{enumerate}
Therefore, we have the following number of contributions to $\Commute{\hh}{\dd}$,
\[
\begin{tabular}{l|cc}
&$\ell_1=0$&$\ell_1\geq 1$\\ \hline
$\ell_2=0$&&$1+(2^{\ell_1}-1)$\\
$\ell_2\geq 1$&$1+(2^{\ell_2}-1)$&$2$
\end{tabular}
\]
thus concluding the proof.
\end{proof}

\section{Invariance}\label{sec:invariance}

We now turn to invariance. The master invariant that we study is the
bigraded chain homotopy type of the inclusion
$\iota\from\OurCx^{-}\{1\}\into \OurCx$ over the ring $\F_2[H,W]$. All the
other variants that we have considered can easily be recovered from
this variant. For example, if one wishes to recover the plus version
of the filtered \Szabo{} geometric chain complex, one can obtain it by
taking the mapping cone of the inclusion $\OurCx^{-}\{1\}\into \OurCx$,
then setting $H=0$ and $W=1$, and then shifting the quantum grading by $1$.
\[
\SzCx^+\simeq \Cone(\iota)/\{H=0,W=1\}\{1\}.
\]

\begin{definition}
  Consider the bigraded polynomial ring $\F_2[H,W]$ with $H$ and $W$
  in bigradings $(0,-2)$ and $(-1,-2)$ respectively. Let
  $\mathrm{Kom}(\F_2[H,W])$ be the \emph{category of chain complexes}
  over $\F_2[H,W]$. The objects are bigraded chain complexes over
  $\F_2[H,W]$ with the differentials in bigrading $(1,0)$, and the
  morphisms are $(0,0)$-graded $\F_2[H,W]$-module chain maps. For
  $C\in\Ob(\mathrm{Kom}(\F_2[H,W])$ and integer $a$, define
  $C\{a\}\in\Ob(\mathrm{Kom}(\F_2[H,W])$ by shifting the second
  grading by $a$.

  Let $\hocat(\F_2[H,W])$ be the \emph{homotopy category}
  over $\F_2[H,W]$. The objects are same as the objects of
  $\mathrm{Kom}(\F_2[H,W])$. The morphisms are equivalence classes of
  morphisms in $\mathrm{Kom}(\F_2[H,W])$; we declare
  $f_1,f_2\in\Hom_{\mathrm{Kom}(\F_2[H,W])}(A,B)$ equivalent if there is a
  $(-1,0)$-graded $\F_2[H,W]$-module map $h\from A\to B$ such that
  \[
  f_1-f_2=d_B\circ h+h\circ d_A.
  \]

  Let $\hocatpair(\F_2[H,W])$ be the
  \emph{homotopy category of pairs} over $\F_2[H,W]$, defined as
  follows. The objects are triples $(A,B,\phi)$ where
  $A,B\in\Ob_{\mathrm{Kom}(\F_2[H,W])}$, and
  $\phi\in\Hom_{\mathrm{Kom}(\F_2[H,W])}(A,B)$. Morphisms from
  $(A_1,B_1,\phi_1)$ to $(A_2,B_2,\phi_2)$ are pairs $(f,g)$, where
  $f\in\Hom_{\mathrm{Kom}(\F_2[H,W])}(A_1,A_2)$ and
  $g\in\Hom_{\mathrm{Kom}(\F_2[H,W])}(B_1,B_2)$ 
  such that the following commutes 
  \[
  \xymatrix{
    A_1\ar[r]^{\phi_1}\ar[d]_f&B_1\ar[d]^g\\
    A_2\ar[r]^{\phi_2}&B_2,
  }
  \]
  modulo the following equivalence relation. We declare $(f_1,g_1)$
  equivalent to $(f_2,g_2)$ if there are $(-1,0)$-graded
  $\F_2[H,W]$-module maps $h\from A_1\to A_2$ and $k\from B_1\to B_2$
  such that the following commutes
  \[
  \xymatrix{
    A_1\ar[r]^{\phi_1}\ar[d]_h&B_1\ar[d]^k\\
    A_2\ar[r]^{\phi_2}&B_2,
  }
  \]
  and
  \begin{align*}
    f_1-f_2&=d_{A_2}\circ h+h\circ d_{A_1}\\
    g_1-g_2&=d_{B_2}\circ h+h\circ d_{B_1}.
  \end{align*}
\end{definition}

\begin{proposition}\label{prop:main-invariance}
  Let $(D,p)$ and $(D',p')$ be two pointed decorated link diagrams
  representing isotopic pointed links in $S^3$.  Then the objects
  $(\tuplethree{\OurCx^-(D,p)\{1\}}{\OurCx(D)}{\iota(D,p)})$ and
  $(\tuplethree{\OurCx^-(D',p')\{1\}}{\OurCx(D')}{\iota(D',p')})$ are
  isomorphic in the homotopy category of pairs
  $\hocatpair(\F_2[H,W])$.
\end{proposition}

\begin{corollary}\label{cor:easy-invariance}
  Let $D$ and $D'$ be two decorated link diagrams representing
  isotopic links in $S^3$.  Then the objects $\OurCx(D)$ and
  $\OurCx(D')$ are isomorphic in the homotopy category
  $\hocat(\F_2[H,W])$.
\end{corollary}

We will basically check invariance under the three Reidemeister moves
from \Figure{Reid-moves}, following the standard arguments. Along the
way, we will need the following well-known (and heavily used)
cancellation principle.

\begin{definition}\label{def:cancellation-data}
  Fix a link diagram $D$. A $5$-tuple $(\Crossings_0,u,v,c,a)$ is
  called a \emph{cancellation data} if $u\subsetneq\{c\}\cup
  u=v\subseteq\Crossings_0\subseteq\Crossings$, and one of the
  following holds.
  \begin{enumerate}
  \item\label{item:cancel-1} The surgery arc $\al_c$ joins two different circles in
    $\AssRes{u}$, the complete resolution of $D$ corresponding to $u$,
    and $a\in Z(\AssRes{u})$ is one of the two circles connected by
    $\al_c$, and $a$ is disjoint from $\al_{c'}$ for all
    $c'\in\Crossings\setminus\Crossings_0$. In this case, for any
    $w\subseteq\Crossings\setminus\Crossings_0$, there is a natural
    bijection between $Z(\AssRes{u\cup w})\sm\{a\}$ and
    $Z(\AssRes{v\cup w})$, and we call Khovanov generators $(u\cup
    w,x)$ and $(v\cup w,y)$ to be a \emph{canceling pair} if $a$ does
    not appear in the monomial $x$, and the monomials $x$ and $y$ are
    related by the above bijection. 
  \item\label{item:cancel-2} The arc $\al_c$ has both its endpoints on
    the same circle in $\AssRes{u}$, and $a\in Z(\AssRes{v})$ is one
    of the two circles obtained by surgering that circle along
    $\al_c$, and $a$ is disjoint from $\al_{c'}$ for all
    $c'\in\Crossings\setminus\Crossings_0$. In this case, for any
    $w\subseteq\Crossings\setminus\Crossings_0$, there is a natural
    bijection between $Z(\AssRes{u\cup w})$ and $Z(\AssRes{v\cup
      w})\sm\{a\}$, and we call Khovanov generators $(u\cup w,x)$ and
    $(v\cup w,y)$ to be a \emph{canceling pair} if $a$ appears in the
    monomial $y$, and the monomials $x$ and $a^{-1}y$ are related by
    the above bijection. (This is the dual of the above case.)
  \end{enumerate}
\end{definition}

\begin{lemma}\label{lem:cancel-prelim}
  Let $(\Crossings_0,u,v,c,a)$ be a cancellation data. For any
  $w,w'\subseteq\Crossings\setminus\Crossings_0$, and any Khovanov
  generators $(u\cup w,x)$ and $(v\cup w',y)$,
  \[
  \langle
  \diffOur((u\cup w,x)),(v\cup w',y)\rangle=\begin{cases}
    1&\text{if $w=w'$ and $(u\cup w,x)$ and $(v\cup w,y)$ is a canceling
      pair,}\\
    0&\text{otherwise.}
  \end{cases}
  \]
  In particular, if $S$ is the subset of the Khovanov generators
  consisting of all the canceling pairs for $(\Crossings_0,u,v,c,a)$,
  then the subquotient complex spanned by $S$ is isomorphic to the
  trivial object in $\hocat(\F_2[H,W])$.
\end{lemma}

\begin{proof}
  Let us assume the cancellation data $(\Crossings_0,u,v,c,a)$
  corresponds to Case~(\ref{item:cancel-1}) of
  \Definition{cancellation-data}. Case~(\ref{item:cancel-2}), being
  the dual, should follow. 

  The coefficient $\langle \diffOur((u\cup w,x)),(v\cup w',y)\rangle$
  can be non-zero only if $w\subseteq w'$, and the resolution
  configuration $(\ResConfig{u\cup w}{v\cup w'},x,y)$ either has a
  non-zero contribution in $\cont{d}$ or a non-zero contribution in
  $\cont{h}$. To have a non-zero contribution in $\cont{d}$, $x$ and
  $y$ must agree on the passive circles, and the active part of
  $(\ResConfig{u\cup w}{v\cup w'},x,y)$ must be equivalent to one of
  the five families described in \Definition{szabo-differential}. The
  circle $a$ is one of the active starting circles, has only one arc
  incident to it, and does not appear in the starting monomial $x$. A
  quick glance at \Figure{zoltan-configurations} implies the active
  part of the configuration must be an index-$1$ Type-A or Type-E
  configuration. This occurs when $w=w'$, and $(u\cup w,x)$ and
  $(v\cup w,y)$ form a canceling pair. These are precisely the
  configurations that contribute to the Khovanov differential $\dd_1$;
  see also \cite[Proof of Theorem~7.2]{Szab-kh-geometric}.

  In order to have a non-zero contribution in $\cont{h}$,
  $(\ResConfig{u\cup w}{v\cup w'},x,y)$ must be equivalent to disjoint
  union of trees and dual trees, as in
  \Definition{our-contribution-function}. However, since the arc
  $\al_c$ joins the circle $a$ to another circle, and $a$ does not
  appear in the starting monomial $x$, it cannot be a part of either a
  tree or a dual tree. Therefore, $(\ResConfig{u\cup w}{v\cup
    w'},x,y)$ can never contribute to $\cont{h}$. Consequently, the
  only contributions come from $\cont{d}$, and as analyzed earlier,
  they are exactly of the form as described in the lemma.
\end{proof}

Now we are almost set to prove invariance under Reidemeister
moves. The standard strategy is to delete certain canceling pairs to
simplify the chain complex, at the cost of adding new \emph{zigzag
  differentials}. For an example of how these zigzag differentials
work, assume we have a chain complex freely generated over some ring
$R$ with four generators $a_1,a_2,b_1,b_2$, and the differential is
\[
\diff a_1= b_1c_{11}+b_2c_{12}\qquad \diff a_2 =b_1c_{21}+b_2c_{22},
\]
and assume $c_{11}$ is a unit. Then we may perform a change of basis
\[
a'_1=a_1\qquad a'_2=a_2-a_1c_{11}^{-1}c_{21}\qquad b'_1=b_1+b_2c_{12}c_{11}^{-1}\qquad b'_2=b_2,
\] 
and then the chain complex breaks up into two direct summands,
generated by $\{a'_1,b'_1\}$ and $\{a'_2,b'_2\}$. The former is
acyclic, so we may cancel it, and then we are left with a complex with
just two generators $a'_2$ and $b'_2$ with differential $\diff
a'_2=b'_2(c_{22}-c_{12}c^{-1}_{11}c_{21})$. This operation may be
viewed as simply canceling the arrow $a_1\to b_1$ in the original
chain complex and adding an extra zigzag arrow $a_2\to b_2$ with
coefficient $(-c_{12}c^{-1}_{11}c_{21})$.
\begin{equation}\label{eq:zigzag}
\left(\vcenter{\hbox{\begin{tikzpicture}
      \node (a1) at (-1,1) {$a_1$};
      \node (a2) at (1,1) {$a_2$};
      \node (b1) at (-1,-1) {$b_1$};
      \node (b2) at (1,-1) {$b_2$};
      
      \draw[->] (a1) -- (b1) node[midway,fill=white] {\small $c_{11}$};
      \draw[->] (a1) -- (b2) node[pos=0.3,fill=white] {\small $c_{12}$};
      \draw[->] (a2) -- (b1) node[pos=0.3,fill=white] {\small $c_{21}$};
      \draw[->] (a2) -- (b2) node[midway,fill=white] {\small $c_{22}$};
    \end{tikzpicture}}}\right)
\cong
\left(\vcenter{\hbox{\begin{tikzpicture}
      \node (a1) at (-1,1) {$a_1$};
      \node (a2) at (1,1) {$a_2-a_1c_{11}^{-1}c_{21}$};
      \node (b1) at (-1,-1) {$b_1+b_2c_{12}c_{11}^{-1}$};
      \node (b2) at (1,-1) {$b_2$};
      
      \draw[->] (a1) -- (b1) node[midway,fill=white] {\small $c_{11}$};
      \draw[->] (a2) -- (b2) node[midway,fill=white] {\small $c_{22}-c_{12}c_{11}^{-1}c_{21}$};
    \end{tikzpicture}}}\right)
=
\left(\vcenter{\hbox{\begin{tikzpicture}
      \node (a1) at (-1,1) {$a'_1$};
      \node (b1) at (-1,-1) {$b'_1$};
      
      \draw[->] (a1) -- (b1) node[midway,fill=white] {\small $c_{11}$};
    \end{tikzpicture}}}\right)
\oplus
\left(\vcenter{\hbox{\begin{tikzpicture}
      \node (a1) at (-1,1) {};
      \node (a2) at (1,1) {$a'_2$};
      \node (b1) at (-1,-1) {};
      \node (b2) at (1,-1) {$b'_2$};
      
      \draw[->] (a2) -- (b2) node[midway,fill=white] {\small $c_{22}$};
      \draw[->,rounded corners] (a2) -- (-1,-1) -- 
      node[midway,fill=white,inner sep=0pt,outer sep=0pt] {\small $-c_{12}c_{11}^{-1}c_{21}$} (-1,1) -- (b2) ;
    \end{tikzpicture}}}\right)
\end{equation}

\setlength\tempfigdim{0.25\textheight}
\captionsetup[subfloat]{width=0.2\textwidth}
\begin{figure}
  \centering
  \subfloat[RI.]{\label{fig:R-I}
    \xymatrix@C=0.1\textwidth{
      \vcenter{\hbox{\includegraphics[height=\tempfigdim]{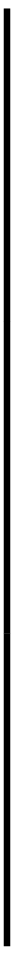}}}\ar@{->}[r]& \vcenter{\hbox{\includegraphics[height=\tempfigdim]{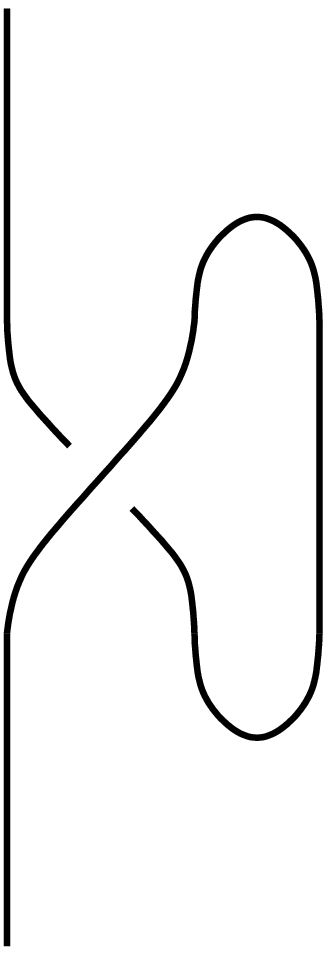}}}
    }
  }
  \hspace{0.03\textwidth}
  \subfloat[RII.]{\label{fig:R-II}
    \xymatrix@C=0.1\textwidth{
      \vcenter{\hbox{\includegraphics[height=\tempfigdim]{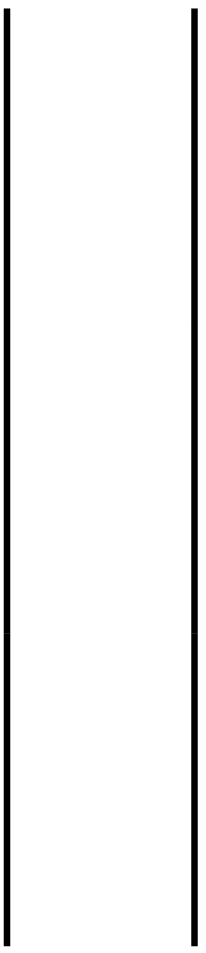}}}\ar@{->}[r]& \vcenter{\hbox{\includegraphics[height=\tempfigdim]{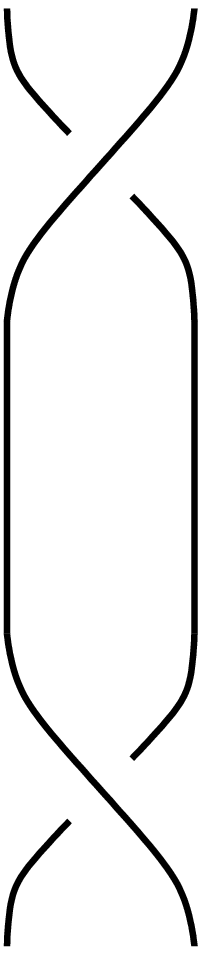}}}
    }
  }
  \hspace{0.03\textwidth}
  \subfloat[RIII.]{\label{fig:R-III}
    \xymatrix@C=0.1\textwidth{
      \vcenter{\hbox{\includegraphics[height=\tempfigdim]{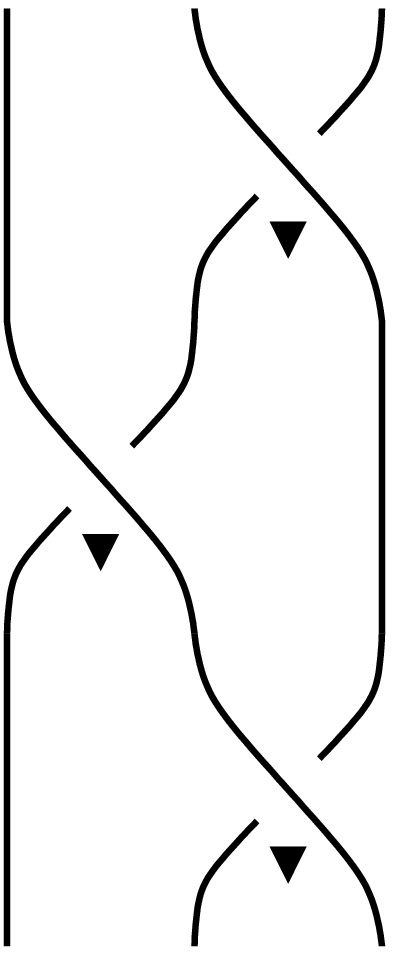}}}\ar@{->}[r]& \vcenter{\hbox{\includegraphics[height=\tempfigdim]{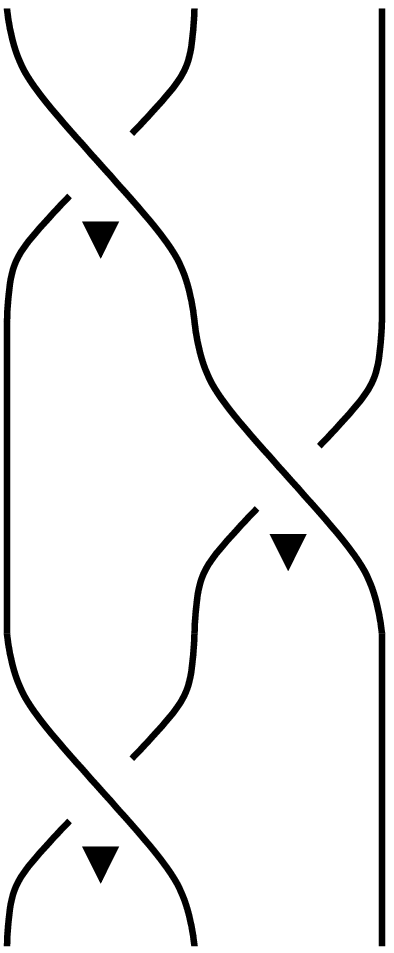}}}
    }
  }
  \caption{The three Reidemeister moves. The orientations of the
    strands are arbitrary. For RI and RII, the decorations at the new
    crossings are also arbitrary; for RIII, the decorations are as
    shown.}\label{fig:Reid-moves}
\end{figure}

\begin{proposition}\label{prop:R-I-inv}
  Assume that the decorated pointed link diagram $(D',p)$ is obtained
  from the decorated pointed link diagram $(D,p)$ by doing a positive
  Reidemeister-I stabilization away from the basepoint $p$, see
  \Figure{R-I}, and by extending the decoration arbitrarily on the
  extra crossing. Then
  $(\tuplethree{\OurCx^-(D,p)\{1\}}{\OurCx(D)}{\iota(D,p)})$ and
  $(\tuplethree{\OurCx^-(D',p)\{1\}}{\OurCx(D')}{\iota(D',p)})$ are
  isomorphic in $\hocatpair(\F_2[H,W])$.
\end{proposition}

\begin{proof}
  Let $\Crossings(D)$ and $\Crossings(D')$ be the set of crossings for
  $D$ and $D'$, respectively, and let $\{c\}=\Crossings(D')\setminus
  \Crossings(D)$.  Doing the $0$-resolution at $c$ produces a complete
  circle, say $a$, contained in the neighborhood of $c$ where the
  Reidemeister-I stabilization takes place. Then
  $(\{c\},\emptyset,\{c\},c,a)$ constitute a cancellation data for
  $D'$ as in \Definition{cancellation-data}, see also
  \cite[Section~5.1]{Kho-kh-categorification}.

  Consider the union of all the canceling pairs for this cancellation
  data.  They span a subcomplex of $\OurCx(D')$, and the corresponding
  quotient complex is naturally isomorphic to $\OurCx(D)$. Since the
  subcomplex spanned by the canceling pairs is trivial (from
  \Lemma{cancel-prelim}), and $p\notin a$, we get
  $(\tuplethree{\OurCx^-(D',p)\{1\}}{\OurCx(D')}{\iota(D',p)})$ is
  isomorphic to
  $(\tuplethree{\OurCx^-(D,p)\{1\}}{\OurCx(D)}{\iota(D,p)})$ in
  $\hocatpair(\F_2[H,W])$.
\end{proof}

\begin{proposition}\label{prop:R-II-inv}
  Assume that the decorated pointed link diagram $(D',p)$ is obtained
  from the decorated pointed link diagram $(D,p)$ by adding a pair of
  crossings via a Reidemeister-II move away from the basepoint $p$,
  see \Figure{R-II}, and by extending the decoration arbitrarily on
  the two extra crossings. Then
  $(\tuplethree{\OurCx^-(D,p)\{1\}}{\OurCx(D)}{\iota(D,p)})$ and
  $(\tuplethree{\OurCx^-(D',p)\{1\}}{\OurCx(D')}{\iota(D',p)})$ are
  isomorphic in $\hocatpair(\F_2[H,W])$.
\end{proposition}

\setlength\tempfigdim{0.1\textheight}
\begin{figure}
  \[
  \xymatrix{
    \vcenter{\hbox{\includegraphics[height=\tempfigdim]{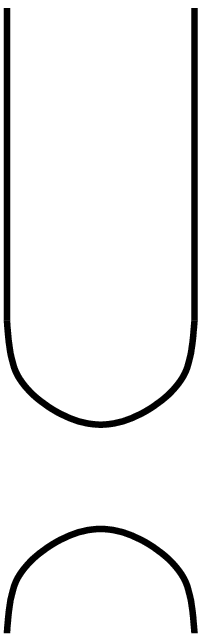}}}\ar[r]\ar@{=>}[d]&
    \vcenter{\hbox{\includegraphics[height=\tempfigdim]{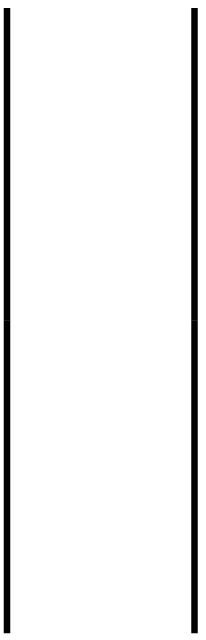}}}\ar[d]\\
    \vcenter{\hbox{\includegraphics[height=\tempfigdim]{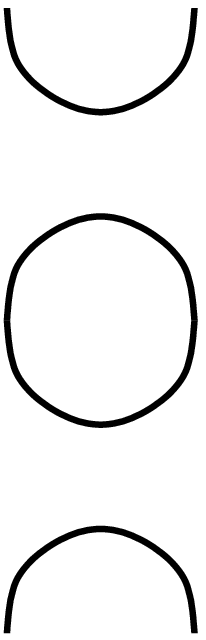}}}\ar@{=>}[r]&
    \vcenter{\hbox{\includegraphics[height=\tempfigdim]{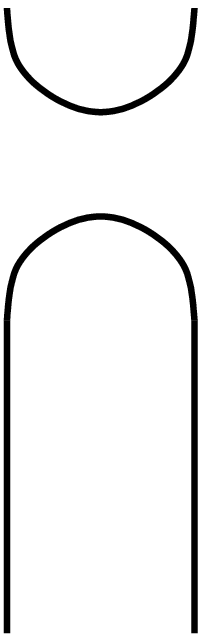}}}
  }
  \]
  \caption{The partial cube of resolutions of $D'$ for the RII
    invariance. We cancel along the double arrows leaving a subquotient
    complex isomorphic to the complex for $D$.}\label{fig:Reid2-cube}
\end{figure}

\begin{proof}
  The argument is similar to the previous one. Let
  $\{c_1,c_2\}=\Crossings(D')\setminus\Crossings(D)$, numbered top to
  bottom (as per \Figure{R-II}). Then the subquotient complex of
  $\OurCx(D')$ spanned by the Khovanov generators that live over the
  $0$-resolution at $c_1$ and the $1$-resolution at $c_2$ is naturally
  isomorphic (after the correct bi-grading shifts) to $\OurCx(D)$.

  Doing the $1$-resolution at $c_1$ and the $0$-resolution at $c_2$
  produces a complete circle, say $a$, contained in the neighborhood
  of $\{c_1,c_2\}$ where the Reidemeister-II move takes place. Then
  both $(\{c_1,c_2\},\emptyset,\{c_1\},c_1,a)$ and
  $(\{c_1,c_2\},\{c_1\},\{c_1,c_2\},c_2,a)$ are cancellation data for
  $D'$, see also \cite[Section~5.3]{Kho-kh-categorification}.

  Using \Lemma{cancel-prelim}, we can first cancel the subcomplex of
  $\OurCx(D')$ spanned by all the canceling pairs for
  $(\{c_1,c_2\},\{c_1\},\{c_1,c_2\},c_2,a)$, and then cancel the
  quotient complex spanned by the canceling pairs of
  $(\{c_1,c_2\},\emptyset,\{c_1\},c_1,a)$. After all the
  cancellations, we will be left with the subquotient complex
  isomorphic to $\OurCx(D)$, and (and since $p\notin a$) this
  establishes the isomorphism between
  $(\tuplethree{\OurCx^-(D,p)\{1\}}{\OurCx(D)}{\iota(D,p)})$ and
  $(\tuplethree{\OurCx^-(D',p)\{1\}}{\OurCx(D')}{\iota(D',p)})$ in
  $\hocatpair(\F_2[H,W])$. See also \Figure{Reid2-cube}.
\end{proof}

\begin{proposition}\label{prop:R-III-inv}
  Assume that the decorated pointed link diagram $(D',p)$ is obtained
  from the decorated pointed link diagram $(D,p)$ by performing a
  Reidemeister-III move aways from the basepoint $p$, with the
  decorations being consistent, as shown in \Figure{R-III}.  Then
  $(\tuplethree{\OurCx^-(D,p)\{1\}}{\OurCx(D)}{\iota(D,p)})$ and
  $(\tuplethree{\OurCx^-(D',p)\{1\}}{\OurCx(D')}{\iota(D',p)})$ are
  isomorphic in $\hocatpair(\F_2[H,W])$.
\end{proposition}

\setlength\tempfigdim{0.15\textheight}
\begin{figure}
  \[
  \xymatrix{
    &&\vcenter{\hbox{\includegraphics[height=\tempfigdim]{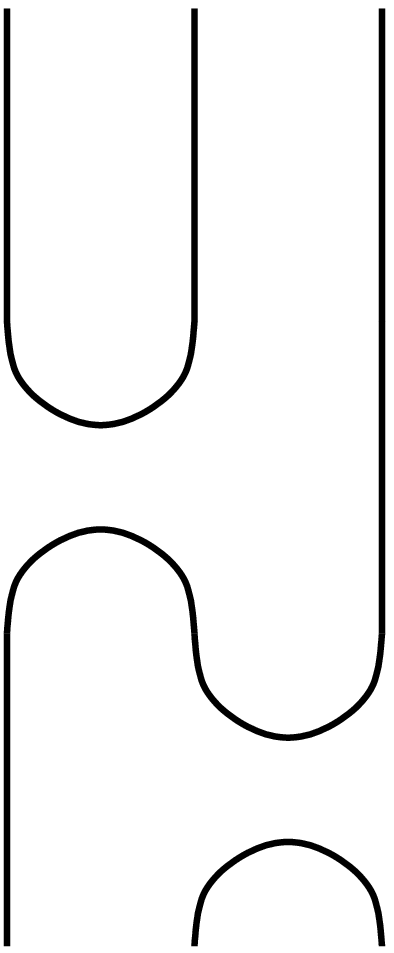}}}\ar[r]\ar[d]|(0.5)\hole&\vcenter{\hbox{\includegraphics[height=\tempfigdim]{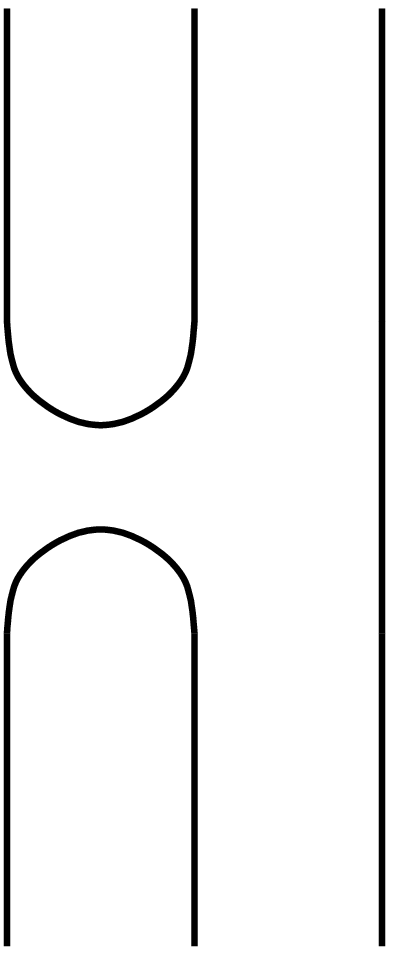}}}\ar[d]\\
    \vcenter{\hbox{\includegraphics[height=\tempfigdim]{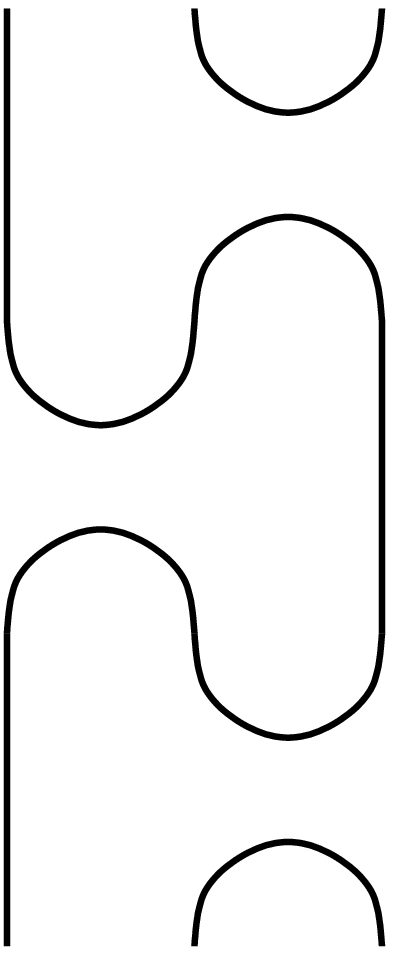}}}\ar[r]\ar@{=>}[d]\ar[urr]|(0.37)\hole&\vcenter{\hbox{\includegraphics[height=\tempfigdim]{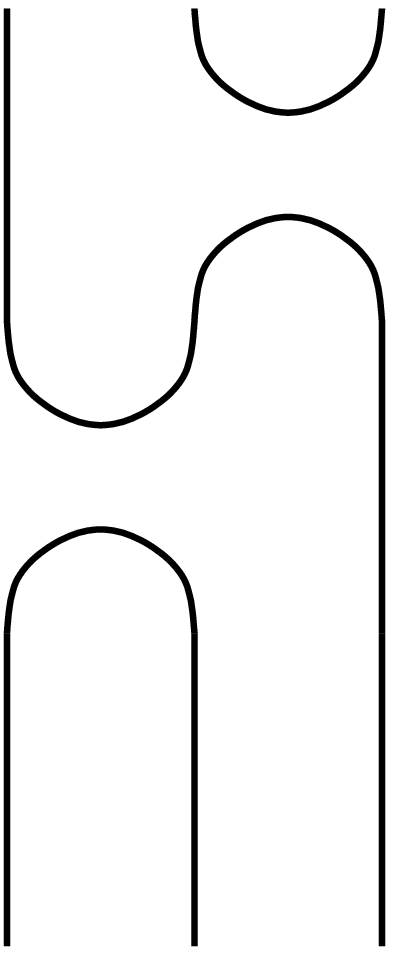}}}\ar[d]\ar@{.>}[r]\ar[urr]&\vcenter{\hbox{\includegraphics[height=\tempfigdim]{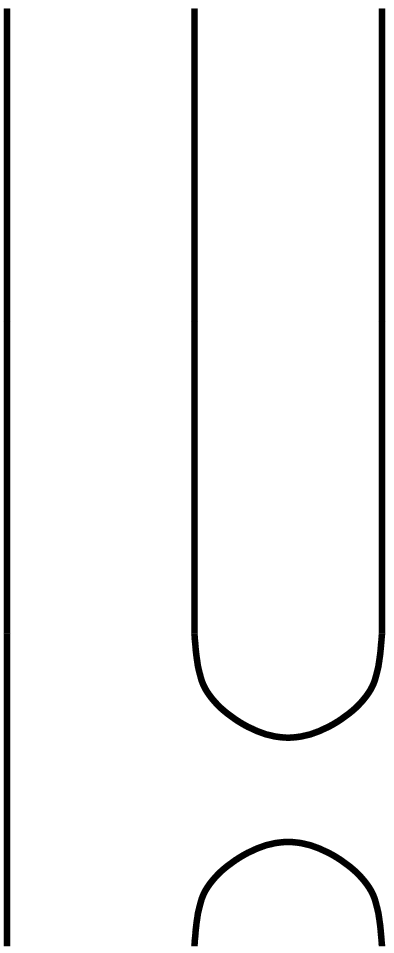}}}\ar[r]&\vcenter{\hbox{\includegraphics[height=\tempfigdim]{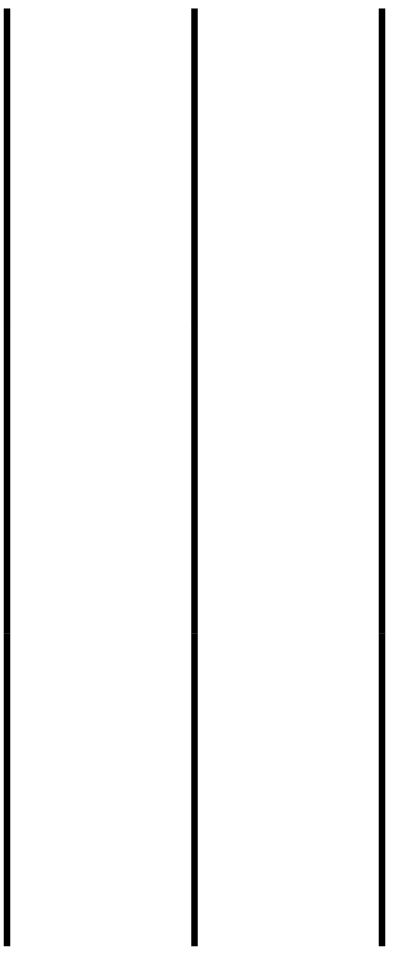}}}\\
    \vcenter{\hbox{\includegraphics[height=\tempfigdim]{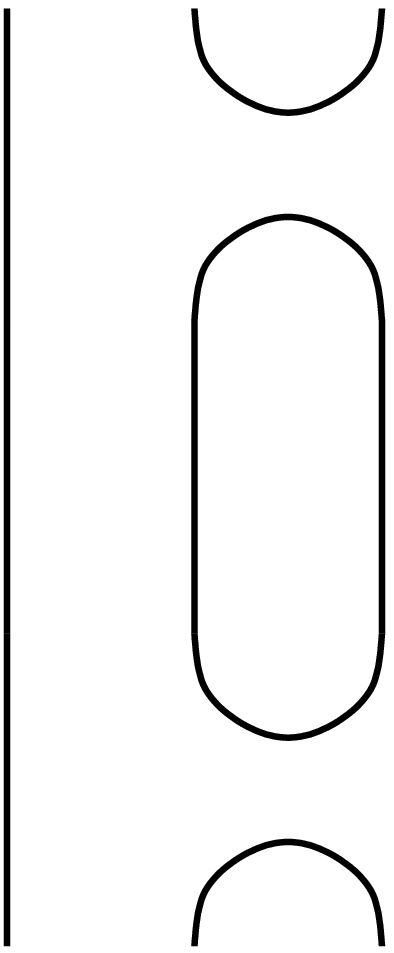}}}\ar@{=>}[r]\ar[urr]|(0.37)\hole|(0.5)\hole|(0.63)\hole&\vcenter{\hbox{\includegraphics[height=\tempfigdim]{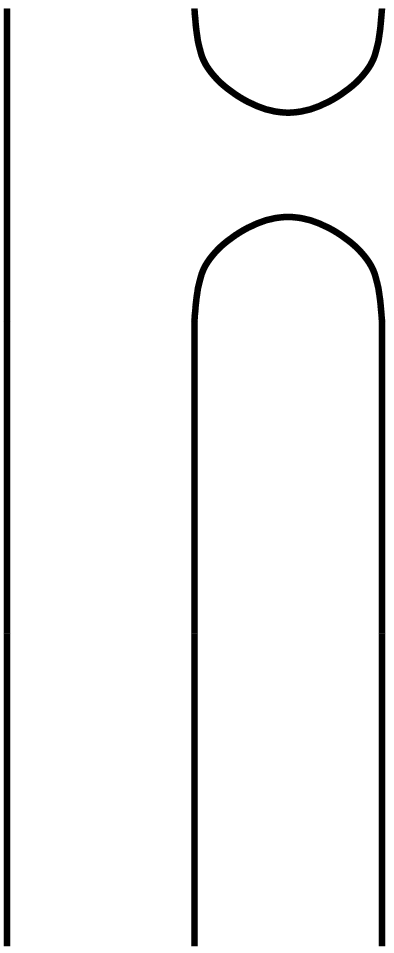}}}\ar[urr]
  }
  \]
  \caption{The partial cube of resolutions of $D$ for the RIII
    invariance. We cancel along the double arrows; the remainder is
    not a subquotient complex, so the cancellation produces the new
    dotted arrow. We can produce similar cancellations for $D'$ to
    arrive at the same diagram.}\label{fig:Reid3-cube}
\end{figure}

\begin{proof}
  The proof is slightly different from the previous ones. We will do
  cancellations as before, but on each of the diagrams $D$ and $D'$,
  and reduce both to the same complex. Let us describe the
  cancellations for $D$ in more detail.

  Let $c_1,c_2,c_3$ be the crossings of $D$ where the Reidemeister-III
  move takes place, numbered top to bottom (as per
  \Figure{R-III}). Doing the $0$-resolutions at $c_1$ and $c_3$, and
  the $1$-resolution at $c_2$ produces a complete circle, say $a$,
  contained in the neighborhood of $\{c_1,c_2,c_3\}$ where the
  Reidemeister-III move takes place. Then all three of
  $(\{c_1,c_2,c_3\},\emptyset,\{c_2\},c_2,a)$,
  $(\{c_1,c_2,c_3\},\{c_2\},\{c_1,c_2\},c_1,a)$, and
  $(\{c_1,c_2,c_3\},\{c_2\},\{c_2,c_3\},c_3,a)$ are cancellation data
  for $D$, see also \cite[Section~5.5]{Kho-kh-categorification} and
  \cite[Theorem~7.2]{Szab-kh-geometric}.

  Using \Lemma{cancel-prelim}, we first cancel the quotient complex of
  $\OurCx(D)$ spanned by all the canceling pairs for
  $(\{c_1,c_2,c_3\},\emptyset,\{c_2\},c_2,a)$. We then cancel all the
  canceling pairs for
  $(\{c_1,c_2,c_3\},\{c_2\},\allowbreak\{c_2,c_3\},\allowbreak
  c_3,a)$; this is neither a subcomplex nor a quotient complex, so
  this cancellation produces new zigzag differentials, as in
  \Equation{zigzag}. The new differentials go from the Khovanov
  generators living over $\{c_3\}\cup w_1$ to Khovanov generators
  living over $\{c_1,c_2\}\cup w_2$ as $w_1,w_2$ vary over subsets of
  $\Crossings(D)\setminus\{c_1,c_2,c_3\}$. This is shown by the dotted
  arrow in \Figure{Reid3-cube}.

  Let us analyze the new differentials in more detail. First set up
  some more notation. Let $C_2^+$, $C_3$, $C_{12}$, and $C_{23}$
  denote the subquotient complexes of $\OurCx(D)$ spanned by Khovanov
  generators living over $\{c_2\}\cup w$ with $a$ not appearing in the
  monomial, $\{c_3\}\cup w$, $\{c_1,c_2\}\cup w$, $\{c_2,c_3\}\cup w$,
  respectively, for arbitrary
  $w\subseteq\Crossings(D)\setminus\{c_1,c_2,c_3\}$. There is a
  natural identification $C_{12}\cong C_{23}$ since the resolutions
  $\AssRes{\{c_1,c_2\}\cup w}$ and $\AssRes{\{c_2,c_3\}\cup w}$ are
  identical outside a neighborhood of $a$, and are canonically
  isotopic to each other inside the neighborhood. The new differential
  is from $C_3$ to $C_{12}$, and we claim that it is identical (under
  the above identification) to the part of the old differential
  $\diffOur$ that went from $C_3$ to $C_{23}$. To see this, recall
  that both $(\{c_1,c_2,c_3\},\{c_2\},\{c_1,c_2\},c_1,a)$, and
  $(\{c_1,c_2,c_3\},\{c_2\},\{c_2,c_3\},c_3,a)$ are cancellation data
  for $D$. Therefore, using \Lemma{cancel-prelim}, only $\dd_1$
  contributes the part of the differential $\diffOur$ that goes from
  $C_2^+$ to $C_{12}$ or from $C_2^+$ to $C_{23}$; and in either case,
  it produces a bijection between the Khovanov generators in $C_2^+$
  with the Khovanov generators in $C_{12}$ or $C_{23}$; the induced
  bijection between the Khovanov generators in $C_{12}$ and $C_{23}$
  is easily seen to be the above identification. The new zigzag
  differential from $C_3$ to $C_{12}$ is obtained by composing the
  part of the old differential $\diffOur$ from $C_3$ to $C_{23}$, and
  then mapping $C_{23}$ to $C_{12}$ by the above bijection. This shows
  that the new differential is identical to the part of the old
  differential $\diffOur$ that went from $C_3$ to $C_{23}$.

  We can perform similar cancellations for the diagram $D'$. To wit,
  if $c'_1,c'_2,c'_3$ are the crossings of $D'$ where the
  Reidemeister-III move takes place, numbered top to bottom (as per
  \Figure{R-III}), then doing the $0$-resolutions at $c'_1$ and $c'_3$
  and the $1$-resolution at $c'_2$ produces a complete circle, say
  $a'$, contained in the neighborhood where the Reidemeister-III move
  takes place. Then we cancel all Khovanov generators for the
  cancellation data $(\{c'_1,c'_2,c'_3\},\emptyset,\{c'_2\},c'_2,a')$
  and
  $(\{c'_1,c'_2,c'_3\},\{c'_2\},\allowbreak\{c'_2,c'_3\},\allowbreak c'_3,a')$. Since we
  had decorated the diagrams $D$ and $D'$ coherently, it is
  straightforward to see that after performing these cancellations, we
  end up in an isomorphic picture. (The post-cancellation complexes
  for $D$ and $D'$ are identical outside the region where the
  Reidemeister-III move occurs, and inside the region, the
  corresponding resolutions are canonically isotopic to one another.)

  This establishes (after noting that the basepoint $p$ is not in the
  circle $a$) that
  $(\tuplethree{\OurCx^-(D,p)\{1\}}{\OurCx(D)}{\iota(D,p)})$ and
  $(\tuplethree{\OurCx^-(D',p)\{1\}}{\OurCx(D')}{\iota(D',p)})$ are
  isomorphic in $\hocatpair(\F_2[H,W])$.
\end{proof}

\begin{lemma}\label{lem:complete-moves}
  If $(D,p)$ and $(D',p')$ are two decorated pointed link diagrams
  representing isotopic pointed links, then they can be connected by
  some sequence of positive Reidemeister-I moves, Reidemeister-II
  moves, and Reidemeister-III moves, as described in the statements of
  Propositions~\ref{prop:R-I-inv}--\ref{prop:R-III-inv} and
  \Figure{Reid-moves}, their inverses, and isotopy in $S^2$.
\end{lemma}

\begin{proof}
  This is essentially Reidemeister's theorem which states that any two
  link diagrams for the same link can be connected by isotopy in
  $\R^2$ and the three Reidemeister moves (and in particular, we only
  need one variant for each of the Reidemeister I and III moves).

  In presence of a single basepoint, we need two additional moves as
  shown in \Figure{pointed-Reid}, corresponding to moving the
  basepoint past a strand. However, as observed in
  \cite[Section~3]{Kho-kh-patterns}, these moves may be achieved via
  the usual three Reidemeister moves away from the basepoint and
  isotopy in $S^2$.

  \setlength\tempfigdim{0.15\textwidth}
  \begin{figure}
    \[
    \xymatrix@C=0.1\textwidth{
      \vcenter{\hbox{\includegraphics[width=\tempfigdim]{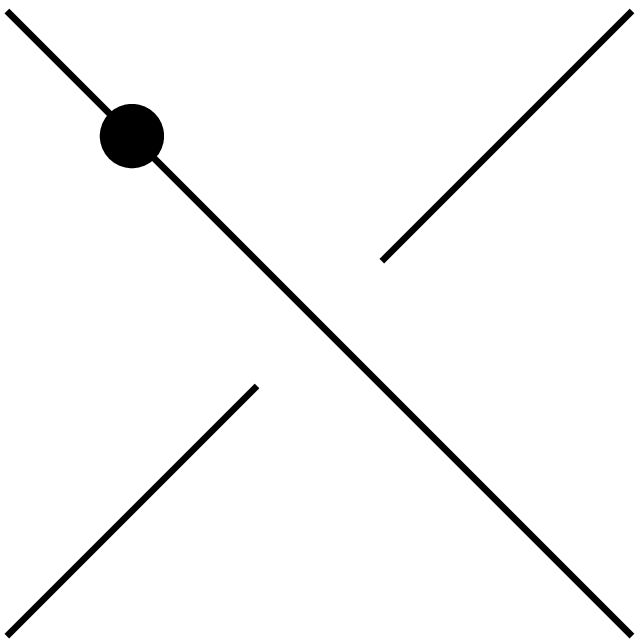}}}\ar@{<->}[r]&
      \vcenter{\hbox{\includegraphics[width=\tempfigdim]{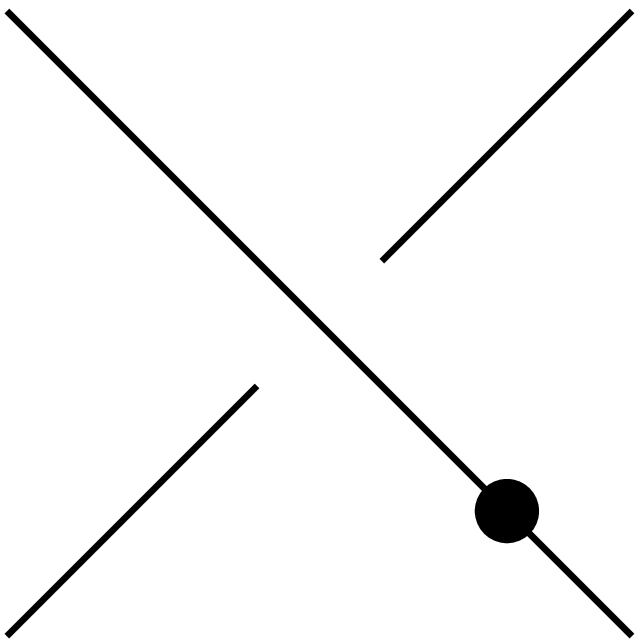}}}&
      \vcenter{\hbox{\includegraphics[width=\tempfigdim]{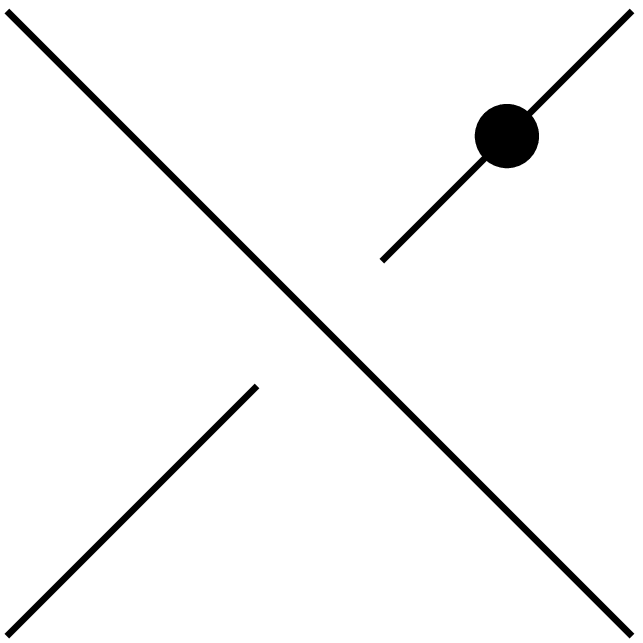}}}\ar@{<->}[r]&
      \vcenter{\hbox{\includegraphics[width=\tempfigdim]{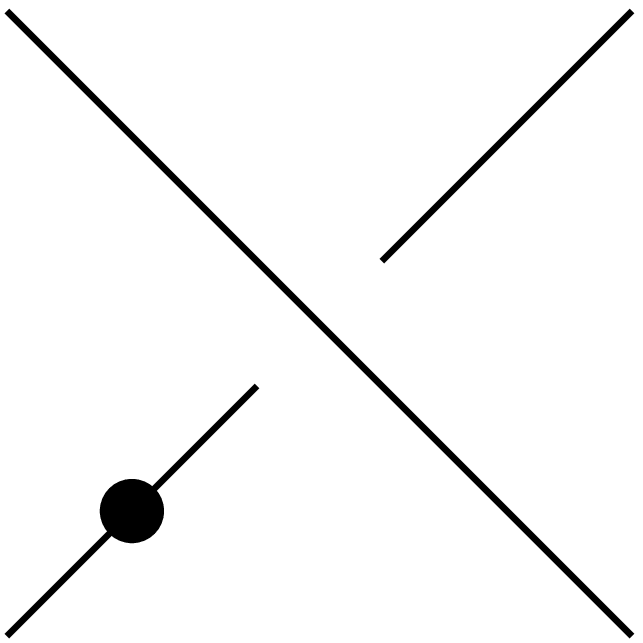}}} 
    }
    \]
    \caption{Additional moves for the basepoint.}\label{fig:pointed-Reid}
  \end{figure}
  
  \setlength\tempfigdim{0.25\textheight}
  \begin{figure}
    \[
    \xymatrix@C=0.1\textwidth{
      \vcenter{\hbox{\includegraphics[height=\tempfigdim]{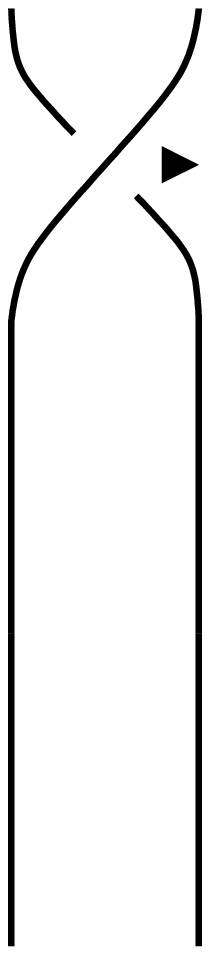}}}\ar@{<->}[r]&
      \vcenter{\hbox{\includegraphics[height=\tempfigdim]{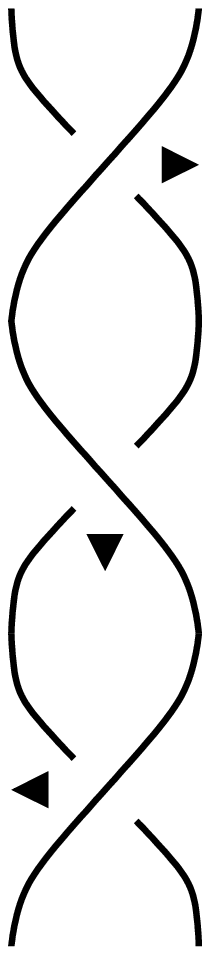}}}\ar@{<->}[r]&
      \vcenter{\hbox{\includegraphics[height=\tempfigdim]{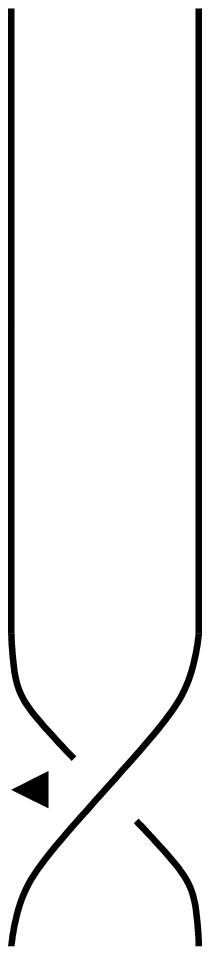}}} 
    }
    \]
    \caption{Changing the decoration at a single crossing.}\label{fig:decor-Reid}
  \end{figure}

  Finally, with regard to decorations, we need an additional move
  which changes the decoration at a single crossing. This move may be
  achieved by two Reidemeister II moves, as shown in
  \Figure{decor-Reid}.
\end{proof}

\begin{proof}[Proof of \Proposition{main-invariance}]
  This is immediate from \Lemma{complete-moves} and
  Propositions~\ref{prop:R-I-inv}--\ref{prop:R-III-inv}. The only
  thing to note is that the contribution functions $\cont{d}$ and
  $\cont{h}$ satisfy the naturality rule, i.e., they are preserved
  under isotopy in $S^2$; consequently, the chain complex $\OurCx$
  defined using $\cont{d}$ and $\cont{h}$ also remains invariant under
  isotopy in $S^2$.
\end{proof}

\section{Properties}\label{sec:properties}

From this section onwards, we will restrict to the unpointed case, and
only study the unreduced version $\OurCx$. In
Propositions~\ref{prop:R-I-inv}--\ref{prop:R-III-inv}, we associated
isomorphisms in $\hocat(\F_2[H,W])$ corresponding to knot
isotopy. We will now construct morphisms in
$\hocat(\F_2[H,W])$ for three additional local moves.

\begin{definition}\label{def:index-0-2}
  Fix a decorated link diagram $D$, and let $D'$ be the decorated link
  diagram obtained from $D$ by adding a small unknotted circle $a$
  disjoint from the $D$. (The transformation $D\to D'$ is usually
  called a `birth', and it corresponds to a link cobordism in
  $\R^3\times[0,1]$ with a single index-zero critical point. The
  transformation $D'\to D$ is usually called a `death', and it
  corresponds to a link cobordism in $\R^3\times[0,1]$ with a single
  index-two critical point.) We have a decomposition
  \[
  \OurCx(D')\cong\OurCx(D')_0\oplus\OurCx(D')_1,
  \]
  where $\OurCx(D')_0$ is the direct summand of $\OurCx(D')$ where the
  circle $a$ does not appear in the monomials for the Khovanov
  generators, while $\OurCx(D')_1$ is the direct summand of
  $\OurCx(D')$ where the circle $a$ does appear in the monomials. Each
  of $\OurCx(D')_0$ and $\OurCx(D')_1$ is identified with $\OurCx(D)$,
  after shifting the quantum gradings correctly.

  To a birth, we associate a morphism from $\OurCx(D)$ to $\OurCx(D')\{1\}$
  in $\hocat(\F_2[H,W])$ as the composition
  \[
  \OurCx(D)\cong\OurCx(D')_0\{1\}\into\OurCx(D')_0\{1\}\oplus\OurCx(D')_1\{1\}\cong\OurCx(D')\{1\},
  \]
  where the inclusion is an inclusion as a direct summand.

  To a death, we associate a morphism from $\OurCx(D')$ to $\OurCx(D)\{1\}$
  in $\hocat(\F_2[H,W])$ as the composition
  \[
  \OurCx(D')\cong\OurCx(D')_0\oplus\OurCx(D')_1\onto\OurCx(D')_1\cong\OurCx(D)\{1\},
  \]
  where the surjection is a projection to a direct summand.
\end{definition}

\begin{definition}\label{def:index-1}
  Assume two decorated link diagrams $D_0$ and $D_1$ are related as
  shown in \Figure{saddle}. That is, assume that there is a
  decorated link diagram $D$ with a distinguished crossing $c$, so
  that resolving $c$ by the $i$-resolution produces $D_i$, for
  $i=0,1$; and further assume the link diagrams $D_0$ and $D_1$ can
  be, and are, oriented coherently. (The transformation $D_0\to D_1$
  is usually called a `saddle', and it corresponds to a link cobordism
  in $\R^3\times[0,1]$ with a single index-one critical point.) After
  an overall shift of the bigradings (which may depend on the
  diagrams), there is an identification
  \[
  \OurCx(D)\cong\Cone(f\from\OurCx(D_0)\to\OurCx(D_1)\{-1\}).
  \]
  where $f$ is the part of the differential $\diffOur$ for $\OurCx(D)$
  that goes from the $0$-resolution at $c$ to the $1$-resolution at
  $c$. To the saddle move $D_0\to D_1$, we associate the morphism $f$
  from $\OurCx(D_0)$ to $\OurCx(D_1)\{-1\}$ in
  $\hocat(\F_2[H,W])$.
\end{definition}

\setlength\tempfigdim{0.15\textwidth}
\begin{figure}
    \[
    \xymatrix@C=0.1\textwidth{
      \vcenter{\hbox{\includegraphics[width=\tempfigdim]{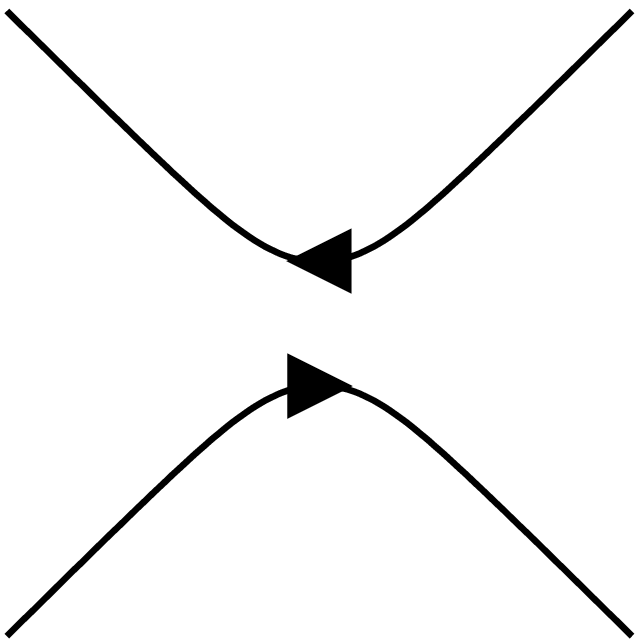}}}\ar[r]&
      \vcenter{\hbox{\includegraphics[width=\tempfigdim]{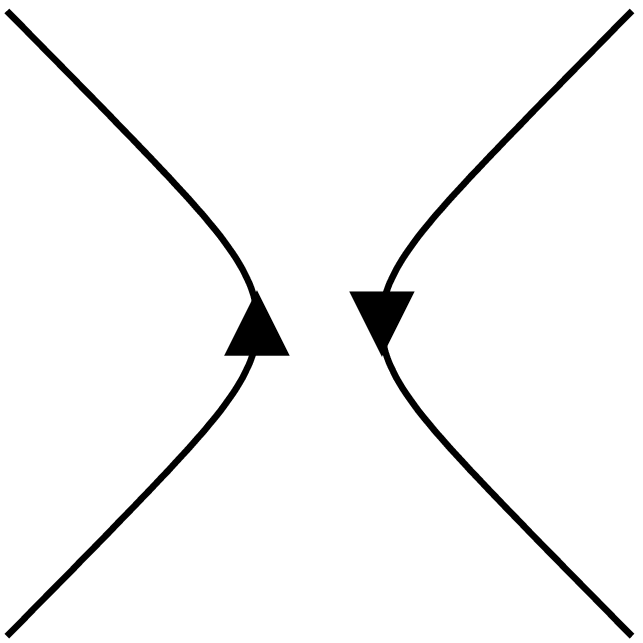}}}
    }
    \]
    \caption{The saddle cobordism from $D_0$ to $D_1$. Note that the
      link diagrams are oriented coherently.}\label{fig:saddle}
  \end{figure}

  To a link cobordism in $\R^3\times[0,1]$, one can associate maps on
  Khovanov chain complex \cite{Kho-kh-categorification,
    Jac-kh-cobordisms, Kho-kh-cob, CMW-kh-functoriality} and on the
  Bar-Natan chain complex \cite{Bar-kh-tangle-cob,
    CMW-kh-functoriality}. These maps are defined by composing maps
  associated to elementary moves, namely the three Reidemeister moves,
  birth, death, and saddle. We may also use our maps from
  Propositions~\ref{prop:R-I-inv}--\ref{prop:R-III-inv} and
  Definitions~\ref{def:index-0-2}--\ref{def:index-1} to define map
  associated to link cobordisms presented as a sequence of elementary
  moves. We will not prove that this map is well-defined in
  $\hocat(\F_2[H,W])$, that is, it only depends on the isotopy class
  of the link cobordism, and not on a choice of presentation as a
  sequence of elementary moves. Nevertheless, when we specialize
  $W=0$, we get the existing link cobordism map on the Bar-Natan
  theory.

\begin{proposition}\label{prop:induced-cob-map-same}
  For any link cobordism in $\R^3\times[0,1]$ viewed as a sequence of
  elementary moves, the maps on $\OurCx$ defined in
  Propositions~\ref{prop:R-I-inv}--\ref{prop:R-III-inv} and
  Definitions~\ref{def:index-0-2}--\ref{def:index-1} induce the
  standard link cobordism maps in the Bar-Natan theory $\BNCx$.
\end{proposition}

\begin{proof}
  This is immediate from the definitions (which we have not given
  here) of the Bar-Natan link cobordism maps, which are
  specializations \cite[Section~9.3]{Bar-kh-tangle-cob} of maps
  defined in a more general setting
  \cite[Sections~4.3~and~8.1]{Bar-kh-tangle-cob}.
\end{proof}

Next, we will use these link cobordism maps to prove some structure
theorems for the total homology. Before proceeding, let us collect a
few facts about the Bar-Natan theory.
\begin{enumerate}[label=(BN-\arabic*), ref=(BN-\arabic*)]
\item\label{item:BN-total-rank} For an $l$-component link $L$, the
  homology of the localized Bar-Natan complex $\lBNCx$ is $2^l$ copies
  of $\F_2[H,H^{-1}]$, while the homology of the filtered Bar-Natan
  complex $\fBNCx=\OurCx/\{H=1,W=0\}$ has rank $2^l$
  \cite{Lee-kh-endomorphism, Turner-kh-BNSeq}; in either case, the
  generators correspond to the orientations of $L$, and the
  $\gr_h$-preserving reduction $\lBNCx\to\fBNCx$ preserves this
  correspondence. (Since $\fBNCx\otimes\F_2[\Z]$ can be identified
  with $\lBNCx$, as in the proof of \Proposition{total-rank}, the two
  statements are equivalent.) In more detail, consider some
  orientation $o$ on the link $L$, presented as a link diagram $D$ in
  the plane $\R^2$. Fix a checkerboard coloring of the complement of
  $D$ in $\R^2$; for concreteness, one usually decrees the unbounded
  region to be colored white. Consider the oriented resolution of $D$
  according to the orientation $o$, and let $\{x_1,\dots,x_k\}$ be the
  complete circles of the resolution. Then each of the individual
  circles $\{x_i\}$ are also oriented according to $o$. Consider the
  one-variable polynomial (over $\F_2[H]$) in $x_i$, which is $x_i$ or
  $H+x_i$, depending on whether $x_i$ is oriented as the boundary of a
  black region or a white region, respectively; then consider the
  product of all these $k$ one-variable polynomials. This $k$-variable
  polynomial may be viewed as a linear combination over $\F_2[H]$ of
  square-free monomials in the circles $\{x_i\}$, and thereby viewed
  as a linear combination of the Khovanov generators over this
  oriented resolution, see also~\ref{item:Khovanov-generators}
  from \Section{background}. This linear combination represents the
  generator corresponding to $o$ in $\lBNCx$. Let $g_{\BN}(o)$ be the
  corresponding generator in $H_*(\fBNCx)$.

  In particular, note that the homological grading of $g_{\BN}(o)$ is
  given by the linking number 
  \[
  \gr_h(g_{\BN}(o))=2\,\mathrm{lk}(L_0,L\setminus L_0)
  \]
  where $L_0\subseteq L$ is the sublink where the orientation $o$
  agrees with the starting orientation of $L$ (the one that was used
  to define the homological grading $\gr_h$ in the first place,
  cf.~\ref{item:Kh-hom-grading} from \Section{background}), with the
  following understanding: the linking number is computing after
  orienting both $L_0$ and $L\setminus L_0$ according to the starting
  orientation of $L$; and the linking number with the empty link is
  zero.
\item\label{item:Ras-map} For any oriented link cobordism from $L_1$ to $L_2$ in
  $\R^3\times[0,1]$ without any closed components, and for any
  orientation $o$ on $L_1$, the Bar-Natan link cobordism map on
  $\fBNCx$ acts as follows on the generator $g_{\BN}(o)$,
  \[
  g_{\BN}(o)\mapsto\!\!\!\!\!\sum_{\substack{\text{$o'$ orientation
        on $L_2$}\\\mathrlap{\text{$o$ and $o'$ extend to an orientation on the
        cobordism}}\phantom{\text{$o'$ orientation
        on $L_2$}}}}\!\!\!\!\! g_{\BN}(o'),
  \]
  see~\cite{Ras-kh-slice}, see also \cite{LS-rasmussen}. 
\item\label{item:Ras-s} For a knot $K$, the Rasmussen $s$-invariant of
  the knot (over $\F_2$) is defined as
  \begin{align*}
    s(K)&=\max\set{n}{\Filt_n\BNCx(K)\text{ contains a representative for }g_{\BN}(o)}+1\\
    &=\max\set{n}{\Filt_n\BNCx(K)\text{ contains a representative for
      }g_{\BN}(-o)}+1\\
    &=\max\set{n}{\Filt_n\BNCx(K)\text{ contains a representative for }g_{\BN}(o)+g_{\BN}(-o)}-1,
  \end{align*}
  where $o$ is any orientation on $K$ and $\Filt_n\BNCx(K)$ is the
  subcomplex of $\BNCx(K)$ supported in quantum grading $n$ or more.
  This was originally defined over any field of characteristic
  different from $2$ in~\cite{Ras-kh-slice}, and extended to $\F_2$ in
  \cite{LS-rasmussen}.
\end{enumerate}

\begin{proposition}\label{prop:total-rank}
  Fix any $l$-component link $L$. The following hold:
  \begin{enumerate}[label=(R-\arabic*), ref=(R-\arabic*)]
  \item\label{item:rank-Hf} The homology of $\HfOurCx=\OurCx/\{H=1\}$
    is isomorphic to $2^l$ copies of $\F_2[W]$.
  \item\label{item:rank-Hl} The homology of $\HlOurCx$ is isomorphic
    to $2^l$ copies of $\F_2[H,H^{-1},W]$
  \item\label{item:rank-f} The homology of $\fOurCx=(\Cx,\dd+\hh)$ has
    rank $2^l$.
  \end{enumerate}
  In each case, the $2^l$ generators are in a canonical correspondence
  with the orientations of $L$ (as was the case for $\lBNCx$ and
  $\fBNCx$, cf.~\ref{item:BN-total-rank}), and the reductions
  \[
  \xymatrix{  
    \HlOurCx\ar[r]\ar[d]&\HfOurCx\ar[r]\ar[d]&\fOurCx\\
    \lBNCx\ar[r]&\fBNCx
  }
  \]
  preserve this correspondence, and the left four reductions (the
  ones forming the square) also preserve the homological grading.
\end{proposition}

Towards this end, we will need to understand the spectral sequence
$H_*(\HfOurCx/\{W=0\})\otimes\F_2[W]\rightrightarrows H_*(\HfOurCx)$. The
chain complex $\HfOurCx$ is singly graded by $\gr_h$, and carries a
filtration by powers of $W$. Its associated graded object is
isomorphic to
\[
\HfOurCx/\{W=0\}\otimes\F_2[W]=\fBNCx\otimes\F_2[W],
\]
whose homology is $2^l$ copies of $\F_2[W]$, via
\ref{item:BN-total-rank}. The filtrations induce a spectral sequence
over $\F_2[W]$ (see for example~\cite[Theorem~2.6]{McC-top-guide}) starting at the homology of the
associated graded object, and converging to the homology of
$\HfOurCx$. Since the complex $\HfOurCx$ is finitely generated over
$\F_2[W]$, the spectral sequence is forced to collapse after finitely
many pages. We will in fact show that the spectral sequence collapses
immediately.

\begin{lemma}\label{lem:spectral-sequence-collapse}
  The above spectral sequence
  $H_*(\fBNCx)\otimes\F_2[W]\rightrightarrows H_*(\HfOurCx)$
  has no higher differentials.
\end{lemma}

\begin{proof}
  We first prove this when the link $L$ is a disjoint union of some
  copies of the Hopf link and some copies of the unknot. Since adding
  a disjoint unknot component has the effect of tensoring everything
  with a two-dimensional vector space, we might assume $L$ is merely a
  disjoint union of Hopf link components, and let $M_1,\dots,M_k$ be
  the components (so that $L$ has $l=2k$ link components).  Fix a link
  diagram for $L$ with $2k$ crossings. Let $c_{i1}$ and $c_{i2}$ be
  the two crossings in the diagram for $M_i$, and fix a decoration on
  the link diagram so that the surgery arcs $\al_{c_{i1}}$ and
  $\al_{c_{i2}}$ are oriented in parallel. 

  We will first analyze the complex $\HfOurCx(M_i)$. For any subset
  $u\subseteq\{c_{i1},c_{i2}\}$, let $\ol{u}\in\{0,1\}^2$ be the
  corresponding vertex (that is, $c_{ij}\in u$ iff $\ol{u}_j=1$). Let
  $\{x^{\ol{u}}_j\}$ be the circles appearing in the complete
  resolution of $M_i$ corresponding to $u$; therefore, the chain group
  over $u$ will be generated by the square-free monomials in
  $x^{\ol{u}}_j$; for clarity, we will denote the monomial $1$ as
  $1^{\ol{u}}$. Then the complex $\HfOurCx(M_i)$ is the following:
  \[
  \xymatrix@R=1ex@W=0ex@M=0ex@C=1ex{
    & &&&&&&&&&& &1^{10}&\ar@{-->}[2,10]\ar[3,10]\ar[4,10]\\
    & &&&&&&&&&& &x^{10}_1&\ar[4,10]\\
    1^{00}&\ar[-2,10]\ar[4,10]\ar[0,22]|-W &&&&&&&&&& && &&&&&&&&&& &1^{11}\\
    x^{00}_1&\ar[-2,10]\ar[4,10] &&&&&&&&&& && &&&&&&&&&& &x^{11}_1(1^{11}+x^{11}_2)\\
    x^{00}_2&\ar[-3,10]\ar[3,10] &&&&&&&&&& && &&&&&&&&&& &x^{11}_2(1^{11}+x^{11}_1)\\
    x^{00}_1x^{00}_2&\ar@{-->}[-4,10]\ar@{-->}[2,10] &&&&&&&&&& && &&&&&&&&&& &x^{11}_1x^{11}_2\\
    & &&&&&&&&&& &1^{01}&\ar@{-->}[-4,10]\ar[-3,10]\ar[-2,10]\\
    & &&&&&&&&&& &x^{01}_1&\ar[-2,10]
    }
  \]
  Here, we have performed a change of basis on the generators coming
  from the resolution at $\{c_{i1},c_{i2}\}$. The dotted arrows come
  from the differentials $\hh_1$. The short solid arrows come from the
  differentials $\dd_1$, while the long solid arrow (from
  $1^{00}$ to $1^{11}$) comes from the
  differential $\dd_2$; it picks up a power of $W$, which we have
  indicated.

  We observe that after this change of basis, the complex
  $\HfOurCx(M_i)$ breaks up into two direct summands: the summand
  $S_i$ generated by $1^{00}$, $1^{10}$, $1^{01}$, $1^{11}$,
  $x_1^{11}(1^{11}+x_2^{11})$, and $x_2^{11}(1^{11}+x_1^{11})$; and
  the summand $T_i$ generated by $x_1^{00}$, $x_2^{00}$,
  $x_1^{00}x_2^{00}$, $x_1^{10}$, $x_1^{01}$, and
  $x_1^{11}x_2^{11}$. Furthermore, $S_i$ contains two of the four
  homology generators of the filtered Bar-Natan complex $\fBNCx$,
  namely, $x_1^{11}(1^{11}+x_2^{11})$ and $x_2^{11}(1^{11}+x_1^{11})$,
  and they live in the same homological grading; and $T_i$ contains
  the other two homology generators of $\fBNCx$, namely,
  $x_1^{00}+x_1^{00}x_2^{00}$ and $x_2^{00}+x_1^{00}x_2^{00}$, and
  they too live in the same homological grading.

  Now look at the complex for $L=\coprod_i M_i$. The chain complex
  $\HfOurCx(L)$ is not directly related to the tensor product of the
  chain complexes $\HfOurCx(M_i)$. The chain group is indeed the
  tensor product of the individual chain groups, and the differential
  coming from $\dd$ behaves like a tensor product, but the
  differential coming from $\hh$ is gotten by applying it to any
  non-empty subset of the individual chain groups (as opposed to just
  one, which would have been the case for the tensor product). Since
  the only non-zero terms in the differential for $\HfOurCx(M_i)$ come
  from $\dd_1$, $\dd_2$ and $\hh_1$, we can write the differential for
  $\HfOurCx(L)$ succinctly as follows. For any generators
  $\ga_i\in\Cx(M_i)$, the differential on
  $\ga_1\otimes\dots\otimes\ga_k$ in $\HfOurCx(L)$ is the
  following sum:
  \begin{align*}
    &\qquad\sum_i \ga_1\otimes\dots\otimes(\dd_1+W\dd_2)(\ga_i)\otimes\dots\otimes\ga_k+\sum_{\emptyset\neq A\subseteq\{1,\dots,k\}}W^{\card{A}-1}(\bigotimes_{i\in A}\hh_1(\ga_i))\otimes(\bigotimes_{i\notin A}\ga_i)\\
    &=\sum_i \ga_1\otimes\dots\otimes(\dd_1+W\dd_2)(\ga_i)\otimes\dots\otimes\ga_k\\
    &\qquad\qquad{}+\frac{(\Id+W\hh_1)(\ga_1)\otimes\dots\otimes(\Id+W\hh_1)(\ga_k)-\ga_1\otimes\dots\otimes\ga_k}{W}.
  \end{align*}

  Therefore, despite not being the tensor product, we still get $2^k$
  direct summands for $\HfOurCx(L)$ coming from the direct summands
  $S_i$ and $T_i$ for $\HfOurCx(M_i)$, for $i=1,\dots,k$. That is, the
  chain complex $\HfOurCx(L)$ can be viewed as $2^k$ different
  filtered chain complexes, not interacting with one another, each
  with an associated spectral sequence. Furthermore, each summand
  contains $2^k$ homology generators of the filtered Bar-Natan complex
  $\fBNCx$, all living in the same homological grading.

  We have so far not delved into the details of the the spectral
  sequence associated to a filtered chain complex. Peeking into
  \cite[the proof of Theorem~2.6]{McC-top-guide}, we see that the
  higher differentials correspond to zigzags of the same form as
  described in \Equation{zigzag}. Therefore, the spectral sequence
  associated to a direct sum of filtered chain complexes is the direct
  sum of the individual spectral sequences. Consequently, in the
  situation at hand, the entire spectral sequence decomposes into
  $2^k$ summands. On the other hand, it is easy to see from grading
  considerations that the higher differentials are zero for each of
  the summands. To wit, the homology of the first page of each summand
  has $2^k$ generators (over $\F_2[W]$) living in the same homological
  grading. Since the higher differentials increase homological grading
  $\gr_h$ by one, and $\gr_h(W)=-1$, we see that there are no higher
  differentials.

  \setlength\tempfigdim{0.15\textwidth}
  \begin{figure}
    \[
    \xymatrix@C=0.1\textwidth{
      \vcenter{\hbox{\includegraphics[width=\tempfigdim]{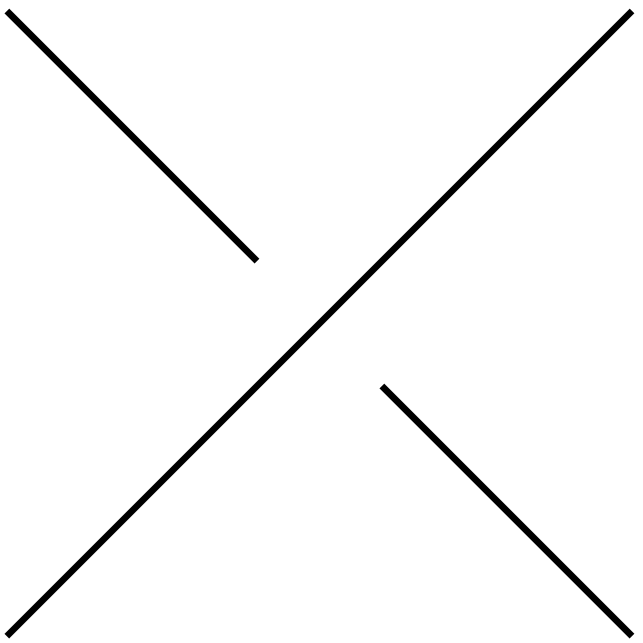}}}\ar[r]&
      \vcenter{\hbox{\includegraphics[width=\tempfigdim]{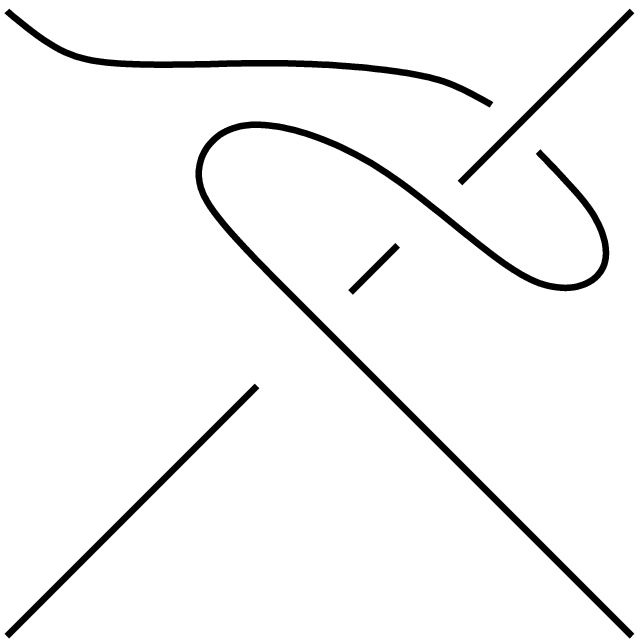}}}\ar[r]&
      \vcenter{\hbox{\includegraphics[width=\tempfigdim]{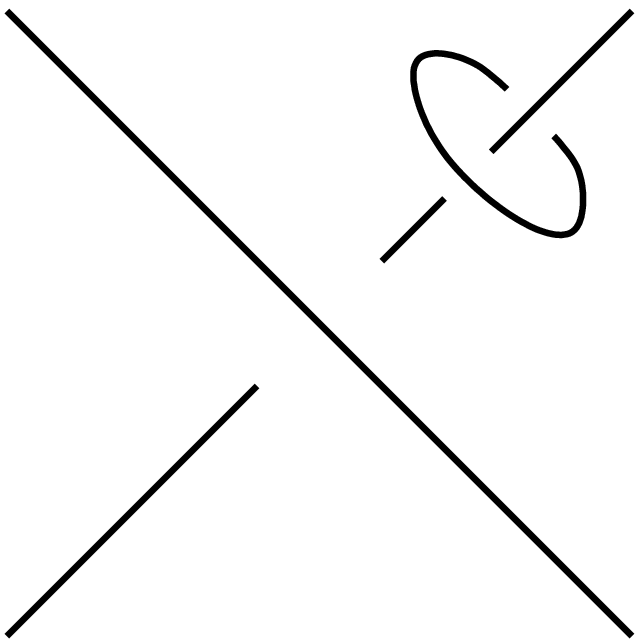}}} 
    }
    \]
    \caption{Performing a crossing change by a saddle (which splits a
      component into two) at the cost of introducing a meridional
      circle on one of the strands.}\label{fig:saddle-meridian}
  \end{figure}

  This completes the proof that the spectral sequence
  $H_*(\fBNCx(L))\otimes\F_2[W]\rightrightarrows H_*(\HfOurCx(L))$ has
  no higher differentials when $L$ is a disjoint union of some copies
  of the Hopf link and some copies of the unknot. To see how this
  implies the general statement, we will construct an oriented link
  cobordism $S$ in $\R^3\times[0,1]$ from our starting link $L$ to
  some other link $L'$ which is a disjoint union of some copies of the
  Hopf link and some copies of the unknot, satisfying the following
  properties:
  \begin{enumerate}
  \item The map $\pi_0(L)\to\pi_0(S)$ (induced from the inclusion
    $L\cong S\cap(\R^3\times\{0\})\into S$) is a bijection.
  \item The map $\pi_0(L')\to\pi_0(S)$ (induced from the inclusion
    $L'\cong S\cap(\R^3\times\{1\})\into S$) is a surjection.
  \end{enumerate}
  One way to construct such a cobordism is illustrated in
  \Figure{saddle-meridian}. We may perform a crossing change by a
  single saddle which adds a meridional circle to one of the strands,
  and the saddle splits a link component into two. After suitable such
  crossing changes, we can produce an unlink, with each component
  having some number meridional circles attached to them. After
  performing a few more splits using saddles, we can convert this
  picture into a disjoint union of Hopf links and an unlink. During
  this cobordism, the only elementary moves that we used were link
  isotopy and saddles that were splits; therefore, this link cobordism
  satisfies the above two properties.

  Using the maps from
  Propositions~\ref{prop:R-I-inv}--\ref{prop:R-III-inv} and
  Definitions~\ref{def:index-0-2}--\ref{def:index-1}, we get a map \(
  \HfOurCx(L)\to\HfOurCx(L') \) in $\hocat(\F_2[W])$. Being a map over
  $\F_2[W]$, the filtrations given by the powers of $W$ are
  preserved. Therefore, we get a map between the two spectral
  sequences, see for example, \cite[Theorem~3.5]{McC-top-guide}.

  \Proposition{induced-cob-map-same} implies the map on the first page
  is the standard map on the filtered Bar-Natan theory (tensored with
  $\F_2[W]$). Since our link cobordism $S$ satisfies the above
  conditions, for any orientation $o$ on $L$, there exists a unique
  orientation $o'$ on $L'$ so that $o$ and $o'$ can be extended to an
  orientation on $S$. Therefore, \ref{item:Ras-map} implies that the
  map on $H_*(\fBNCx)$ is injective.

  Summarizing, we get a map from the spectral sequence for $L$ to the
  spectral sequences for $L'$. It is injective on the homology of the
  first page; and the spectral sequence for $L'$ has no higher
  differentials. Therefore, the spectral sequence for $L$ has no
  higher differentials either. (The last step is merely the
  observation that if $f$ is an injective chain map from a chain
  complex $(C_1,d_1)$ to a chain complex $(C_2,d_2=0)$, then $d_1=0$.)
\end{proof}

\begin{proof}[Proof of \Proposition{total-rank}]
  First note that the statements \ref{item:rank-Hf} and
  \ref{item:rank-Hl} are equivalent, since we have an
  $\F_2[H,H^{-1},W]$-module $\gr_h$-graded isomorphism
  \[
  \HfOurCx\otimes\F_2[\Z]=(\OurCx/\{H=1\})\otimes\F_2[\Z]
  \to\HlOurCx
  \]
  induced by the map
  \[
  [W^a (u,x),b] \mapsto H^{b-a+(\gr_q((u,x))-l)/2}W^a(u,x).
  \]
  (Here the $\F_2[H,H^{-1},W]$-module structure on
  $\HfOurCx\otimes\F_2[\Z]$ is given by
  \(
  H^a W^b [W^c (u,x),d]=[W^{b+c}(u,x),a+b+d].)
  \)
  The map, by definition, is $\F_2[H,H^{-1},W]$-equivariant; since
  $\gr_h(H)=0$, the map preserves the $\gr_h$-grading; and it has an
  obvious inverse map induced by
  \[
  H^a W^b (u,x)\mapsto [W^b(u,x),a+b+(l-\gr_q((u,x)))/2].
  \]
  To see that these maps are chain maps, observe that if $H^c
  W^d(v,y)$ appears in $\diffOur((u,x))$, then
  $\gr_q((v,y))=\gr_q((u,x))+2c+2d$, and therefore, we have a
  commuting diagram
\[
\xymatrix{
H^a W^b (u,x)\ar@{<->}[r]\ar[d]_-{\diffOur}&
[W^b(u,x),a+b+(l-\gr_q((u,x)))/2] \ar[d]^-{\diffOur/\{H=1\}\otimes\Id}\\
H^{a+c}W^{b+d}(v,y)\ar@{<->}[r]&[W^{b+d}(v,y),a+b+(l-\gr_q((u,x)))/2],
}
\]
where the horizontal arrows are the maps defined above, and the
vertical arrows are parts of the differentials $\diffOur$ and
$\diffOur/\{H=1\}\otimes\Id$.

The statement \ref{item:rank-f} follows from \ref{item:rank-Hf} by the
following well-known trick in homological algebra (and can also be
seen as an application of the universal coefficient theorem). The
complex $\fOurCx$ can be viewed as the mapping cone
  \[
  \fOurCx=\HfOurCx/\{W=1\}\simeq\Cone(W-1\from\HfOurCx\to\HfOurCx).
  \]
  The homology of the mapping cone is the homology of the mapping cone
  of the homology. That is, we have an exact triangle
  \[
  \xymatrix{
    H_*(\HfOurCx)\ar[rr]^-{(W-1)_*}&&H_*(\HfOurCx)\ar[dl]\\
    &H_*(\fOurCx),\ar[ul]
    }
  \]
  which implies
  \begin{align*}
  H_*(\fOurCx)&\cong H_*(\Cone((W-1)_*\from H_*(\HfOurCx)\to
  H_*(\HfOurCx)))\\
  &=H_*(\Cone(W-1\from \bigoplus^{2^l}\F_2[W]\to \bigoplus^{2^l}\F_2[W]))\\
  &=\bigoplus^{2^l}\F_2.
  \end{align*}

  Therefore, we only need to prove the statement for
  \ref{item:rank-Hf}. This follows immediately from
  \Lemma{spectral-sequence-collapse}. The $E^2=E^{\infty}$-page of the
  spectral sequence $H_*(\fBNCx)\otimes\F_2[W]\rightrightarrows
  H_*(\HfOurCx)$ is isomorphic to $2^l$ copies of $\F_2[W]$, with the
  copies in a canonical correspondence with the orientations of
  $L$. The $E^{\infty}$-page being a free module over $\F_2[W]$, we do
  not encounter any extension problems, and can conclude that it is
  isomorphic to the homology of $\HfOurCx$, which therefore is $2^l$
  copies of $\F_2[W]$ as well.
\end{proof}

\section{Concordance invariants}\label{sec:new-s}

In this section, we will construct concordance invariants in the same
spirit as \cite{Ras-kh-slice}, as described in \ref{item:Ras-s}
in \Section{properties}. We concentrate only on knots $K$, although
most of the constructions generalize for links. We also only work with
the filtered version $\fOurCx(K)$; using the more general version
would allow us to construct to similar other invariants, although
computing them might be more challenging.

\begin{definition}
  An \emph{upright set} is a subset $\upright$ of $\Z\times(2\Z+1)$
  satisfying the following condition: If $(a,b)$ is in $\upright$ and
  $a'\geq a$ and $b'\geq b$, then $(a',b')$ is also in $\upright$.
  For any even integer $n$, the \emph{translate} $\upright[n]$ is
  another upright set defined as
  \[
  (a,b)\in\upright[n]\text{ if and only if }(a,b-n)\in\upright.
  \]
  A \emph{centered upright set} is an upright set that contains $(0,1)$, but
  not $(0,-1)$.
\end{definition}

\begin{example}\label{exam:minmax-uprights}
  The intersection of the all the centered upright sets is the following
  centered upright set
  \[
  \upright_{\min}\defeq\set{(a,b)}{a\geq 0\text{ and }b> 0},
  \]
  and the union of the all centered upright sets is the following centered
  upright set
  \[
  \upright_{\max}\defeq\set{(a,b)}{a>0\text{ or }b> 0}.
  \]
\end{example}

\begin{definition}
  A sequence of upright sets $\upright_1,\upright_2,\ldots$ is said to
  have a limit if for all points $(a,b)$, there exists $N$ (depending
  on $a,b$) such that either
  \begin{enumerate}
  \item $(a,b)\in\upright_i$ for all $i>N$; or
  \item $(a,b)\notin\upright_i$ for all $i>N$.
  \end{enumerate} 
  In that case, the
  \emph{limit upright set} is defined as
  \[
  \lim_{i\to\infty}\upright_i=\set{(a,b)}{\text{there exists $N$
      (depending on $a,b$) such that for all $i>N$, $(a,b)\in\upright_i$}}.
  \]
  The following properties are immediate from the definition.
  \begin{enumerate}
  \item If a sequence $\upright_1,\upright_2,\ldots$ has a limit, then
    any subsequence also has the same limit.
  \item The limit of centered upright sets, if exists, is centered.
  \item For nested sequences
    $\upright_1\subseteq\upright_2\subseteq,\cdots$, the limit is the
    union. For nested sequences
    $\upright_1\supseteq\upright_2\supseteq,\cdots$, the limit is the
    intersection.
  \end{enumerate}
\end{definition}

\begin{example}\label{exam:projective-uprights}
  For $t\in[0,1]$, define $\upright_{(t)}$ to be the following
  centered upright set,
  \[
  \upright_{(t)}=
    \set{(a,b)}{at+b(1-t)>0\text{ or }[at+b(1-t)=0\text{ and }b>0]}.
  \]
  For any increasing sequence $t_1\leq t_2\leq\cdots$ of points in $[0,1]$ converging to $t$,
  \[
  \upright_{(t)}=\lim_{i\to\infty}\upright_{(t_i)}.
  \]
\end{example}

\begin{example}\label{exam:complicated-upright}
  Extending \Example{projective-uprights}, for $t\in[0,1]$,
  $s\in[-1,1]$, consider any function
  \[
  r\from\set{(a,b)}{at+b(1-t)=s(1-t)}\to\{\pm 1\}
  \] 
  satisfying:
  \begin{enumerate}
  \item if $a=0$, $r(0,b)=\text{sgn}(b)$ (only relevant when $t=1$ or $s=\pm 1$);
  \item if $b$ is fixed, $r(a,b)$ is a non-decreasing function of $a$
    (only relevant when $t=0$).
  \end{enumerate}
  Define $\upright_{(t,s,r)}$ to be the following centered upright set,
  \[
  \upright_{(t,s,r)}= \set{(a,b)}{at+b(1-t)>s(1-t)\text{ or
    }[at+b(1-t)=s(1-t)\text{ and }r(a,b)> 0]},
  \]
  see also \Figure{complicated-upright}
\end{example}

\tikzstyle{solid}=[circle,thick,fill=black,minimum width=6pt,inner sep=0pt, outer sep=0pt]
\tikzstyle{hollow}=[circle,thick,draw=black,fill=white,minimum width=6pt,inner sep=0pt, outer sep=0pt]
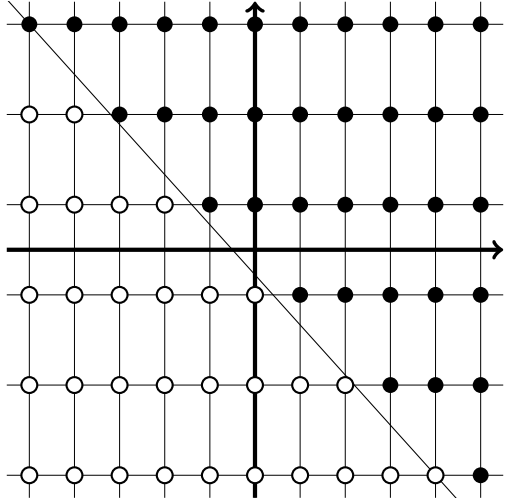
\begin{figure}
  \centering
  \begin{tikzpicture}[scale=0.6]
    \draw[ultra thick,->] (-5.5,0) -- (5.5,0);
    \draw[ultra thick,->] (0,-5.5) -- (0,5.5);

    \clip (-5.5,-5.5) rectangle (5.5,5.5);
    
    \foreach \i in {-5,...,5}
    {
      \draw[ultra thin] (\i,-5.5) -- (\i,5.5);
    }    
    \foreach \i in {-5,-3,...,5}
    {
      \draw[ultra thin] (-5.5,\i) -- (5.5,\i);
    }

    \draw (-9.5,10) -- (8.5,-10);

    \foreach \i/\j in {-5/5,-4/5,-3/3,-3/5,-2/3,-2/5,-1/1,-1/3,-1/5,0/1,0/3,0/5,1/-1,1/1,1/3,1/5,2/-1,2/1,2/3,2/5,3/-3,3/-1,3/1,3/3,3/5,4/-3,4/-1,4/1,4/3,4/5,5/-5,5/-3,5/-1,5/1,5/3,5/5}
    {
      \node[solid] at (\i,\j) {};
    }
    \foreach \i/\j in {-5/-5,-5/-3,-5/-1,-5/1,-5/3,-4/-5,-4/-3,-4/-1,-4/1,-4/3,-3/-5,-3/-3,-3/-1,-3/1,-2/-5,-2/-3,-2/-1,-2/1,-1/-5,-1/-3,-1/-1,0/-5,0/-3,0/-1,1/-5,1/-3,2/-5,2/-3,3/-5,4/-5}
    {
      \node[hollow] at (\i,\j) {};
    }

  \end{tikzpicture}
  \caption{The upright set $\upright_{(t,s,r)}$ from
    \Example{complicated-upright}. Let $\mathscr{L}$ be the line
    passing through the point $(0,s)$ and making an angle
    $\theta=2\tan^{-1}(t)$ with the negative $x$-axis. The subset
    $\upright_{(t,s,r)}\subset\Z\times(2\Z+1)$ consists of all the
    points above and to the right of $\mathscr{L}$ (marked solid),
    none of the points below and to the left of $\mathscr{L}$ (marked
    hollow), and some of the points on $\mathscr{L}$ (marked solid or
    hollow) determined by the function $r$. (If $t=0,1$ or $s=\pm 1$,
    there are some restrictions on the function $r$ to ensure that
    $\upright_{(t,s,r)}$ is indeed a centered upright set.) We have
    depicted the case when $t=10/19, s=-5/9$ and $r$ satisfies
    $r(-5,5)=1,r(4,-5)=-1$.
  }\label{fig:complicated-upright}
\end{figure}

\begin{definition}
  For any knot $K$, and any upright set $\upright$, let
  $\Filt_{\upright}\fOurCx(K)$ denote the subcomplex of
  $\fOurCx(K)=(\Cx(K),\dd+\hh)$ generated by Khovanov generators whose
  $(\gr_h,\gr_q)$-bigradings are in $\upright$.
\end{definition}

\begin{definition}\label{def:new-s-invariants}
  Fix a knot $K$ and an orientation $o$ on $K$. For any centered
  upright set $\upright$, define the following three numbers:
  \begin{align*}
  s^{\upright}_{o}(K)&=\max\set{n\in 2\Z}{\text{$\Filt_{\upright[n]}\fOurCx(K)$
      contains a representative for $g(o)$}}+2\\
  s^{\upright}_{-o}(K)&=\max\set{n\in 2\Z}{\text{$\Filt_{\upright[n]}\fOurCx(K)$
      contains a representative for $g(-o)$}}+2\\
  s^{\upright}_{o,-o}(K)&=\max\set{n\in 2\Z}{\text{$\Filt_{\upright[n]}\fOurCx(K)$
      contains a representative for $g(o)+g(-o)$}},
  \end{align*}
  where $g(\pm o)$ are the two generators for $H_*(\fOurCx(K))$
  corresponding to the orientations $\pm o$,
  from \Proposition{total-rank}~\ref{item:rank-f}.
\end{definition}

\begin{example}
  Let us compute these new $s$-invariants for the three-crossing
  diagram of the positive trefoil from
  \Example{trefoil-complex}. Continuing the same notation from that
  example, and for some orientation $o$ of the trefoil, the three
  non-zero elements of the two-dimensional homology of
  $\fOurCx=\OurCx/\{H=W=1\}$ has the following cycle representatives:
  \begin{align*}
    g(o)&=[x^{000}_1+x^{000}_1x^{000}_2]\\
    g(-o)&=[x^{000}_2+x^{000}_1x^{000}_2]\\
    g(o)+g(-o)&=[x^{000}_1+x^{000}_2].
  \end{align*}
  Since none of $x^{000}_1$, $x^{000}_2$, and $x^{000}_1x^{000}_2$ are
  hit by the differential, any cycle representative for these elements
  must contain these Khovanov generators. The
  $(\gr_h,\gr_q)$-bigradings of $x^{000}_1$, $x^{000}_2$ and
  $x^{000}_1x^{000}_2$ are $(0,3)$, $(0,3)$, and $(0,1)$,
  respectively. Therefore, for any centered upright set $\upright$,
  \begin{align*}
    &\phantom{\Longleftrightarrow}\Filt_{\upright[n]}\fOurCx\text{ contains a cycle representative of $g(o)$ or $g(-o)$}\\
    &\Longleftrightarrow (0,1)\in \upright[n]\\
    &\Longleftrightarrow n\leq 0,\\
\shortintertext{and}
    &\phantom{\Longleftrightarrow}\Filt_{\upright[n]}\fOurCx\text{ contains a cycle representative of $g(o)+g(-o)$}\\
    &\Longleftrightarrow (0,3)\in \upright[n]\\
    &\Longleftrightarrow n\leq 2.
  \end{align*}
  (Here, the last step is justified since the centered upright
  $\upright\subset\Z\times(2\Z+1)$ contains $(0,1)$ but not $(0,-1)$.)
  Therefore,
  $s^{\upright}_{o}=s^{\upright}_{-o}=s^{\upright}_{o,-o}=2$.
\end{example}

It is perhaps not immediate why these numbers are knot invariants. We
will prove this in
\Proposition{new-s-properties}~\ref{item:s-invariant}. Along the way,
we need the following lemma.
\begin{lemma}\label{lem:generators-map}
  Consider any connected oriented cobordism from a knot $K_1$ to a
  knot $K_2$ in $\R^3\times[0,1]$. After viewing the cobordism as a
  sequence of Reidemeister moves, births, deaths, and saddles,
  consider the map $f\from H_*(\fOurCx(K_1))\to H_*(\fOurCx(K_2))$
  induced from the maps defined in
  Propositions~\ref{prop:R-I-inv}--\ref{prop:R-III-inv} and
  Definitions~\ref{def:index-0-2}--\ref{def:index-1}. If $\pm o_1$ are
  the orientations on $K_1$, and $\pm o_2$ are the two corresponding
  orientations on $K_2$ (induced from the connected oriented
  cobordism), the map acts as follows on the generators of
  $H_*(\fOurCx)$:
  \[
  g(o_1)\mapsto g(o_2)\qquad g(-o_1)\mapsto g(-o_2).
  \]
\end{lemma}

\begin{proof}
  Consider the induced map $\wt{f}$ on the partially filtered complex
  $\HfOurCx=\OurCx/\{H=1\}$.  Let $\wt{g}(\pm o_1)$ and $\wt{g}(\pm
  o_2)$ be the generators of $H_*(\HfOurCx(K_1))$ and
  $H_*(\HfOurCx(K_2))$ over $\F_2[W]$ corresponding to the
  orientations $\pm o_1$ and $\pm o_2$, respectively, 
  cf.~\Proposition{total-rank}~\ref{item:rank-Hf}. Assume
  \begin{align*}
    \wt{f}(\wt{g}(o_1))&=\al W^a \wt{g}(o_2)+\be W^b \wt{g}(-o_2)\\
    \wt{f}(\wt{g}(-o_1))&=\ga W^c \wt{g}(o_2)+\delta W^d \wt{g}(-o_2)
  \end{align*}
  for some $\al,\be,\ga,\delta\in\F_2$, and some integers $a,b,c,d\geq
  0$.  Here, and elsewhere, we may be viewing the equations at the
  level of homology, or at the chain level, where $\wt{g}(\pm o_i)$
  should read `a cycle representative for $\wt{g}(\pm o_i)$', and the
  equality sign should read `equal relative boundary'. Since $\wt{f}$
  preserves the $\gr_h$-grading, and all of the four elements
  $\wt{g}(\pm o_i)$ live in homological grading zero
  (\Proposition{total-rank}), we must have $a=b=c=d=0$.

  Therefore, the induced map $f_{\BN}$ on the filtered Bar-Natan
  complex $\fBNCx=\OurCx/\{H=1,W=0\}$ is
  \begin{align*}
    f_{\BN}(g_{\BN}(o_1))&=\al g_{\BN}(o_2)+\be g_{\BN}(-o_2)\\
    f_{\BN}(g_{\BN}(-o_1))&=\ga g_{\BN}(o_2)+\delta g_{\BN}(-o_2)
  \end{align*}
  where $g_{\BN}(\pm o_i)$ denotes the standard generators of
  $H_*(\fBNCx(K_i))$ from \ref{item:BN-total-rank}. However, since
  this induced map $f_{\BN}$ is the standard Bar-Natan cobordism map
  (from \Proposition{induced-cob-map-same}), \ref{item:Ras-map}
  implies that $\al=\delta=1$ and $\be=\ga=0$. Therefore, the map $f$
  on $\fOurCx=\OurCx/\{H=W=1\}$ is
  \[
  g(o_1)\mapsto g(o_2)\qquad g(-o_1)\mapsto g(-o_2).
  \]
  as desired.
\end{proof}

\begin{proposition}\label{prop:new-s-properties}
These new $s$-invariants for any knot $K$ and any centered upright set
$\upright$, satisfy the following.
\begin{enumerate}[label=(S-\arabic*), ref=(S-\arabic*)]
\item\label{item:s-invariant} Each of the three numbers $s^{\upright}_*(K)$ is a knot
  invariant, and each is zero for the unknot.
\item\label{item:s-subset} If $\upright'$ is another centered upright
  set with
  $\upright\subset\upright'$, then for each of the three variants
  \[
  s^{\upright}_*(K)\leq s^{\upright'}_*(K).
  \]
\item\label{item:s-limit} For any sequence of centered upright sets
  $\upright_1,\upright_2,\ldots$ that have a limit, each of the three
  variants satisfy
  \[
  \lim_{i\to\infty}s^{\upright_i}_*(K)=s^{\lim_{i}\upright_i}_*(K).
  \]
\item\label{item:s-equal} $s^{\upright}_{o,-o}+2\geq s^{\upright}_{o}(K)=s^{\upright}_{-o}(K)$.
\item\label{item:s-agrees} For each variant,
  $s^{\upright_{(1)}}_*(K)$ agrees with the Rasmussen $s$-invariant
  (where $\upright_{(1)}$ is described in \Example{projective-uprights}).
\item\label{item:s-finite} Each of the three numbers
  $s^{\upright}_*(K)$ is finite.
\end{enumerate}
\end{proposition}

\begin{proof}
  For the proof of \ref{item:s-invariant}, note that the bigraded
  chain homotopy type of the full theory $\OurCx(K)$ is a knot
  invariant, see \Proposition{main-invariance}; therefore, the
  $(\gr_h,\gr_q)$-bifiltered chain homotopy type of $\fOurCx(K)$ is a
  knot invariant as well. Furthermore, the maps inducing the homotopy
  equivalence preserve the generators corresponding to $\pm o$, see
  \Lemma{generators-map}; therefore, each of the three numbers is a
  knot invariant. The computation for the unknot is immediate from the
  $0$-crossing diagram of the unknot.

  For \ref{item:s-subset}, observe that for all $n$,
  $\Filt_{\upright[n]}\fOurCx(K)$ is a subcomplex of
  $\Filt_{\upright'[n]}\fOurCx(K)$. Consequently, if
  $\Filt_{\upright[n]}\fOurCx(K)$ contains a cycle representative for
  some element in $H_*(\fOurCx(K))$, so does
  $\Filt_{\upright'[n]}\fOurCx(K)$.

  For \ref{item:s-limit}, fix some knot diagram for $K$. The Khovanov
  chain group $\Cx$ for this knot diagram is supported on some finite
  subset of $\Z\times(2\Z+1)$. Choose $N$ large enough so that for all
  $i>N$, $\upright_i$ agrees with $\lim_i\upright_i$ on this finite
  subset. The claim then follows immediately.

  \begin{figure}
    \[
    \xymatrix{
      &&&&&&\\
      *+[F]{K} \ar@{-}`d[dr]`[r]`[u]`[][]&\ar@/^1pc/[rrrr]^{\text{isotopy in
          $S^2$}}&&&&\ar@/^1pc/[llll]^{\text{Reidemeister moves in $\R^2$}}&*+[F]{K} \ar@{-}`d[dl]`[l]`[u]`[][] \\
      &&&&&&
    }
    \]
    \caption{An automorphism of $\OurCx$ that reverses the orientation.}\label{fig:reverse-auto}
  \end{figure}
  
  For the equality in \ref{item:s-equal}, we will produce an
  automorphism of $\OurCx(K)$ (in $\hocat(\F_2[H,W])$) whose induced
  automorphism on $H_*(\fOurCx(K))$ interchanges $g(o)$ and $g(-o)$.
  Fix some knot diagram $D_1$ for $K$, and consider the rightmost
  strand. We may move it over the point at $\infty$ in $S^2$ to obtain
  a diagram $D_2$. This isotopy in $S^2$ induces an identification
  between $\OurCx(D_1)$ and $\OurCx(D_2)$.  We then perform a sequence
  of Reidemeister moves in the plane $\R^2$ to get back to $D_1$ from
  $D_2$, producing a map $\OurCx(D_2)\to\OurCx(D_1)$. The composition
  is the required automorphism.  See \Figure{reverse-auto} (the same
  trick was used in \cite{Kho-kh-patterns} to deal with the
  basepoint).  Since the isotopy moves over $\infty\in S^2$ once, the
  checkerboard coloring in \ref{item:BN-total-rank} is reversed, and
  therefore, the induced automorphism of the filtered Bar-Natan
  complex $\fBNCx=\HfOurCx/\{W=0\}$ interchanges the two
  generators. As in the proof of \Lemma{generators-map}, this shows
  that the automorphism on $\fOurCx$ interchanges the two generators
  as well.

  The inequality in \ref{item:s-equal} follows immediately from the
  observation that if $\Filt_{\upright[n]}\fOurCx(K)$ contains cycle
  representatives for $g(o)$ and $g(-o)$, then it contains a cycle
  representative for $g(o)+g(-o)$ as well.

  The proof of \ref{item:s-agrees} takes up most of the work. Let $s$
  denote the Rasmussen invariant. Assume we are working with the
  variant $s^{\upright_{(1)}}_{o}$ (the argument for the other
  variants are similar). To show $s^{\upright_{(1)}}_{o}\geq s$, we
  need to show $\Filt_{\upright_{(1)}[s-2]}\fOurCx$ contains a cycle
  representative for $g(o)$; and to show $s^{\upright_{(1)}}_{o}\leq
  s$, we need to show $\Filt_{\upright_{(1)}[s]}\fOurCx$ does not
  contain a representative for $g(o)$.

  For convenience, let us fix a few more notations. Let
  $\ff=\dd_1+\hh_1$ and $\fg=\dd+\hh-\ff$. Both are endomorphisms on
  the total chain group $\Cx$; neither drops the quantum grading
  $\gr_q$, $\ff$ increases $\gr_h$ by one, while $\fg$ increases it by
  at least two. Furthermore, we have
  $\ff^2=\Commute{\ff}{\fg}=\fg^2=0$. Also, let $\Cx^{\geq i}$ be the
  subgroup of $\Cx$ that lives in homological grading at least $i$,
  and $\Cx^i$ be the subgroup of $\Cx$ that lives in homological
  grading $i$.

  Recall from \ref{item:BN-total-rank}, the $\gr_h$-graded chain
  complex $\fBNCx=(\Cx,\ff)$ has homology of rank two, generated by
  $g_{\BN}(o)$ and $g_{\BN}(-o)$, both supported in $\gr_h$-grading
  zero. Using \ref{item:Ras-s}, choose a cycle representative $c_0$
  for $g_{\BN}(o)$ living in
  $\Cx^0\cap\Filt_{\upright_{(1)}[s-2]}\Cx$. Therefore, we have
  $\ff(c_0)=0$; let $b_0=\fg(c_0)\in\Cx^{\geq 2}$.  Now assume by
  induction that we have defined, for $i=0,\dots,k-1$, chains
  $b_i\in\Cx^{\geq i+2}$ and $c_i\in\Cx^{\geq i}$, so that
  $\fg(c_i)=b_i$, and (for $i\neq 0$) $\ff(c_i)=b_{i-1}$. We will
  extend the construction to $i=k$. Since
  \[
  \ff(b_{k-1})=\ff\fg(c_{k-1})=\fg\ff(c_{k-1})=
  \begin{cases}0&\text{if $k=1$}\\
    \fg(b_{k-2})=\fg^2(c_{k-2})=0&\text{if $k>1$,}\end{cases}
  \]
  and $b_{k-1}$ lives in homological grading at least $k+1$, and the
  entire homology of $(\Cx,\ff)$ is supported in grading zero, there
  is some chain $c_k\in\Cx^{\geq k}$ with $\ff(c_k)=b_{k-1}$; define
  $b_k=\fg(c_k)$. Then $\sum_i c_i$ is a cycle for the chain complex
  $\fOurCx=(\Cx,\ff+\fg)$, and indeed, represents the generator
  $g(o)$. Moreover, by construction, it is supported in
  $\Filt_{\upright_{(1)}[s-2]}\Cx$, and this establishes
  $s^{\upright_{(1)}}_{o}\geq s$.

  For the other direction, assume if possible, $g(o)$ has a cycle
  representative in $\Filt_{\upright_{(1)}[s]}\Cx$. Let $c'_i$ be
  the part of this cycle representative that lives in homological
  grading $i$. Therefore, $\sum_i(c_i+c'_i)$ is a boundary, say
  $(\ff+\fg)(a)$, for some chain $a$; let $a_i$ be the part of $a$
  that lives in $\gr_h=i$, and let $k=\min\set{i}{a_i\neq 0}$. Since
  $(c_0+c'_0)\neq 0$, we have $k\leq -1$. Indeed, we may assume
  $k=-1$. Otherwise, if $k<-1$, then $\ff(a_k)\in\Cx^{k+1}$ is zero,
  and since the homology of $(\Cx,\ff)$ is supported in grading zero,
  there exists a chain $e\in\Cx^{k-1}$ with $\ff(e)=a_k$. Then
  $(a+(\ff+\fg)(e))\in\Cx^{\geq k+1}$ is another chain whose boundary
  is also $\sum_i(c_i+c'_i)$. Therefore, we may assume, $a\in\Cx^{\geq
    -1}$. Then we must have $\ff(a_{-1})=c_0+c'_0$. This implies
  $c'_0$ is also a cycle representative for the generator $g_{\BN}(o)$
  in $H_*(\fBNCx)$. Since $c'_0$ is supported in quantum grading
  $\gr_q\geq s+1$, this is a contradiction, thereby establishing
  $s^{\upright_{(1)}}_{o}\leq s$.

  The statement of \ref{item:s-finite} is an immediate corollary of
  \ref{item:s-agrees}.  Let us assume that we are working with the
  variant $s^{\upright}_{o}(K)$ (the argument for the other variants
  are similar). Since
  \[
  s^{\upright_{\min}}_o(K)\leq s^{\upright}_o(K)\leq s^{\upright_{\max}}_o(K),
  \]
  from \ref{item:s-subset} (where $\upright_{\min}$ and
  $\upright_{\max}$ are defined in \Example{minmax-uprights}), it is enough
  to show that $s^{\upright_{\min}}_o(K)>-\infty$ and
  $s^{\upright_{\max}}_o(K)<\infty$. Let
  $s=s^{\upright_{(1)}}_{o}(K)$ be Rasmussen's $s$-invariant. Then
  from definitions, $\Filt_{\upright_{(1)}[s-2]}\fOurCx(K)$
  contains a representative for $g(o)$, but
  $\Filt_{\upright_{(1)}[s]}\fOurCx(K)$ does not.

  Now fix some knot diagram for $K$, and let $S$ be the finite subset
  of $\Z\times(2\Z+1)$ that supports the bigrading of the Khovanov chain
  complex $\Cx$ for this knot diagram. There exists $n$ sufficiently
  small so that 
  \[
  \upright_{(1)}[s-2]\cap S\subset\upright_{\min}[n],
  \]
  and therefore, $\Filt_{\upright_{\min}[n]}\fOurCx(K)$ contains a
  cycle representative for $g(o)$ as well. Similarly, there exists $m$
  sufficiently large so that
  \[
  \upright_{\max}[m]\cap S\subset\upright_{(1)}[s],
  \]
  and therefore, $\Filt_{\upright_{\max}[m]}\fOurCx(K)$ cannot contain
  a cycle representative for $g(o)$.
\end{proof}

We conclude by observing that each of these new $s$-invariants produce
a lower bound for the four-ball genus.
\begin{proposition}\label{prop:lower-bound}
  For any connected oriented genus-$g$ knot cobordism between any
  knots $K_1$ and $K_2$ in $\R^3\times[0,1]$, and any centered upright
  set $\upright$,
  \[
  g\geq\frac{1}{2}\card{s^{\upright}_*(K_1)-s^{\upright}_*(K_2)}.
  \]
  Therefore, the four-ball genus of any knot $K$ is bounded below
  \[
  g_4(K)\geq\frac{1}{2}\card{s^{\upright}_*(K)}.
  \]
  (Here $s_*^{\upright}$ denotes any one of the three versions defined
  in \Definition{new-s-invariants}.)
\end{proposition}

\begin{proof}
  Present the cobordism as a sequence of elementary moves, and
  consider the link-cobordism map \(
  f\from\OurCx(K_1)\to\OurCx(K_2)\{-2g\} \) in $\hocat(\F_2[H,W])$ as
  defined in Propositions~\ref{prop:R-I-inv}--\ref{prop:R-III-inv} and
  Definitions~\ref{def:index-0-2}--\ref{def:index-1}.  (This map might
  depend on how the cobordism is presented, but that turns out to be
  irrelevant.) We will analyze the map $\ol{f}$ induced on
  $\fOurCx=\OurCx/\{H=W=1\}$.

  Let $(u,x)$ and $(v,y)$ be Khovanov generators so that $(v,y)$
  appears in $\ol{f}((u,x))$. Then $(v,y)$ must appear in $f((u,x))$
  with some non-zero coefficient, say $H^aW^b$, for some $a,b\geq
  0$. Therefore, 
  \begin{align*}
    \gr_h((v,y))&=b+\gr_h((H^aW^b(v,y))=b+\gr_h((u,x))\geq\gr_h((u,x)),\text{
      and}\\
    \gr_q((v,y))&=2(a+b)+\gr_q((H^aW^b(v,y))=2(a+b)+\gr_q((u,x))-2g\geq\gr_q((u,x))-2g.
  \end{align*}
  In other words, if $(u,x)$ is contained in any upright set $\upright'$,
  then $(v,y)$ is contained in the translate $\upright'[-2g]$.

  From \Lemma{generators-map}, we know that 
  \[
  \ol{f}(g(o_1))=g(o_2)\qquad \ol{f}(g(-o_1))=g(-o_2),
  \] 
  where $\pm o_1$ and $\pm o_2$ are the orientations on $K_1$ and
  $K_2$, induced from the two orientations of the cobordism.

  Summarizing what we have said so far, for every $n\in\Z$, we have a
  commuting diagram
  \[
  \xymatrix@C=20ex{
    \Filt_{\upright[n]}\fOurCx(K_1)\ar@{^(->}[d]\ar[r]^-{\ol{f}|}&\Filt_{\upright[n-2g]}\fOurCx(K_2)\ar@{^(->}[d]\\
    \fOurCx(K_1)\ar^-{\ol{f}}_-{g(o_1)\mapsto g(o_2),g(-o_1)\mapsto
      g(-o_2)}[r]&\fOurCx(K_2).}
  \]
  Therefore, if $\Filt_{\upright[n]}\fOurCx(K_1)$ contains a cycle
  representative for $g(o_1)$ (respectively, $g(o_1)+g(-o_1)$), then
  $\Filt_{\upright[n-2g]}\fOurCx(K_2)$ contains a cycle representative
  for $g(o_2)$ (respectively, $g(o_2)+g(-o_2)$). Therefore, we get the
  inequality $s^{\upright}_*(K_2)\geq s^{\upright}_*(K_1)-2g$.

  Viewing the cobordism in reverse, we get $s^{\upright}_*(K_1)\geq
  s^{\upright}_*(K_2)-2g$, and combining the two inequalities, we
  reach our desired goal.
\end{proof}

One can easily construct model chain complexes where these new
$s$-invariants are different from Rasmussen's $s$-invariant. The
following are perhaps the simplest of such models (the two models are
duals of one another). Assume that $(\Cx,\dd+\hh)$ contains a direct
summand in one of the following two forms, and assume $b$ is a
representative for the Bar-Natan generator $g_{\BN}(o)$. Therefore,
$b$ lies in bigrading $(0,s-1)$, where $s$ is the Rasmussen invariant.
\[
\xymatrix{
&&a   &&&&  &e&\\
b\ar[urr]|-{\dd_2}&c\ar[r]|-{\dd_1}\ar[ur]|-{\hh_1}&d  &&&&  d\ar[r]|-{\dd_1}\ar[ru]|-{\hh_1}&c&b \\
&e\ar[ur]|-{\hh_1}&  &&&&  a\ar[rru]|-{\dd_2}\ar[ru]|-{\hh_1}&&
}
\]
In the first case, we get that $g(o)$ has a unique cycle
representative $b+c+e$, and therefore, for any centered upright set $\upright$,
\[
s^{\upright}_o=\begin{cases}
s&\text{if $(1,-1)\in\upright$,}\\
s-2&\text{otherwise.}
\end{cases}
\]
In the second case, $g(o)$ has three cycle representatives, $b$, $c$,
or $e$; and therefore, for any centered upright set $\upright$,
\[
s^{\upright}_o=\begin{cases}
s+2&\text{if $(-1,3)\in\upright$,}\\
s&\text{otherwise.}
\end{cases}
\]
Note that in both cases, the $E^3$-page of the Bar-Natan spectral
sequence, induced from the filtered Bar-Natan complex
$\fBNCx=(\Cx,\dd_1+\hh_1)$, contains one of the following two
configurations (the higher differential is a zigzag differential from
\Equation{zigzag}):
\[
\vcenter{\hbox{\begin{tikzpicture}
  \draw[ultra thin] (-1.5,-1.5) grid (1.5,1.5);
  
  \node at (-1,0) {$\bullet$};
  \node (e) at (0,-1) {$\bullet$};
  \node (a) at (1,1) {$\bullet$};
    
  \node[anchor=east] at (-1.5,-1) {$s-3$};
  \node[anchor=east] at (-1.5,0) {$s-1$};
  \node[anchor=east] at (-1.5,1) {$s+1$};

  \node[anchor=north] at (-1,-1.5) {$0$};
  \node[anchor=north] at (0,-1.5) {$1$};
  \node[anchor=north] at (1,-1.5) {$2$};

  \draw[->] (e) -- (a);
\end{tikzpicture}}}
\qquad\qquad\text{or}\qquad\qquad
\vcenter{\hbox{\begin{tikzpicture}
  \draw[ultra thin] (-1.5,-1.5) grid (1.5,1.5);
  
  \node (ap) at (-1,-1) {$\bullet$};
  \node (ep) at (0,1) {$\bullet$};
  \node  at (1,0) {$\bullet$};

  \node[anchor=east] at (-1.5,-1) {$s-3$};
  \node[anchor=east] at (-1.5,0) {$s-1$};
  \node[anchor=east] at (-1.5,1) {$s+1$};

  \node[anchor=north] at (-1,-1.5) {$-2$};
  \node[anchor=north] at (0,-1.5) {$-1$};
  \node[anchor=north] at (1,-1.5) {$0$};

  \draw[->] (ap) -- (ep);
\end{tikzpicture}}}
\]
However, the Bar-Natan spectral sequence usually collapses very
quickly. Indeed, for all knots up to $16$ crossings, except the
connect sum of the torus knot $T(3,4)$ with its mirror, the $E^3$-page
of the Bar-Natan spectral sequence does not contain any of the above
two configurations. Therefore, it remains a challenging exercise to
find knots where these new $s$-invariants are different from the
existing one.

\bibliography{double}

\end{document}